\documentclass[a4paper]{article}

\usepackage[hidelinks]{hyperref}
\usepackage{mathtools}
\usepackage{amsfonts}
\usepackage{amssymb}
\numberwithin{equation}{section}

\usepackage[scr=rsfs]{mathalpha}

\newcommand{\stat}{{\mathrm{st}}}
\newcommand{\flow}{\Phi}
\newcommand{\Cc}{{\mathcal C}}
\newcommand{\Cs}{{\mathcal S}}
\newcommand{\Ct}{{\mathcal T}}
\newcommand{\dataspace}{{\mathcal H}}

\newcommand{\Prob}{\mathbb P}
\newcommand{\I}{\mathbf 1}
\newcommand{\N}{\mathbb N}
\newcommand{\R}{\mathbb R}
\newcommand{\Z}{\mathbb Z}
\newcommand{\Tor}{\mathbb T}
\newcommand{\Fou}{\mathcal F}
\newcommand{\inv}{^{-1}}
\newcommand{\abs}[1]{{\left|#1\right|}}
\newcommand{\smallabs}[1]{{|#1|}}
\newcommand{\bignorm}[1]{{\left\|#1\right\|}}
\newcommand{\norm}[1]{{\|#1\|}}
\newcommand{\inorm}[1]{{\left\langle{#1}\right\rangle}}
\newcommand{\dual}[1]{{\left\langle#1\right\rangle}}
\newcommand{\diff}{\operatorname{d}\!}
\newcommand{\dx}{\diff x}
\newcommand{\dy}{\diff y}
\newcommand{\ds}{\diff s}
\newcommand{\dt}{\diff t}
\DeclareMathOperator{\E}{\mathbb E}
\DeclareMathOperator{\supp}{Supp}
\DeclareMathOperator{\Law}{Law}
\newcommand{\wick}[1]{\mathopen:#1\mathclose:}

\newenvironment{aligncases}{\left\{\begin{aligned}}{\end{aligned}\right.}

\usepackage{amsthm}
\newtheorem{theorem}{Theorem}[section]
\newtheorem{lemma}[theorem]{Lemma}
\newtheorem{corollary}[theorem]{Corollary}
\theoremstyle{definition}
\newtheorem{remark}[theorem]{Remark}
\newtheorem{example}[theorem]{Example}
\newtheorem{definition}[theorem]{Definition}

\title{Invariance of $\phi^4$ measure under nonlinear wave and Schrödinger equations on the plane}
\author{Nikolay Barashkov%
\footnote{\url{nikolay.barashkov@mis.mpg.de};
Max Planck Institute for Mathematics in the Sciences, Leipzig, Germany},
Petri Laarne%
\footnote{\url{petri.laarne@helsinki.fi};
Department of Mathematics and Statistics, University of Helsinki, Finland;
ORCiD~0000-0002-9044-8583}}
\date{}

\begin{document}

\maketitle

\begin{abstract}
We show almost sure wellposedness of mild solution to
the cubic nonlinear wave equation in a weighted Besov space over $\R^2$.
To achieve this, we show that any weak limit of $\phi^4$ measures on increasing tori
is invariant under the equation.
We review and slightly simplify the periodic theory and the construction of the weak limit measure,
and then use finite speed of propagation to reduce the infinite-volume case to the previous setup.
Our argument also gives a weaker invariance result
on the nonlinear Schrödinger equation in the same setting.\\

\end{abstract}


%
%
%
\section{Introduction}

Since Jean Bourgain's work in the 1990s,
invariant measures have been an important tool in probabilistic solution theory
of dispersive PDEs.
Bourgain originally studied the nonlinear Schrödinger equation
\begin{equation}
i\partial_t u(x, t) + \Delta u(x,t) = \pm\lambda\, u(x,t) \abs{u(x,t)}^p
\end{equation}
on one-dimensional torus $\Tor$ \cite{bourgain_periodic_1994}.
He proved almost sure wellposedness when the initial data
is sampled from the natural Gibbs measure.
We are interested in $p=2$,
in which case the Gibbs measure is the (complex) $\phi^4$ measure from quantum field theory.
Later on in \cite{bourgain_invariant_1996}, he extended the result to $\Tor^2$.
In two or more dimensions the $\phi^4$ measure is supported on distributions,
and it then becomes necessary to renormalize the nonlinearity
by Wick ordering:
\begin{equation}\label{eq:nls}
\tag{NLS}
i\partial_t u(x, t) + (m^2 + \Delta) u(x,t) = \lambda\, \wick{u(x,t) \abs{u(x,t)}^2}
\end{equation}

Our main subject is the defocusing massive nonlinear wave equation
\begin{equation}\label{eq:nlw}
\tag{NLW}
\partial_{tt} u(x, t) + (m^2 - \Delta) u(x,t) = -\lambda\, \wick{u(x,t)^3},
\quad m^2 > 0,
\end{equation}
on spatial domain $\R^2$.
This equation with Gibbsian initial data (and a more general nonlinearity)
was previously solved on $\Tor^2$ by Oh and Thomann \cite{oh_invariant_2020}.
The main result of this article can be stated as follows:

\begin{theorem}[Global existence and uniqueness]\label{thm:global moment bounds}
Let $\vec{\mu}$ be the product of infinite-volume $\phi^4$ and white noise measures
and fix $\varepsilon > 0$.
Let $H^{-\varepsilon}(\rho)$ be the Besov space
with a sufficiently integrable polynomial weight $\rho$.
For $\vec{\mu}$-almost all initial data,
the nonlinear wave equation~\eqref{eq:nlw}
has a unique mild solution in $C(\R_+;\, H^{-\varepsilon}(\rho))$.
\end{theorem}
The precise definition of $H^{-\varepsilon}(\rho)$ is given in Section~\ref{sec:besov},
and of mild solution in Definition~\ref{def:Solution Whole Space}.

Our approach is to construct solutions on periodic domains $\Lambda_L \coloneqq [{-L},{L}]^2$
and then approximate infinite-volume solutions with them.
The high-level proof strategy on periodic domain goes back to Bourgain:
\begin{enumerate}
\item Define a probability distribution on the initial data.
\item Prove deterministic wellposedness for time interval $[0,\tau]$
when the initial data belongs to some set $A$ of large probability.
The small time $\tau$ depends on the size of $A$.
\item Prove that the probability measure is invariant in time under the equation.
\item Intersect the sets of initial and final values, which have same probability by invariance.
By iteration, the probability of blow-up by time $T = n\tau$ is bounded by $n(1-\Prob(A))$.
\item Use stochastic estimates to show that an increase of $\Prob(A)$ cancels
the corresponding increase of iterations $n$;
thus the probability of blow-up can be made arbitrarily small.
\end{enumerate}
This argument reduces the global-in-time solution theory into understanding the
invariance and large deviations of the Gibbs measure.
To show invariance, we use finite-dimensional approximation.
\mbox{Liouville}'s theorem states that the Gibbs measure associated to the Hamiltonian
of Fourier-truncated \eqref{eq:nlw} is invariant.
These approximate measures converge in total variation to the untruncated, periodic-domain measures.

The extension to infinite volume relies on two further insights:
\begin{enumerate}
\item By \cite{mourrat_global_2017}, there are uniform bounds for the $2L$-periodic $\phi^4$ measures
in the polynomially weighted space $H^{-\varepsilon}(\rho)$.
This yields a convergent subsequence of measures as $L \to \infty$.
\item Thanks to the finite speed of propagation of \eqref{eq:nlw},
all statements about measurable events on $H^{-\varepsilon}(\rho)$ can be reduced to
bounded regions of $\R^2$.
This lets us go back to the periodic solution theory.
\end{enumerate}

For the nonlinear Schrödinger equation the situation is more complicated,
as there is no finite speed of propagation.
This means that we cannot reduce the problem to the periodic setup.
We can still prove a weaker form of invariance
in a larger Besov space by giving up some spatial differentiability.
This sense of invariance was initially developed
for Euler and Navier--Stokes equations by Albeverio and Cruzeiro \cite{albeverio_global_1990},
and was explored in the case of periodic 2D NLS in \cite{oh_pedestrian_2018}.
However, we are not able to comment on the uniqueness of solutions,
as can be done in one dimension \cite{bourgain_invariant_2000}.

\begin{theorem}[Weak invariance of NLS]\label{thm:nls main result}
Let $\mu$ be the complex $\phi^4$ measure on $\R^2$
and $\rho$ as above.
There exists $s > 4$ such that for $\mu$-almost all initial data,
the nonlinear Schrödinger equation~\eqref{eq:nls} with~$p=2$
has a mild solution $u \in C(\R_+;\, H^{-s}(\rho))$ in the sense of equation~\eqref{eq:nls mild}.
Moreover, for any $t \in \R_+$ we have $\Law(u(t)) = \mu$.
\end{theorem}

\subsection{The \texorpdfstring{$\phi^4$}{phi\textasciicircum4} measure}\label{sec:intro phi4}
As mentioned above, Fourier-truncated versions of these equations conserve the Hamiltonian $H$,
with which we can define the Gibbs measure proportional to $\exp(-\beta H)$.
The parameter $\beta > 0$ is called the inverse temperature.
For $N$-truncated and $2L$-periodic \eqref{eq:nlw}, the Gibbs measure is proportional to%
\footnote{In the following, we set $\beta = \lambda = 1$
as they are not too relevant for our present topic.}
\begin{equation}\label{eq:phi4 measure}
\exp\left(-\beta \int_{[{-L},{L}]^2}
    \frac{\lambda\, \wick{u^4}}{4} + \frac{m^2 \abs{u}^2 + \abs{\nabla u}^2 + \abs{\partial_t u}^2}{2}
\dx \right)
\prod_{\abs k \leq N} \mathrm d \hat u(k).
\end{equation}
The expression without restriction to $\abs k \leq N$ is only formal
since an infinite Lebesgue product measure does not exist.
However, the second exponential term yields a Gaussian factor that makes the $N \to \infty$ limit still well-defined.

The continuum versions of these Gibbs measures are
studied in constructive quantum field theory \cite{glimm_quantum_1987}.
Stochastic quantization (see e.g.\ \cite{parisi_perturbation_1981})
is a rigorous PDE approach for their study.
In this approach the $\phi^4_d$ measure is regarded as
an invariant measure for a nonlinear heat equation with white noise forcing
(see Theorem~\ref{thm:SQ} below).
These equations are singular and cannot be solved classically.  

The periodic $\phi^4_2$ equation was solved by Da~Prato and \mbox{Debussche} \cite{da_prato_strong_2003}.
The limit measure is absolutely continuous with respect to a Gaussian measure.
Existence of infinite-volume solutions for the 2D equation was later shown by Mourrat and Weber
in a polynomially weighted space \cite{mourrat_global_2017};
see also \cite{mourrat_dynamic_2017}.
We will rely heavily on these ideas in Section~\ref{sec:stochastic}.

The local wellposedness theory for the more singular $\Tor^3$ case
came in three approaches in mid-2010s:
Hairer's regularity structures \cite{hairer_theory_2014};
\mbox{Gubinelli}, \mbox{Imkeller} and \mbox{Perkowski's} \cite{gubinelli_paracontrolled_2015} paracontrolled distributions;
and \mbox{Kupiainen's} renormalization group approach \cite{kupiainen_renormalization_2016}.
The bounds of Mourrat and Weber were then exploited by Albeverio and Kusuoka \cite{albeverio_invariant_2020}
and Gubinelli and Hofmanová \cite{gubinelli_pde_2021} to give a self-contained construction of the $\phi^4_3$ measure.

In dimensions $d \geq 4$, the $\phi^4_d$ measures collapse to trivial Gaussian measures.
The last outstanding case $d=4$ was proved recently by Aizenman and Duminil-Copin;
see their article \cite{aizenman_marginal_2021} for discussion.

\bigskip
The $\phi^4$ measure is expected to be invariant under three PDEs
that share essentially the same Hamiltonian:
\eqref{eq:nls}, \eqref{eq:nlw}, and the cubic stochastic nonlinear heat equation.
As shown in \cite[Figure~1]{bringmann_invariant_2022},
the periodic-domain invariance theory is almost done,
with only the three-dimensional \eqref{eq:nls} missing.

This theory, and hence the global wellposedness of the equations,
is much less developed in the infinite volume.
For wave and Schrödinger equations the previous results are limited to
one dimension \cite{bourgain_invariant_2000} or radial setting \cite{xu_invariant_2014}.

The largest complication is that the infinite-volume $\phi^4$ measures are only defined as weak limits
of approximating sequences,
and in particular they are no longer absolutely continuous with respect to a Gaussian measure.
This means that total variation convergence is no longer available and we have to prove
local wellposedness for non-Gaussian initial data.
Depending on the coupling constant $\lambda$,
the sequence might have more than one accumulation point.

However, the invariant distribution 
can still be coupled to a Gaussian, and the perturbation term is of better Besov regularity.
This idea underlies the variational approach in \cite{barashkov_variational_2020}.
A similar fact was exploited by Bringmann and collaborators in
\cite{Bringmann_invariant_2020, Bringmann_invariant_2022-1, bringmann_invariant_2022}
in situations where the singularity of the measure arises in finite volume due to short scale divergences.

\begin{remark}
As this manuscript was being prepared, Oh, Tolomeo, Wang, and Zheng published their work
\cite{oh_hyperbolic_2022} where similar ideas appear.
They prove Theorem~\ref{thm:global moment bounds} for
a more challenging equation, \eqref{eq:nlw} with additive stochastic forcing.
This equation is also known as the \emph{canonical stochastic quantization equation};
we further discuss this hyperbolic approach to SQ in Remark~\ref{rem:sq hyperbolic}.

The approach in \cite{oh_hyperbolic_2022} is based on an optimal transport argument developed in \cite{oh_stochastic_2021},
and involves convergence of measures in a Wasserstein metric.
Our globalization argument depends more heavily on finite speed of propagation
and only uses weak convergence.
Although weaker, some of our arguments are simpler due to the use of parabolic stochastic quantization.
Moreover, our approach easily yields the weak invariance result for \eqref{eq:nls}.
\end{remark}

\subsection{Previous literature and extensions}
Let us take a moment here to review some of the history of this question.
As mentioned above, the general globalization-in-time argument
was developed by Bourgain \cite{bourgain_periodic_1994}
in context of the one-dimensional periodic \eqref{eq:nls}.
This was in response to earlier work of Lebowitz, Rose, and Speer \cite{lebowitz_statistical_1988}
in late 1980s.

Invariant measures for the one-dimensional wave equation were considered
by Zhidkov \cite{zhidkov_invariant_1994} and McKean and Vaninsky \cite{mckean_statistical_1994}.
Radially symmetric \eqref{eq:nlw} on a three-dimensional ball was considered
by Burq and Tzvetkov \cite{burq_invariant_2007} and Bourgain and Bulut \cite{bourgain_invariant_2014},
and extended by Xu to infinite volume \cite{xu_invariant_2014}.
Recently progress has been made in three dimensions,
culminating in the proof of invariance of periodic $\phi^4_3$ under the wave equation
\cite{Bringmann_invariant_2020, Bringmann_invariant_2022-1, bringmann_invariant_2022}.

NLW has also been considered with random data not sampled from the invariant measure
\cite{kenig_focusing_2021}.
Related to the invariance of Gibbs measures is the program for showing
quasi-invariance of Gaussian measures under Hamiltonian PDEs \cite{tzvetkov_quasiinvariant_2015};
in this notion the law of solutions at any given time
remains absolutely continuous with respect to the initial measure.
For the wave equation this was carried out in \cite{gunaratnam_quasi-invariant_2022, oh_quasi-invariant_2020}.

Another related development is the solution theory for \eqref{eq:nlw} with additive white noise forcing,
either with or without an additional damping term $\partial_t u$.
Local wellposedness on $\Tor^2$ was achieved in \cite{gubinelli_renormalization_2018}
and extended to global wellposedness in \cite{gubinelli_global_2022, tolomeo_global_2021}.
If the damped equation also includes dispersion,
the invariant measure is moreover ergodic \cite{tolomeo_unique_2020}.
Oh, Tolomeo, Wang, and Zheng \cite{oh_hyperbolic_2022}
consider the damped case on $\R^2$.

The nonlinearity can be replaced by a general polynomial, exponential or trigonometric term;
see \cite{oh_invariant_2021, oh_parabolic_2021, oh_hyperbolic_2022} and references therein.
These correspond to very different physical models
and feature interesting renormalization behaviour.
It is also possible to let the solution take values in a manifold instead of $\R$;
there is recent progress on invariant measures of these wave maps equations
\cite{bringmann_wave_2021,brzezniak_statistical_2022}.

For \eqref{eq:nls} in one dimension it is possible to consider both focusing and defocusing nonlinearities,
due to the presence of an $L^2$ conservation law.
Restricting to a ball in $L^2$ leads to a normalizable measure if the nonlinearity is subquintic.
In the quintic case the measure is normalizable if and only if the coupling is suffiently weak;
remarkably, this threshold is known exactly \cite{oh_optimal_2022}.

In two dimensions the defocusing case can still be investigated,
as was done by Bourgain \cite{bourgain_invariant_1996} for the cubic case and
later for general polynomial nonlinearities by Deng, Nahmod, and Yue \cite{deng_invariant_2019}.
For the focusing NLS the $L^2$ cutoff does not lead to
a normalizable measure anymore \cite{brydges_statistical_1996}.
Quasi-invariance has also been investigated for the NLS
\cite{oh_quasi-invariant_2021, oh_optimal_2018, oh_quasi-invariant_2017}.

In \cite{oh_remark_2020, seong_invariant_2022} invariant measures of the Zakharov--Yukawa system were studied.
This is a system of coupled wave and Schrödinger equations with nonpositive Hamiltonian
and an $L^2$ conservation law.
Due to these properties it behaves similarly to the defocusing NLS.

The activity described above has mostly taken place on the torus.
In infinite volume we mention the early result of Bourgain
on one-dimensional NLS \cite{bourgain_invariant_2000},
as well as the work of Cacciafesta and Suzzoni on the NLS
and other Hamiltonian equations \cite{cacciafesta_invariance_2019}.
These are in addition to the aforementioned papers \cite{oh_hyperbolic_2022, xu_invariant_2014}
on two- and three-dimensional NLW.

\bigskip
Let us conclude this review with a comment on possible extensions of our work and open problems.
Our method extends in a straightforward way to more general polynomial nonlinearities and to vector-valued models.

\begin{example}
The mass term $m^2 > 0$ in~\eqref{eq:nlw}
is used to avoid problems with the zero Fourier mode.
There are however setups (e.g.\ \cite{barashkov_eyringkramers_2024})
where the equation is formulated with a negative mass term:
\[
\partial_{tt} u(x, t) - (m^2 + \Delta) u(x,t) = -\wick{u(x,t)^3}.
\]
Mourrat and Weber~\cite{mourrat_global_2017} consider also this case.
If we add $2m^2\, u(x,t)$ to both sides of the equation,
the modified nonlinearity $-\wick{u^3} + 2m^2 u$ will still be dominated by the cubic term.
In the present work we assume a positive mass to simplify the exposition.
\end{example}

For the weak invariance we also expect the extension to long-range models (with fractional Laplacian) to be straightforward,
provided the resulting measures are not too singular.
The strong invariance of~\eqref{eq:nls} on~$\R$
under general polynomial nonlinearities (so-called $P(\phi)_1$ theories) is interesting.
The $\phi^4$ case was solved by Bourgain \cite{bourgain_invariant_2000},
and Bringmann and Staffilani \cite{bringmann_invariant_2025}
recently extended the proof to $u \abs{u}^p$ up to $p \leq 4$.
The corresponding 2D problem in the full space is very interesting,
as well as the case of non-polynomial nonlinearities.

Given the recent work \cite{bringmann_invariant_2022}
on invariance of three-dimensional periodic \eqref{eq:nlw},
it is intriguing to ask about the extension to $\R^3$.
While the measure-theoretic part of our argument is dimension-independent,
the analytic estimates would require significant changes to account for the more singular behaviour.

\subsection{Outline and notation}

Sections~\ref{sec:besov} and~\ref{sec:stochastic} are mostly toolbox sections.
In the former we define Besov spaces and their basic properties,
and in the latter we outline the construction of the $\phi^4$ measure over
polynomially weighted $\R^2$.

We review the solution of~\eqref{eq:nlw} on a periodic domain in Section~\ref{sec:periodic}.
We present a simplified version of the argument of Oh and Thomann~\cite{oh_invariant_2020},
and also provide full details on the Bourgain globalization argument.

The main result in this article is presented in Section~\ref{sec:globalization}.
We use a measure-theoretic argument to reduce the full flow to the periodic case,
and thus prove invariance of the infinite-volume $\phi^4$ measure.

In Section~\ref{sec:schrodinger}, we finally consider \eqref{eq:nls} on $\R^2$.
We prove invariance in Albeverio--Cruzeiro sense
with some weaker estimates on the solutions.

We use the following notation throughout the article:

\begin{itemize}
\item $A \lesssim B$ if $A \leq cB$ for some independent $c > 0$,
    and $A \simeq B$ if $A \lesssim B \lesssim A$.
    Positive constants $c, C$ may vary from line to line.
\item $\inorm x \coloneqq (1 + \abs x^2)^{1/2}$.
\item $P_N$ is a sharp Fourier cutoff to $B(0, 2^N)$.
\item $B^s_{p,r}(\rho)$ are weighted Besov spaces defined in Section~\ref{sec:besov}.
    We abbreviate $H^s(\rho) \coloneqq B^s_{2,2}(\rho)$
    and $\mathcal C^s(\rho) \coloneqq B^s_{\infty, \infty}(\rho)$.
\item $\rho(x) = \inorm{x}^{-\alpha}$ is a polynomial weight;
    $\alpha > 0$ may change between sections.
\item $\Lambda_L \coloneqq {[{-L},{L}]}^2$ is the periodic domain
    and $B^s_{p,r}(\Lambda_L)$ Besov space over it.
\item $\mu$ is the $\phi^4_2$ measure, and
    $\vec{\mu}$ the product of $\phi^4_2$ and white noise measures.
\item $\dataspace^{-\varepsilon}(\rho) \coloneqq H^{-\varepsilon}(\rho) \times H^{-1-\varepsilon}(\rho)$,
    where $\varepsilon > 0$ may change between sections.
\item $\mu_L$ and $\mu_{L,N}$ are bounded-domain and bounded-domain Fourier-truncated
    versions of $\mu$.
\item $\flow_t$ is the flow of \eqref{eq:nlw},
    and $\flow_{L,t}$ and $\flow_{L,N,t}$ are the flows of the periodic
    and the periodic truncated equations.
\item $\Cc_t$ and $\Cs_t$ are the linear propagators of \eqref{eq:nlw},
    defined in Section~\ref{sec:periodic}.
\item $\Ct_t$ is the linear propagator of \eqref{eq:nls},
    defined in Section~\ref{sec:schrodinger}.
\end{itemize}

\subsection{Acknowledgements}

NB was supported by the ERC Advanced Grant 741487 ``Quantum Fields and Probability''.
PL was supported by the Academy of Finland project 339982
``Quantum Fields and Probability''.
PL would like to thank Kalle Koskinen, Jaakko Sinko, and Aleksis Vuoksenmaa for useful conversations.
Both authors would like to thank Leonardo Tolomeo for helpful comments on the preprint,
and the anonymous referees for their thorough reading of the manuscript.

\section{Besov spaces}\label{sec:besov}

Besov spaces are a generalization of Sobolev spaces
that support some useful multiplication estimates and embeddings.
An excellent introduction to the topic is in the
article of Mourrat and Weber \cite{mourrat_global_2017}.
Some results are also collected in the appendix of \cite{gubinelli_pde_2021}.
The textbook of Bahouri, Chemin, and Danchin \cite{bahouri_fourier_2011}
treats the unweighted case.
Due to differences in setup and conventions,
the proofs of the following results are straightforward modifications of those in the listed references.

We will use throughout the article a nonhomogeneous polynomial weight
\begin{equation}\label{eq:besov poly weight}
\rho(x) \coloneqq \inorm x^{-\alpha} \coloneqq (1 + \abs{x}^2)^{-\alpha/2}
\end{equation}
for $\alpha \geq 0$ sufficiently large.
What ``sufficiently large'' means may vary from section to section,
but the final choice is finite.
In some sections we also use the unweighted space ($\alpha=0$);
this is indicated by omitting $\rho$.

\begin{remark}
There are two conventions of weighted $L^p$ spaces in common use.
\cite{mourrat_global_2017} and \cite{gubinelli_pde_2021} respectively define
\[
\norm{f}_{L^p_\rho}^p \coloneqq \int_{\R^d} f(x)^p \rho(x) \dx
\quad\text{and}\quad
\norm{f}_{L^p(\rho)}^p \coloneqq \int_{\R^d} f(x)^p \rho(x)^p \dx.
\]
We use the latter convention since it lets us apply a weight also when $p=\infty$.
For $p < \infty$ the conventions are interchangeable,
and the statements and their proofs require only minor changes.
\end{remark}

\begin{definition}[Littlewood--Paley blocks]\label{def:littlewood-paley}
We fix $\Delta_k$ to be Fourier multipliers
whose symbols form a partition of unity.
More precisely, for $k \geq 0$ they are smoothed indicators of the annuli
$B(0, 2^k\, 8/3) \setminus B(0, 2^k\, 3/4)$,
and for $k=-1$ of the ball $B(0, 3/4)$.
The precise choice of radii is irrelevant.
\end{definition}

\begin{definition}[Weighted Besov space]\label{def:besov}
We define the space $B^s_{p,r}(\rho)$ as the completion of $C_c^\infty(\R^d)$
with respect to the norm
\[
\norm{f}_{B^s_{p,r}(\rho)}
\coloneqq \bignorm{ 2^{ks}\, \norm{\rho(x) [\Delta_k f](x)}_{L^p} }_{\ell^r}
\]
where the $L^p$ norm is taken over $x \in \R^d$
and the $\ell^r$ norm over $k \geq -1$.
We abbreviate
$H^s(\rho) \coloneqq B^{s}_{2,2}(\rho)$
and
$\mathcal C^s(\rho) \coloneqq B^{s}_{\infty,\infty}(\rho)$.
\end{definition}

The following product inequality shows that
products of distributions and smooth enough functions are well-defined distributions.
A recurring `trick' in the following sections is to decompose stochastic objects
into distributional and more regular parts.
There are also analogues of the usual $L^p$ duality and interpolation.

\begin{theorem}[Product inequality]\label{thm:besov multiplication}
Let $s_1 \leq s_2$ be non-zero such that $s_1 + s_2 > 0$,
and let $1 \leq p, p_1, p_2, r \leq \infty$ satisfy $1/p = 1/p_1 + 1/p_2$.
Then
\[
\norm{fg}_{B^{s_1}_{p,r}(\rho_1 \rho_2)}
\lesssim \norm{f}_{B^{s_1}_{p_1,r}(\rho_1)} \norm{g}_{B^{s_2}_{p_2,r}(\rho_2)}.
\]
\end{theorem}
\begin{proof}
\cite[Corollaries~1 and~2]{mourrat_global_2017} and the following remarks therein,
adapted to our convention of polynomial weights.
\end{proof}

\begin{theorem}[Duality]\label{thm:besov duality}
Let $1 \leq p, p' \leq \infty$ and $1 \leq r, r' \leq \infty$ be Hölder conjugate pairs,
$0 < s < 1$, and $\rho_1$ and $\rho_2$ polynomial weights.
Then
\[
\norm{fg}_{L^1(\rho_1 \rho_2)}
\lesssim \norm{f}_{B^s_{p,r}(\rho_1)} \norm{g}_{B^{-s}_{p',r'}(\rho_2)}
\]
\end{theorem}
\begin{proof}
Adaptation of \cite[Proposition~7]{mourrat_global_2017}.
\end{proof}

\begin{theorem}[Interpolation]\label{thm:besov interpolation}
Fix $\theta \in (0, 1)$, $s = \theta s_1 + (1-\theta) s_2$, and
\[
\frac 1 p = \frac{\theta}{p_1} + \frac{1 - \theta}{p_2}, \quad
\frac 1 r = \frac{\theta}{r_1} + \frac{1 - \theta}{r_2}, \quad
\alpha = \theta \beta + (1-\theta) \gamma
\]
for some $1 \leq p, p_1, p_2, r, r_1, r_2 \leq \infty$ and $s_1, s_2, \beta, \gamma \in \R$.
Then
\[
\norm{f}_{B^s_{p,r}(\rho^\alpha)}
\leq \norm{f}_{B^{s_1}_{p_1,r_1}(\rho^\beta)}^\theta
    \norm{f}_{B^{s_2}_{p_2,r_2}(\rho^\gamma)}^{1-\theta}.
\]
\end{theorem}
\begin{proof}
\cite[Lemma~A.3]{gubinelli_pde_2021}.
\end{proof}

We shall use the following three embedding results.
The first lets us trade smoothness for $L^p$ and $\ell^r$ regularity,
whereas the second simplifies some arguments below.
The third one plays a crucial role in the weak convergence argument
by letting us pass to a convergent subsequence in a compact space.

\begin{theorem}[Besov embeddings]\label{thm:besov embedding}
Let $s \in \R$, $1 \leq q \leq p \leq \infty$, and
\[
s' \geq s + d\left( \frac 1 q - \frac 1 p \right).
\]
Then
\[
\norm{f}_{B^s_{p,r}(\rho)} \lesssim \norm{f}_{B^{s'}_{q,r}(\rho)}.
\]
The parameter $1 \leq r \leq \infty$ also satisfies
\[
\norm{f}_{B^{s}_{p,\infty}(\rho)}
\lesssim \norm{f}_{B^{s}_{p,r}(\rho)}
\lesssim \norm{f}_{B^{s+\varepsilon}_{p,\infty}(\rho)}.
\]
\end{theorem}
\begin{proof}
The first claim is an adaptation of
\cite[Proposition~2]{mourrat_global_2017} to our convention of polynomial weights,
and the second follows from Hölder's inequality.
\end{proof}

\begin{theorem}[Relation to Sobolev spaces]\label{thm:besov sobolev}
Let us define the \emph{fractional Sobolev space} $W^{s,p}$, $s \in \R$, $1 \leq p \leq \infty$,
through the norm
\[
\norm{f}_{W^{s,p}(\rho)} \coloneqq \norm{\rho \inorm\nabla^s f}_{L^p},
\]
where $\inorm\nabla^s$ is the Fourier multiplier with symbol $\xi \mapsto \inorm\xi^s$.
Then we have
\[
\norm{f}_{B^s_{p,\infty}(\rho)}
\lesssim \norm{f}_{W^{s,p}(\rho)}
\lesssim \norm{f}_{B^s_{p,1}(\rho)}.
\]
\end{theorem}
\begin{proof}
To show the left inequality, we write
\begin{equation}
2^{ks} \norm{\Delta_k f}_{L^p(\rho)}
= 2^{ks} \norm{\Delta_k \inorm\nabla^{-s} \inorm\nabla^s f}_{L^p}.
\end{equation}
By weighted Young's inequality \cite[Theorem~2.1]{mourrat_global_2017},
this can be bounded by
\begin{equation}
\norm{K_k}_{L^1(\rho\inv)} \norm{\inorm\nabla^s f}_{L^p(\rho)},
\end{equation}
where $K_k$ is the convolution kernel of $\Delta_k \inorm\nabla^{-s}$.
We only need to show that its norm is of order $2^{-ks}$,
as taking the $\ell^\infty$ norm over $k$ then gives the result.

Let us assume that $\alpha \in \N$.
We note that $\inorm{x}^\alpha \lesssim 1 + \abs x^\alpha$,
and that multiplication by $x$ corresponds to differentiation in Fourier space.
Hence
\begin{equation}
\begin{split}
&\mathrel{\phantom{=}} \int_{\R^d} \rho(x)\inv \abs{
    \int_{\R^d} e^{ix \cdot \xi} \inorm\xi^{-s} \hat\Delta_k(\xi) \diff\xi} \dx\\
&= \int_{\R^d} \rho(x) \abs{ \rho(x)^{-2}
    \int_{\R^d} e^{ix \cdot \xi} \inorm\xi^{-s} \hat\Delta_k(\xi) \diff\xi} \dx\\
&\lesssim \int_{\R^d} \rho(x) 
    \int_{\R^d} \abs{(1 + \partial_{\xi_1}^{2\alpha} + \dots + \partial_{\xi^d}^{2\alpha})
        [\inorm\xi^{-s} \hat\Delta_k(\xi)]} \diff\xi \dx.
\end{split}
\end{equation}
The inner integral is then of order $2^{-ks}$ by the support of $\Delta_k$
and the smoothness of $\inorm\xi^{-s} \hat\Delta_k(\xi)$,
and the outer integral is finite if $\rho$ is integrable.

For the right-hand inequality, we write
\begin{equation}
\norm{\inorm\nabla^s f}_{L^p(\rho)}
\leq \sum_{k \geq -1} \norm{\Delta_k \inorm\nabla^s f}_{L^p(\rho)}
= \sum_{k \geq -1} \norm{\Delta_k' \inorm\nabla^s \Delta_k f}_{L^p(\rho)},
\end{equation}
and repeat the above estimate on $\Delta_k' \inorm\nabla^s$,
where $\Delta_k'$ is a slightly larger dyadic multiplier that takes the value $1$
on the support of $\hat\Delta_k$.
\end{proof}

\begin{theorem}[Compact embedding]\label{thm:besov compactness}
Let $\rho_2$ and $\rho_1$ be polynomial weights with respective parameters $\alpha_2 > \alpha_1 > d$;
$p < \infty$, $1 \leq r \leq \infty$, and $s_2 < s_1$.
The space $B^{s_1}_{p,r}(\rho_1)$ then embeds compactly into
the less regular space $B^{s_2}_{p,r}(\rho_2)$.
\end{theorem}
\begin{proof}
\cite[Proposition~11]{mourrat_global_2017}.
\end{proof}

For the finite-volume results, we also need periodic Besov spaces.
The theorems listed above work also in this case,
and in particular Theorem~\ref{thm:besov sobolev} holds with $\varepsilon = 0$.
Furthermore the following lemma shows that we can move between
periodic and polynomial-weight spaces easily.
We use the Mourrat--Weber \cite[Section~4.2]{mourrat_global_2017} definition of these spaces.

\begin{definition}[Periodic Besov space]
Given the set $\Lambda_L \coloneqq {[{-L},{L}]}^d$,
we define the space $B^s_{p,r}(\Lambda_L)$ as the completion
of $2L$-periodic $C^\infty(\R^d)$ functions with respect to the Besov norm
\[
\norm{f}_{B^s_{p,r}(\Lambda_L)}
\coloneqq \bignorm{ 2^{ks}\, \norm{\I_{\Lambda_L}(x) [\Delta_k f](x)}_{L^p_x} }_{\ell^r_k}.
\]
\end{definition}

\begin{lemma}[Embedding into polynomial-weight space]\label{thm:besov periodic into polynomial}
Let $\rho$ be a polynomial weight with parameter $\alpha > d$.
Let $f \in C^\infty(\R^d)$ be $2L$-periodic for $L \geq 1$.
Then
\[
\norm{f}_{B^s_{p,r}(\rho)}
\lesssim \norm{f}_{B^s_{p,r}(\Lambda_L)}
\lesssim L^{\alpha} \norm{f}_{B^s_{p,r}(\rho)}.
\]
These bounds are uniform in $L \geq 1$, $s \in \R$, and $1 \leq p, r \leq \infty$.
\end{lemma}
\begin{proof}
Let us begin with the right-hand-side inequality,
and first consider the $L^p$ norm of a single Littlewood--Paley block:
\begin{equation}
\begin{split}
\norm{\I_{\Lambda_L}(x) [\Delta_k f](x)}_{L^p}
&\leq \left( \sup_{x \in \Lambda_L} (1 + \abs x^2)^{\alpha/2} \right)
    \norm{\rho(x) \I_{\Lambda_L}(x) [\Delta_k f](x)}_{L^p}\\
&\leq (2L^2)^{\alpha/2}
    \norm{\rho(x) [\Delta_k f](x)}_{L^p}.
\end{split}
\end{equation}
This estimate does not depend on $k$ or $p$.
As we multiply by $2^{ks}$ and take the $\ell^r$ norm over $k$,
the prefactor can be moved out.

To get the left-hand side inequality, we apply the triangle inequality.
Let us denote by $\Lambda_L^j$ the translates $\Lambda_L + j2L$.
Then
\begin{equation}
\begin{split}
\norm{\rho(x) [\Delta_k f](x)}_{L^p}
&\leq \sum_{j \in \Z^d} \norm{\rho(x) \I_{\Lambda_L^j}(x) [\Delta_k f](x)}_{L^p}\\
&\leq \norm{\I_{\Lambda_L}(x) [\Delta_k f](x)}_{L^p}
    \sum_{j \in \Z^d} \sup_{x \in \Lambda_L^j} \rho(x)\\
&\leq \norm{\I_{\Lambda_L}(x) [\Delta_k f](x)}_{L^p}
    \left(1 + (2L)^{-\alpha}
        \hspace{-0.8em} \sum_{j \in \Z^d \setminus \{0\}} \hspace{-0.4em} \abs{j}^{-\alpha} \right).
\end{split}
\end{equation}
If $\alpha > d$, then the sum is finite.
Again, this estimate is uniform in $k$.

Finally, let us note that we defined $B^s_{p,r}(\rho)$ as
the closure of $C^\infty_c$ functions with respect to the norm;
it is not \emph{a priori} obvious that the periodic $f$ belongs to this closure.
We can however approximate $f$ with $n$ repeats of $f \I_{\Lambda_L}$
(with a smooth cutoff in the tails).
A modification of the preceding computation shows that the approximation converges
in $B^s_{p,r}(\rho)$ norm as $n \to \infty$.
\end{proof}

Finally, the following lemma about Besov regularity of indicator functions will be used
in Section~\ref{sec:globalization}.
\begin{lemma}[Besov norm of indicator]\label{thm:characteristic function}
For $1 < p < \infty$ and any $K > 1$,
the indicators of balls $B(0,R) \subset \R^d$ satisfy
\[
\sup_{R \leq K}\norm{\I_{B(0,R)}}_{B^{1/p}_{p,\infty}(\R^d)} \lesssim K^{d/p}.
\] 
\end{lemma}
\begin{proof}
By the first theorem in \cite[Section 2.6.1]{triebel_theory_1992} we have
\begin{equation} 
\norm{f}_{B^{s}_{p,\infty}}
\lesssim \norm{f}_{L^{p}}
    + \sup_{\abs h \leq 1} \bignorm{\frac{f(x+h)-f(x)}{h^{s}}}_{L^{p}}.
\end{equation}
Now clearly $\sup_{R\leq K} \norm{\I_{B(0,R)}}_{L^{p}} \lesssim K^{d/p}$, and 
$\abs{\I_{B(0,R)}(x+h)-\I_{B(0,R)}(x)}$
is bounded by $1$ and nonzero only in $\partial B(0,R) + B(0, h)$.
This set has measure bounded by $C_d K^{d-1} \abs h$.
Thus
\begin{equation}
\bignorm{\frac{f(x+h)-f(x)}{h^{s}}}_{L^{p}} \leq K^{(d-1)/p} \abs h^{1/p} \abs h^{-s},
\end{equation}
which is bounded by $K^{(d-1)/p}$ if $s \leq 1/p$.
\end{proof}

\begin{remark}
Let us remark that the sharp Fourier cutoff $P_N$ to $B(0, 2^N)$ is bounded uniformly in $N$
on $L^2$ and $H^s$ equipped with flat weight over $\Lambda_L$ or $\R^2$.
This is not the case in other $L^p$ spaces when $p \neq 2$.

We need to use a sharp cutoff to apply invariance of measure in Section~\ref{sec:periodic stoch}.
A smooth cutoff would have better analytic properties
but not be compatible with our dynamics
(see also~\cite[p.~17]{Bringmann_invariant_2022-1}).
\end{remark}

\section{Stochastic quantization}
\label{sec:stochastic}

In this section we construct the $\phi^4$ measure (later denoted $\mu$) in the infinite domain $\R^2$
equipped with a suitable weight.
This construction is well-known in the literature of stochastic quantization,
and we only outline the results we will need.

We define the stochastic objects both on
the periodic space $\Lambda_L \coloneqq {[{-L}, {L}]}^2$ and the full space $\R^2$.
The basic building block, Gaussian free field, is straightforwardly defined in both cases,
whereas for the $\phi^4_2$ we need to take a weak limit as $L \to \infty$.

\begin{remark}
Since we use the complex $\phi^4$ measure in Section~\ref{sec:schrodinger},
we state results here with respect to both real and complex scalar fields.
The complex case is much less frequent in the literature,
but the basic ideas are essentially same.
It is however important to notice that the definition of some objects (like $\wick{u \abs u^2}$)
depends on the choice of scalar field.
\end{remark}

\subsection{Gaussian free field}\label{sec:GFF}

\begin{definition}[Gaussian free field]
The Gaussian free field $\nu_L$ with mass $m^2 > 0$
is the Gaussian measure on $\mathcal S'(\Lambda_{L})$ with covariance
\[ \int \langle f, Z_L \rangle \langle g, Z_L \rangle \diff \nu_L (Z_L)
= \langle f, (m^2 - \Delta)^{- 1} g \rangle_{L^2 (\Lambda_{L})} .
\]
Similarly we can introduce the infinite-volume massive GFF $\nu$
supported on the space of tempered distributions $\mathcal S'(\R^2)$, with covariance
\[ \int \langle f, Z \rangle \langle g, Z \rangle \diff \nu_L (Z)
= \langle f, (m^2 - \Delta)^{- 1} g \rangle_{L^2 (\R^2)} . \]
\end{definition}

\begin{definition}[Notation for samples]
We will denote random variables from $\nu_{L}$ or $\nu$ by $Z_{L}$ and $Z$.
We will also write their projections as $Z_{L,N} \coloneqq P_{N}Z_{L}$ and $Z_{N} \coloneqq P_{N}Z$.
\end{definition}

Note that we can view $\nu_L$ as a measure on $\mathcal S' (\R^2)$ by periodic
extension. The following proposition is proved in \cite[Theorem~5.1]{mourrat_global_2017}.

\begin{theorem}[Uniform bounds for GFF]\label{thm:gff uniform bounds}
$\nu_L$ and $\nu$ have samples almost surely in $\mathcal{C}^{- \varepsilon} (\rho)$,
and for all $p < \infty$ the expectations are bounded uniformly in $L$:
\[
\sup_L \int \norm{Z_L}_{\mathcal C^{- \varepsilon} (\rho)}^p \diff \nu_L(Z_L) < \infty,
\quad
\int \norm{Z}_{\mathcal C^{- \varepsilon} (\rho)}^p \diff \nu(Z) < \infty.
\]
\end{theorem}

We can sample from the GFF by realizing it as
\begin{equation}
Z_{L} = \frac{1}{L} \sum_{n \in L\inv \Z^{2}}\frac{g_{n} e_n}{(m^2+|n|^{2})^{1/2}},
\end{equation}
where $g_{n}$ are standard complex Gaussians
and $e_n(x) \coloneqq \exp(2\pi i n \cdot x)$.
In case of the real scalar field we require $g_{-n} = \overline{g_n}$,
but otherwise $g_n$ are independent.
For the full-space case we can write
\begin{equation}
Z = \int_{\R^{2}} \frac{\xi(y) e_y}{(m^2+|y|^{2})^{1/2}} \dy
\end{equation}
where $\xi$ is a white noise as defined below.
  
\begin{definition}[White noise]\label{def:white noise}
Let $X = \Lambda_L$ or $X = \R^2$.
The white noise $\xi$ is a Gaussian process on $\mathcal{S}'(X)$ with covariance
\[
\E [\dual{f, \xi}_{L^{2}(X)} \dual{g, \xi}_{L^{2}(X)}]
= \dual{f, g}_{L^{2}(X)}.
\]
\end{definition}

The argument of Theorem~\ref{thm:gff uniform bounds}
also gives that the white noise has bounded expectation in $\mathcal C^{-1-\varepsilon}(\rho)$.

The GFF measure $\nu_L$ does not have samples of positive regularity. This means
that taking powers of distributions sampled from $\nu_L$ does not make sense.
Yet the Gaussian structure of the randomness allows us to still define
powers of the field by so-called Wick ordering.

\begin{definition}[Wick ordering, periodic space]
Let $a_{L, N} = \E \abs{Z_{L, N}(0)}^2$.
When the scalar field is real, we define the first Wick powers of $Z_{L,N}$ as
\begin{align*}
\wick{Z^3_{L, N}}_L & = Z^3_{L, N} - 3 a_{L, N} Z_{L, N},\\
\wick{Z^2_{L, N}}_L & = Z^2_{L, N} - a_{L, N},\\
\wick{Z_{L, N}}_L & = Z_{L, N}.
\end{align*}
\end{definition}

This definition is based on Hermite polynomials,
and higher-order powers can be defined accordingly.
As $N \to \infty$, the constants $a_{L, N}$ diverge logarithmically, and
the counter\-terms cancel the divergence of $Z^k_{L,N}$.
For more details, see e.g.~\cite[Chapter~I]{simon_pphi_2_1974} or~\cite{glimm_quantum_1987}.

Wick-ordered polynomials are defined by Wick-ordering each monomial term separately.
We remark that $\E \abs{Z_{L,N}(x)}^2$ does not depend on the choice of $x$
since the GFF is translation-invariant.

It will be useful to define the Wick powers with a renormalization constant
that is independent of $L$.
For this purpose we will use the expectation of the full-space GFF.

\begin{definition}[Wick ordering, full space]\label{def:wick full space}
When the scalar field is real, we denote $a_N = \E \abs{Z_N (0)}^2$ and define
\begin{align*}
\wick{Z^3_{L, N}} & = Z^3_{L, N} - 3 a_N Z_{L, N},\\
\wick{Z^2_{L, N}} & = Z^2_{L,N} - a_N,\\
\wick{Z_{L, N}} & = Z_{L,N}.
\end{align*}
\end{definition}

The difference between these two renormalizations is a polynomial of strictly lower degree;
for the third Wick powers it is
\begin{equation}
\wick{Z^3_{L, N}}_L - \wick{Z^3_{L, N}}
= - 3 (a_{L, N} - a_N) Z_{L, N}.
\end{equation}
The next lemma asserts that the difference of renormalization constants goes to zero as $N, L \to \infty$.
This lets us always take Wick ordering with respect to the full-space GFF.

\begin{lemma}[Difference of renormalization constants]\label{lemma:wick-ordering-rem}
We have
\[
| a_{L, N} - a_N | \lesssim \frac{1}{N} + \frac{1}{L},
\quad \text{when } L > 1,\; N \in \N.
\]
\end{lemma}
\begin{proof}
By covariance of the continuum white noise, the second renormalization constant is
\begin{equation}
a_{N}
= \int_{\abs x \leq N} \frac{1}{m^2 + \abs x^2} \dx.
\end{equation}
The first renormalization constant can be written as
\begin{equation}
a_{L,N}
= \frac{1}{L^2} \sum_{\substack{n \in L\inv \Z^{2}\\ \abs n \leq N}} \frac{1}{m^2 + \abs n^2}
= \int_{S_{L,N}} \frac{1}{m^2 + \abs{n(x)}^2} \dx,
\end{equation}
where $P(n)$ is the rectangle $n + [{0}, {1/L})^2$,
$n(x)$ is the unique $n\in L^{-1}\Z^{2}$ such that $x \in P(n)$,
and the collection of rectangles is denoted by
\begin{equation}
S_{L,N} \coloneqq \bigcup_{\substack{n \in L^{-1}\Z^2,\\ |n| \leq N}} P(n).
\end{equation}

Observe that by triangle inequality $B(0, N-2/L) \subset S_{L,N} \subset B(0, N+2/L)$.
Thus we can estimate 
\begin{equation} 
\begin{split}
\abs{a_N - a_{L,N}}
&\leq \int_{B(0,N-2/L)} \abs{\frac{1}{m^2 + \abs x^2}-\frac{1}{m^2 + \abs{n(x)}^2}} \dx\\
&\quad {}+ \int_{\mathcal R} \left[ \frac{1}{m^2 + \abs x^2} + \frac{1}{m^2 + \abs{n(x)}^2} \right] \dx,
\end{split}
\end{equation}
where we denote the annulus $B(0,N+2/L) \setminus B(0,N-2/L)$ by $\mathcal R$.
The first term is estimated by 
\begin{equation}
\begin{split}
\int_{\R^2} \frac{\abs{\abs x^2 - \abs{n(x)}^2}}{(m^2 + \abs{n(x)}^2)(m^2 + \abs x^2)} \dx
&= \int_{\R^2} \frac{\big|{\abs x - \abs{n(x)}}\big| (\abs x + \abs{n(x)})}
    {(m^2 + \abs{n(x)}^2)(m^2 + \abs x^2)} \dx\\
&\lesssim \frac{1}{L} \int_{\R^2} \frac{1}{1 + \abs x^3} \dx,
\end{split}
\end{equation}
since $\abs x \simeq \abs{n(x)}$ away from the origin.
For the same reason,
\begin{equation}
\int_{\mathcal R} \left[ \frac{1}{m^2 + \abs x^2} + \frac{1}{m^2 + \abs{n(x)}^2} \right] \dx
\lesssim \frac{\abs{\mathcal R}}{N^{2}}
\lesssim \frac{1}{NL}.
\qedhere
\end{equation}
\end{proof}

Let us then define the complex renormalized nonlinearity used in~\eqref{eq:nls}.
The idea is to renormalize the real and imaginary parts of the GFF separately,
as they are independent.
See~\cite{oh_pedestrian_2018} for more exposition.
In fact the same argument gives all $\wick{\abs{Z_{L,N}}^{2n}}$ for $n \in \N$,
but we only use $\wick{\abs{Z_{L,N}}^2}$ in what follows.

\begin{lemma}[Wick-ordered complex objects]\label{thm:wick complex}
When the scalar field is complex,
\[
\begin{split}
\wick{\abs{Z_{L,N}}^2} &= \abs{Z_{L,N}}^2 - a_N,\\
\wick{Z_{L,N} \abs{Z_{L,N}}^2} &= Z_{L,N} \abs{Z_{L,N}}^2 - 2a_N Z_{L,N}.
\end{split}
\]
\end{lemma}
\begin{proof}
Let us abbreviate $R = \operatorname{Re} Z_{L,N}$ and $I = \operatorname{Im} Z_{L,N}$.
It then follows from the definition that $R$ and $I$ are independent real GFFs
such that $\E R(x)^2 = \E I(x)^2 = a_N/2$.
Then
\begin{equation}
\wick{\abs{Z_{L,N}}^2}
= \wick{R^2 + I^2}
= R^2 - \frac{a_N}{2} + I^2 - \frac{a_N}{2},
\end{equation}
from which the first statement follows.
Similarly,
\begin{equation}\label{eq:wick complex expansion}
\begin{split}
\wick{Z_L \abs{Z_L}^2}
&= \wick{R (R^2 + I^2) + iI(R^2 + I^2)}\\
&= \wick{R^3} + i \wick{I^3} + \wick{R I^2} + i \wick{R^2 I}.
\end{split}
\end{equation}
By the Wick product expansion (see e.g.~\cite[p.~12]{simon_pphi_2_1974}) we have
\begin{equation}
\wick{R I^2} = RI^2 - 2 \E [R I] - R \E I^2 = R (I^2 - a_N / 2),
\end{equation}
and similarly for $\wick{R^2 I}$.
Hence
\begin{equation}
\eqref{eq:wick complex expansion}
= (R^3 + iI^3) - \frac{3a_N(R + iI)}{2} + (RI^2 + iR^2 I) - \frac{a_N(R + iI)}{2},
\end{equation}
which is exactly the second proposition.
\end{proof}

We can now state that the relevant Wick powers of the Gaussian free field are well-defined.
Furthermore, we show that the result extends to sufficiently regular perturbations of the GFF,
of which the $\phi^4$ measure will be an example.

\begin{lemma}[Moments of GFF powers]\label{thm:wick gff moments}
First consider the real scalar field.
For any $p < \infty$ and $j = 1, 2, \ldots$ we have
\[
\sup_N \E \left[\norm{\wick{Z^j_{L, N}}}^p_{\mathcal C^{-\varepsilon}(\rho)} \right]
< \infty.
\]
The sequence $\wick{Z^j_{L, N}}$ converges in $L^p (\nu_L, C^{-\varepsilon}(\rho))$ to a
well-defined limit $\wick{Z_L^j}$ as $N \to \infty$.
The limit satisfies
\[
\sup_L \E \left[\norm{\wick{Z^j_L}}^p_{\mathcal C^{-\varepsilon}(\rho)} \right]
< \infty.
\]
In the complex case the same convergence result holds for
$Z_{L,N}^2$, $\wick{\abs{Z_{L,N}}^2}$, and $Z_{L,N} \abs{Z_{L,N}}^2$,
and the respective limits satisfy
\[
\sup_L \E \left[\norm{Z_L^2}^p_{\mathcal C^{-\varepsilon}(\rho)}
+ \norm{\wick{\abs{Z_L}^2}}^p_{\mathcal C^{-\varepsilon}(\rho)}
+ \norm{\wick{Z_L \abs{Z_L}^2}}^p_{\mathcal C^{-\varepsilon}(\rho)} \right]
< \infty.
\]
\end{lemma}
\begin{proof}
The proof of the first statement is a variation of \cite[Lemma~3.2]{da_prato_strong_2003},
and the infinite-volume bound is done in \cite[Section~5]{mourrat_global_2017}.

The complex results then follow from these real-valued objects.
Let us again denote $R = \operatorname{Re} Z_L$ and $I = \operatorname{Im} Z_L$.
By Lemma~\ref{thm:wick complex} we have that
\begin{equation}
\wick{\smallabs{Z_L^2}} = \wick{R^2} + \wick{I^2},
\end{equation}
so its moment bound follows immediately from the real case.

The second power can be written as
\begin{equation}
Z_L^2 = R^2 + 2i R I - I^2 = \wick{R^2} + 2i R I - \wick{I^2},
\end{equation}
so we need to show the bound for $RI$.
This is a matter of adapting the proof of \cite[Theorem~5.1]{mourrat_global_2017}
using two observations:
\begin{itemize}
\item $RI$ belongs to the second Wiener chaos
    (over a tensorized space so that the Gaussian has two independent real components)
    so hypercontractivity can be used;
\item In the notation of \cite[Lemma~9]{mourrat_global_2017}, we can compute
\begin{equation}
\begin{split}
&\mathrel{\phantom{=}} \E \abs{[RI](t, \eta_k(\mathord{\cdot} - x))}^2\\
&= \E \int_{\R^2} \int_{\R^2} \eta_k(x_1 - x) \eta_k(x_2 - x)
    R(x_1) R(x_2) I(x_1) I(x_2) \dx_1 \dx_2\\
&= \int_{\R^2} \int_{\R^2} \eta_k(x_1 - x) \eta_k(x_2 - x)
    \E [R(x_1) R(x_2)] \E [I(x_1) I(x_2)] \dx_1 \dx_2\\
&= \int_{\R^2} \int_{\R^2} \eta_k(x_1 - x) \eta_k(x_2 - x)
    \mathscr K(t, t, x_1 - x_2)^2 \dx_1 \dx_2,
\end{split}
\end{equation}
where $\mathscr K(t, t, x_1 - x_2)^2$ is the same kernel as for $\wick{R^2}$.
\end{itemize}
From here on, the proof is hence identical to that of $\wick{R^2}$.

The finite-volume bound and convergence of $Z_{L,N} \wick{\abs{Z_{L,N}}^2}$
are shown in~\cite[Proposition~1.3]{oh_pedestrian_2018}.
With the expansion~\eqref{eq:wick complex expansion} and the same observations as above,
the infinite-volume bounds are analogous to those for $\wick{R^3}$.
\end{proof}

\begin{lemma}[Wick powers of perturbations]\label{thm:wick binomial}
Suppose the real scalar field.
Let $\psi \in L^{2qj}(\nu_L, B_{pj, p}^{2 \varepsilon}(\rho))$,
where $\varepsilon > 0$ and $1 \leq p, q < \infty$.
Then
\[
\wick{(Z_L + \psi)^j} = \sum_{i = 0}^j \binom j i \wick{Z^{j - i}_L} \psi^i
\]
as an element of $L^q(\nu_L, B_{p, p}^{-\varepsilon}(\rho^{j+1}))$.
\end{lemma}
\begin{proof}
It follows from properties of Hermite polynomials that
\begin{equation}
\wick{(Z_{L, N} + \psi)^j}
= \sum_{i = 0}^j \binom j i \wick{Z^{j - i}_{L, N}} \psi^i.
\end{equation}
Hence by Theorem~\ref{thm:besov multiplication} we have
\begin{equation}
\begin{split}
\E \norm{\wick{(Z_{L, N} + \psi)^j}}_{B_{p, p}^{-\varepsilon}(\rho^{j+1})}^q
&\lesssim \sum_{i = 0}^j
    \norm{\wick{Z^{j - i}_{L, N}} \psi^i}_{B_{p, p}^{-\varepsilon}(\rho^{j+1})}^q\\
&\lesssim \sum_{i = 0}^j
    \norm{\wick{Z^{j - i}_{L, N}}}_{\mathcal C^{-\varepsilon}(\rho)}^q
    \norm{\psi}_{B_{pi, p}^{2 \varepsilon}(\rho^{j/i})}^{qi},
\end{split}
\end{equation}
and the claim for finite $N$ follows by Jensen's inequality and Lemma~\ref{thm:wick gff moments}.
Since multiplication is a continuous operation, the claim holds also as $N \to \infty$.
\end{proof}

The complex case leads to longer expressions;
for us it suffices to expand
\begin{equation}\label{eq:wick abs perturbation}
\wick{(Z_L + \psi) \abs{Z_L + \psi}^2}
= (Z_L + \psi) \left[ (Z_L + \psi) \overline{(Z_L + \psi)} - 2a_N \right]
\end{equation}
and redistribute the renormalization constant.
This is done in~\eqref{eq:sq tested complex} below.

The following result lets us compute covariances of Wick powers by passing to a Green's function.
For the proof, see e.g.\ \cite[Theorem~I.3]{simon_pphi_2_1974}.

\begin{theorem}[Wick's theorem]
If $X$ and $Y$ are Gaussian, then
\[
\E [\wick{X^m} \; \wick{Y^n}] = \I_{m=n} n! \left( \E [XY] \right)^n.
\]
\end{theorem}

As an application of Wick's theorem,
we see that we can approximate the third Wick power by continuous maps.
We use this lemma to prove that sequences of periodic solutions to~\eqref{eq:nlw} or~\eqref{eq:nls}
satisfy the PDEs also in the limit.
The proof is somewhat technical, and we leave it to Appendix~\ref{sec:wick approx proof}.

\begin{lemma}[Approximation of Wick powers]
\label{thm:wick approx}
Let $2 \leq p < \infty$.
For every $\delta > 0$ and $s > 0$,
there exists a continuous map $f^{3,\delta} \colon H^{-s}(\rho) \to L^2(\rho)$
such that
\[
\lim_{\delta \to 0} \sup_{L}
\E \norm{f^{3,\delta}(Z_L + \psi_L) - \wick{(Z_L + \psi_L)^3}}_{\mathcal C^{-\varepsilon}(\rho)}^p = 0,
\]
where $Z_L$ is sampled from the Gaussian free field with period $1 \leq L \leq \infty$,
and $\psi_L \in L^{4p}(\Prob, B_{p,p}^\varepsilon(\rho))$.
In the complex case, $f^{3,\delta}$ is instead defined such that
\[
\lim_{\delta \to 0} \sup_L
\E \norm{f^{3,\delta}(Z_L + \psi_L) - \wick{(Z_L + \psi_L) \abs{Z_L + \psi_L}^2}}
    _{\mathcal C^{-\varepsilon}(\rho)}^p = 0.
\]
\end{lemma}

\subsection{Coupling of the \texorpdfstring{$\phi_4$}{phi\textasciicircum4} measure and the GFF}\label{sec:SQ}

We now turn to study the $\phi_2^4$ measure.
We can define it directly only in the periodic case;
we need to take a weak limit to get to infinite volume.

Let us first recall the definition and some basic results of weak convergence of probability measures.
These can be found in most probability textbooks;
see for example \cite[Sections~2 and~5]{billingsley_convergence_1999}.

\begin{theorem}[Weak convergence]
Let $\mathcal X$ be a metric space
and $C_b(\mathcal X; \R)$ the space of bounded continuous functions on it.
A sequence of Borel probability measures $(\mu_L)$ on $\mathcal X$
is said to \emph{converge weakly} to $\mu$ if
\[
\lim_{L \to \infty} \int f(\phi) \diff\mu_L(\phi)
= \int f(\phi) \diff\mu(\phi)
\quad\text{for all } f \in C_b(\mathcal X; \R).
\]
If $\mathcal X$ is a Polish space, then the weak limit is unique.
\end{theorem}

\begin{definition}[Tightness]
A family $(\mu_L)_{L \in \N}$ of Borel probability measures on a metric space $\mathcal X$ is \emph{tight}
if for any $\varepsilon > 0$ there exists a compact set $K_\varepsilon$
such that
\[
\sup_{L \in \N} \mu_L(\mathcal X \setminus K_\varepsilon) < \varepsilon.
\]
\end{definition}

\begin{lemma}[Prokhorov's theorem; {\cite[Theorem~5.1]{billingsley_convergence_1999}}]
Suppose that the sequence $(\mu_L)$ defined above is tight.
Then there is a subsequence $(\mu_{L_k})$
that converges weakly to a Borel measure $\mu$ on $\mathcal X$.
\end{lemma}

\begin{lemma}[Weak limits in product spaces; {\cite[Theorem~2.8]{billingsley_convergence_1999}}]
Assume that Borel probability measures
$(\mu_L)$ and $(\mu_L')$ converge weakly to $\mu$ and $\mu'$
on the Polish spaces $\mathcal X$ and $\mathcal X'$ respectively.
Then $(\mu_L \times \mu_L')$ converges weakly to $\mu \times \mu'$ on $\mathcal X \times \mathcal X'$.
\end{lemma}

\begin{lemma}[Skorokhod's theorem; {\cite[Theorem~6.7]{billingsley_convergence_1999}}]
\label{thm:skorokhod}
Suppose that $(\mu_L)$ converge weakly to $\mu$ supported on a Polish space.
Then there exist a common probability space $\tilde\Prob$
and random variables $X_L$, $X$ such that $\text{Law}(X_L) = \mu_L$, $\text{Law}(X) = \mu$,
and $X_L \to X$ almost surely.
\end{lemma}

We will consider a sequence of $\phi^4_{2,L}$ measures over increasingly large tori
and show that it is tight over a polynomially weighted Besov space.
This will give us a weak limiting measure $\phi_2^4$.

\begin{definition}[Periodic $\phi^4_2$]\label{def:phi4 periodic}
The $\phi^4_{2,L}$ measure over $\Lambda_L$ is given by
\[
\diff\mu_L (\phi) \coloneqq \mathcal Z_L\inv
\exp \left( -\int_{\Lambda_L} \wick{ \abs{\phi(x)}^4 } \dx \right) \diff \nu_L (\phi),
\]
where $\mathcal Z_L\inv$ is a normalization constant.
\end{definition}

The Wick power $\wick{\abs\phi^4}$ (meaning $\wick{\phi^4}$ in the real case)
makes sense as a distribution $\nu_L$-almost surely, and one can show that
the exponential belongs to $L^p (\nu_L)$
for any $p < \infty$ and $L < \infty$; see e.g.~\cite{barashkov_variational_2020}
and~\cite[Proposition~1.2]{oh_pedestrian_2018}.

For our purposes, it is easier to view $\mu_L$ as an invariant measure to a stochastic PDE.
This approach is known as \emph{stochastic quantization}.
As discussed in Section~\ref{sec:intro phi4},
this approach has been hugely successful in deducing properties of the measure.
The following result was one of the first breakthroughs in this approach:

\begin{theorem}[Parabolic stochastic quantization]\label{thm:SQ}
For any finite $L$, the measure $\mu_L$ is the unique invariant measure
of the stochastic quantization equation
\begin{equation}\label{eq:SQ}
    \partial_t W_L + (m^2 - \Delta) W_L + \wick{W_L \abs{W_L}^2} = \xi, \quad W_L \in C
    (\R_+, H^{- \varepsilon} (\Lambda_L)).
\end{equation}
Here $\xi$ is space-time white noise as in Definition~\ref{def:white noise}.
\end{theorem}
\begin{proof}
The real case was originally shown
by Da~Prato and Debussche \cite{da_prato_strong_2003};
see also \cite{mourrat_global_2017} for discussion and extension to infinite volume.
Uniqueness follows from \cite[Corollary~6.6]{tsatsoulis_spectral_2018},
although we will not use this fact below.

The complex case follows by a modification of the argument in~\cite{da_prato_strong_2003}.
In the fixpoint argument \cite[Proposition~4.4]{da_prato_strong_2003}
we need to replace the Wick-ordered third power with~\eqref{eq:wick abs perturbation}.
As the stochastic terms have the same Besov regularity as in the real case,
the proof still holds.
Similarly, the globalization argument that ends \cite[Section~4]{da_prato_strong_2003}
can be modified by replacing the polynomial $\wick{p(W_L)}$ with $\wick{W_L \abs{W_L}^2}$
and using Lemma~\ref{thm:wick gff moments}.
\end{proof}

\begin{remark}\label{rem:sq hyperbolic}
We use the better-known parabolic stochastic quantization argument,
but $\mu_L$ can also be viewed as an invariant measure to a stochastic nonlinear wave equation;
this is called \emph{hyperbolic} or \emph{canonical} stochastic quantization.
See \cite{gubinelli_global_2022} for the construction on the torus;
the argument is quite similar to Section~\ref{sec:periodic},
with slightly different linear propagators and the appearance of stochastic forcing.
It was the extension of this equation to $\R^2$ that was completed in \cite{oh_hyperbolic_2022}.

Our proof of Theorem~\ref{thm:sq tightness} requires the parabolic equation.
Corollary~\ref{cor:moments-strong} does not translate to the hyperbolic case at all,
since the wave operator has a smoothing effect of only one derivative
compared to two for the heat operator.
\end{remark}

The Da~Prato--Debussche argument is based on decomposing the solution into two parts.
Since the other part is more regular,
this shows that on short spatial scales the $\phi^4_2$ measure looks like the GFF.

\begin{theorem}[Tightness of $\phi^4_{2,L}$]\label{thm:sq tightness}
Samples from $\mu_L$ can be decomposed as the sum of Gaussian free field $Z_L$
and a random function $\psi_L \in H^{1}(\rho)$.
The laws of $Z_L$ and $\psi_L$ are then tight in $H^{-\varepsilon}(\rho^{1+\varepsilon})$
and $H^{1-\varepsilon}(\rho^{1+\varepsilon})$ respectively.
\end{theorem}
\begin{proof}
We will use {\eqref{eq:SQ}} to control the $\phi_{2, L}^4$ measure in
the limit $L \to \infty$.
We begin by decomposing the solution as $W_L = Z_L + \psi_L$, where $Z_L$ is the Gaussian part
that solves the stationary equation
\begin{equation}\label{eq:linear-heat}
\begin{aligncases}
\partial_t Z_L(t) + (m^2 - \Delta) Z_L(t) &= \xi(t),\\
\Law(Z_L(0)) &= \text{GFF}(\Lambda_L),
\end{aligncases}
\end{equation}
and $\psi_L$ solves
\begin{equation}\label{eq:SQ-change}
\begin{aligncases}
\partial_t \psi_L(t) + (m^2 - \Delta) \psi_L(t) &= -\wick{(Z+\psi_L)\abs{Z+\psi_L}^2}(t),\\
\psi_L(0) &= W_L(0) - Z_L(0).
\end{aligncases}
\end{equation}
We can take $W_L$ and $Z_L$ to be jointly stationary solutions to
{\eqref{eq:SQ}} and {\eqref{eq:linear-heat}}
so that $\Law(Z_L(t)) = \text{GFF}$;
see the beginning of Section~4.3 in \cite{gubinelli_pde_2021}.
In particular the Wick powers $\wick{Z_L^j}(t)$ are well-defined random distributions,
and the laws of $Z_L$ form a tight sequence by Theorem~\ref{thm:gff uniform bounds}.

It then suffices to show that
\begin{equation}
\sup_L \E \norm{\psi_L(t)}^2_{W^{1,2}(\rho)} < \infty,
\end{equation}
since by Theorem~\ref{thm:besov sobolev}
this implies that large balls in $H^1(\rho)$ norm have high probability,
and such balls are compact in $H^{1-\varepsilon}(\rho^{1+\varepsilon})$
by Theorem~\ref{thm:besov compactness}.

\bigskip\noindent
In the real case we multiply {\eqref{eq:SQ-change}} by $\rho^2 \psi_L(t)$ and integrate in $x$ to obtain
\begin{equation}\label{eq:sq tested}
\begin{split}
&\mathrel{\phantom{=}} \frac 1 2 \partial_t \norm{\rho \psi_L(t)}^2_{L^2} + m^2 \norm{\rho \psi_L(t)}^2_{L^2}
    + \norm{\rho \nabla \psi_L(t)}_{L^2}^2 + \norm{\rho^{1/2} \psi_L(t)}^4_{L^4} \\
    &= - G_t (Z_L, \psi_L),
\end{split}
\end{equation}
where the right-hand side is
\begin{equation}
\begin{split}
G_t (Z_L, \psi_L) &= 3 \int \rho^2 \wick{ Z_L(t)^3 }\; \psi_L(t) \dx
    + 3 \int \rho^2 \wick{ Z_L(t)^2 }\; \psi_L^2(t) \dx\\
&\qquad+ \int \rho^2 Z_L(t) \psi_L^3(t) \dx + \int \psi_L(t)
    (\nabla \rho^2 \cdot \nabla \psi_L(t)) \dx.
\end{split}
\end{equation}
In the complex case we instead compute
$\frac{\rho^2}{2} [
    \overline{\psi_L(t)} \cdot \eqref{eq:SQ-change}
    + \psi_L(t) \cdot \overline{\eqref{eq:SQ-change}}]$
to get equation~\eqref{eq:sq tested} with the
right-hand side replaced with $-G_t - \overline{G_t}$, where
\begin{equation}\label{eq:sq tested complex}
\begin{split}
G_t(Z_L, \psi_L)
&= \frac 1 2 \int \rho^2 \wick{Z_L(t) \abs{Z_L(t)}^2}\; \overline{\psi_L(t)} \dx\\
&\qquad + \int \rho^2 \wick{ \abs{Z_L(t)}^2 }\; \abs{\psi_L(t)}^2 \dx
    + \frac 1 2 \int \rho^2 Z_L(t)^2 \overline{\psi_L(t)^2} \dx\\
&\qquad + \int \rho^2 Z_L(t) \overline{\psi_L(t)} \abs{\psi_L(t)}^2 \dx
    + \frac 1 2 \int \rho^2 \overline{Z_L(t)} \psi_L(t) \abs{\psi_L(t)}^2 \dx\\
&\qquad + \frac 1 2 \int \overline{\psi_L(t)} (\nabla \rho^2 \cdot \nabla \psi_L(t)) \dx.
\end{split}
\end{equation}
In Appendix~\ref{sec:phi42 appendix} we show that in either case
\begin{equation}
\abs{G_t (Z_L, \psi_L)}
    \leq \delta \left(\norm{\psi_L(t)}^2_{W^{1,2}(\rho)}
   + \norm{\psi_L(t)}^4_{L^4(\rho^{1/2})}\right)
   + Q_t(Z_L),
\end{equation}
where
$\sup_L \E [| Q_t (Z_L) |^p] = \sup_L \E [| Q_0 (Z_L) |^p] < \infty$ for any $p < \infty$.
Meanwhile the left-hand side of~\eqref{eq:sq tested} is bounded from below by
\begin{equation}
\frac 1 2 \partial_t \norm{\psi_L(t)}^2_{L^2(\rho)} + (m^2 \wedge 1) \norm{\psi_L(t)}^2_{W^{1,2}(\rho)}
    + \norm{\psi_L(t)}^4_{L^4(\rho^{1/2})}.
\end{equation}
Combining these two, we get
\begin{equation}\label{eq:sq near final}
\frac 1 2 \partial_t \norm{\rho \psi_L(t)}^2_{L^2}
+ ((m^2 \wedge 1) - \delta) \left( \norm{\psi_L(t)}^2_{W^{1,2}(\rho)}
    + \norm{\rho^{1/2} \psi_L(t)}^4_{L^4} \right)
\leq Q_t (Z_L).
\end{equation}
The second term on the left is non-negative when $\delta$ is chosen small enough.
We may ignore the $L^4$ term.
If we integrate~\eqref{eq:sq near final} over an arbitrary interval~$[0,T]$
and take expectation, we get
\begin{equation}
\begin{split}
&\mathrel{\phantom{=}}
\frac 1 2 \E \left[ \norm{\rho \psi_L(T)}^2_{L^2} - \norm{\rho \psi_L(0)}^2_{L^2} \right]
+ ((m^2 \wedge 1) - \delta) \E \int_0^T \norm{\psi_L(t)}^2_{W^{1,2}(\rho)} \dt\\
&\leq T \E Q_0 (Z_L).
\end{split}
\end{equation}
Now the $L^2$ terms cancel by stationarity of $\psi_L$.
Similarly we may commute the expectation and integral around the $W^{1,2}$ term,
and be left with
\begin{equation}
\E \norm{\psi_L(0)}^2_{W^{1,2}(\rho)}
\leq \frac{\E Q_0 (Z_L)}{(m^2 \wedge 1) - \delta},
\end{equation}
which implies the claim.
\end{proof}

\subsection{Wick powers of \texorpdfstring{$\phi_2^4$}{phi42}}\label{sec:phi42 powers}

The bounds on the $\phi^4$ samples can be improved
to exponential tails, which then imply $L^p$ expectations for all $p$.
We defer the proof of this result to Appendix~\ref{sec:tails}.

\begin{theorem}[Exponential tails]\label{thm:phi42 exponential tails}
There exists $\delta > 0$ such that
\[
\sup_L \int \exp\left(\delta \norm{W_L}_{\mathcal C^{-\varepsilon}(\rho)}^2\right) \diff\mu_L(W_L)
\lesssim 1.
\]
The bound also holds in the limit $\mu$.
\end{theorem}

Since the nonlinearity in \eqref{eq:nlw} is cubic,
we will need the first three Wick powers of the $\phi^4$ field.
We construct and estimate the Wick powers of $\phi_{2,L}^4$ uniformly in $L$,
and thus in the $L \to \infty$ limit.
In the proof of this lemma, we rescale $\varepsilon$ so that
$\varepsilon$ of Theorem~\ref{thm:sq tightness} is now denoted $\varepsilon/12$.

\begin{theorem}[Wick powers of $\phi^4$]\label{thm:phi42 wick moments}
Let $W_L = Z_L + \psi_L$ be sampled from $\phi_{2, L}^4$ as in Theorem~\ref{thm:sq tightness}. Then
$\wick{ W_L^j }$ is a well-defined random distribution for $j \leq 3$, and for
any $\varepsilon > 0$ and $p < \infty$ we have
\[
\sup_L \E \norm{\wick{ W_L^j }}^p_{\mathcal C^{-\varepsilon}(\rho^2)} < \infty.
\]
Furthermore, if $W$ is sampled from the full-space $\phi_2^4$ measure, then
\[
\E \norm{\wick{ W^j }}^p_{\mathcal C^{-\varepsilon}(\rho^2)} < \infty.
\]
\end{theorem}
\begin{proof}
We only do the proof in the most difficult case $j = 3$. The other cases are
analogous. Recall that we have by \eqref{eq:linear-heat} and \eqref{eq:SQ-change} the decomposition
\begin{equation}
\wick{W_L^3} = \sum_{j = 0}^3 \binom 3 j \wick{Z_L^j}\; \psi_L^{3 - j}.
\end{equation}
Now for $q = 4/\varepsilon$ we can use Theorem~\ref{thm:besov multiplication} to estimate
\begin{equation}\label{eq:wick powers holder}
\begin{split}
\norm{\wick{Z_L^j}\; \psi_L^{3 - j}}_{\mathcal C^{- \varepsilon} (\rho^2)}
& \lesssim \norm{\wick{Z_L^j}\; \psi_L^{3 - j} }_{B_{q, q}^{- \varepsilon / 13} (\rho^2)}\\
& \lesssim \norm{\wick{Z_L^j} }_{\mathcal C^{- \varepsilon / 13} (\rho)}
    \norm{\psi_L^{3 - j}}_{B_{q, q}^{\varepsilon / 12} (\rho)}\\
& \lesssim \norm{\wick{Z_L^j}}_{\mathcal C^{- \varepsilon / 13} (\rho)}^{2}
    + \norm{\psi_L}^{2(3 - j)}_{B_{3 q, 3 q}^{\varepsilon / 12} (\rho^{1 / 3})}.
\end{split}
\end{equation}
The Gaussian part is bounded by Lemma~\ref{thm:wick gff moments}.
Theorem~\ref{thm:phi42 exponential tails} implies that
$\E \norm{W_L}_{\mathcal C^{- \varepsilon/12} (\rho)}^p < \infty$,
so we can estimate the perturbation as
\begin{equation}
\sup_L \E \norm{\psi_L}_{\mathcal C^{- \varepsilon/12}(\rho)}^p
\lesssim \sup_L\, (\E \norm{W_L}_{\mathcal C^{- \varepsilon/12} (\rho)}^p
    + \E \norm{Z_L}_{\mathcal C^{- \varepsilon/12} (\rho)}^p) < \infty.
\end{equation}
This estimate provides integrability,
whereas the estimate $\E \norm{\psi_L}^2_{H^{1-\varepsilon/12}(\rho^{1/6})} < \infty$
from Section~\ref{sec:SQ} provides differentiability.
We can interpolate between these two with Theorem~\ref{thm:besov interpolation}:
\begin{equation}
\begin{split}
\norm{\psi_L}_{B_{3 q, 3 q}^{\varepsilon / 12} (\rho^{1 / 3})}
&\lesssim \norm{\psi_L}^{(1-\theta)}_{\mathcal C^{- \varepsilon/12} (\rho^{1 / 6})}
    \norm{\psi_L}^{\theta}_{H^{1-\varepsilon/12}(\rho^{1 / 6})}\\
&\lesssim \norm{\psi_L}^{2(1-\theta)}_{\mathcal C^{- \varepsilon/12} (\rho^{1 / 6})}
    + \norm{\psi_L}^{2\theta}_{H^{1-\varepsilon/12}(\rho^{1 / 6})},
\end{split}
\end{equation}
where we choose $\theta = \varepsilon/6$.
As we substitute this back into \eqref{eq:wick powers holder},
we find that the final expectation is bounded.
\end{proof}

From Theorem \ref{thm:phi42 wick moments} we can bootstrap a stronger statement for the coupling.
The perturbation $\psi$ is two derivatives more regular than $Z$,
instead of just one derivative as showed earlier.

\begin{corollary}[Strong bound for regular part]\label{cor:moments-strong}
We can find random variables $Z_L, \psi_L$ such that $\Law(Z_L) = \nu_L$,
$\Law(Z_L + \psi_L) = \mu_L$, and
\[ \sup_L \E \norm{ \psi_L }^p_{H^{2 - \varepsilon} (\rho)} \lesssim 1. \]
\end{corollary}
\begin{proof}
For notational simplicity we consider the real case;
the complex case follows by modifying the Duhamel term below.
Recall that from the stochastic quantization equation~\eqref{eq:SQ} we have
\begin{equation}
\psi_L (t)
= \int^t_0 e^{- (t - s) \Delta} \wick{(Z_L (s) + \psi_L (s))^3} \ds + e^{- t \Delta} \psi_L (0).
\end{equation}
So provided $p$ is large enough that $| t - s |^{-(1 - \varepsilon / 2) p/(p-1)}$
has integrable singularity,
we can use the smoothing effect of the heat operator (\cite[Proposition~5]{mourrat_global_2017})
together with Hölder's inequality
to estimate
\begin{equation}
\begin{split}
& \E \norm{ \psi_L (t) }^p_{H^{2 - \varepsilon} (\rho)}\\
\lesssim\; &\E \norm{\int^t_0 e^{-(t-s)\Delta} \wick{(Z_L(s) + \psi_L (s))^3} \ds }^p_{H^{2-\varepsilon}(\rho)}
    + \E \norm{e^{-t\Delta} \psi_L(0)}^p_{H^{2 - \varepsilon}(\rho)}\\
\lesssim\; &\E \left[ \int^t_0
    \frac{\norm{\wick{(Z_L (s) + \psi_L (s))^3}}_{H^{-\varepsilon/2}(\rho)}}{\abs{t-s}^{1-\varepsilon/2}} \ds
    \right]^p + \frac{\E \norm{\psi_L(0)}^p_{H^{-\varepsilon/2}(\rho)}}{t^{1 - \varepsilon / 2}}\\
\leq\; &C_{t,p} \int^t_0 \E \norm{\wick{(Z_L (s) + \psi_L (s))^3}}^p_{H^{-\varepsilon / 2}(\rho)} \ds
    + \frac{\E \norm{\psi_L (0)}^p_{H^{- \varepsilon / 2} (\rho)}}{t^{1 - \varepsilon / 2}}.
\end{split}
\end{equation}
Since $Z_L$ and $\psi_L$ are both stationary, we may choose $t$ as we like.
The integrand is then uniformly bounded by Theorem~\ref{thm:phi42 wick moments}.
\end{proof}

In total we have obtained that
$\sup_L \E \norm{\psi_L}^p_{H^{2-\varepsilon}(\rho)} < \infty$.
By the same compactness argument as above, $\Law (Z_L, \psi_L)$
is tight on $H^{- 2\varepsilon} (\rho^2) \times H^{2 - 2\varepsilon} (\rho^2)$.
In particular $\mu_L = \Law (Z_L + \psi_L)$ is
tight on $H^{- 2\varepsilon} (\rho^2)$ and has a weakly converging
subsequence.
We have thus proved the following:

\begin{theorem}[$\phi^4_2$ as a weak limit]\label{thm:phi42 weak limit}
  Let $\rho$ be a sufficiently integrable polynomial weight.
  The measure $\mu_L$ can be represented as
  \[ \mu_L = \Law (Z_L + \psi_L) \]
  where $Z_L$ is a GFF on $\Lambda_L$, and $\psi_L$ satisfies
  $\sup_L \E \norm{  \psi_L }^p_{H^{2-\varepsilon}(\rho)} < \infty$. Identifying
  $Z_L + \psi_L$ with its periodic extension on $\R^2$ we have that
  $(\mu_L)$ is tight on $H^{- 2\varepsilon} (\rho^2)$ and any limiting point $\mu$ satisfies
  \[ \mu = \Law (Z + \psi) \]
  where $Z$ is a Gaussian free field on $\R^2$ and $\E \norm{\psi}^p_{H^{2-2\varepsilon} (\rho^2)} < \infty$.
\end{theorem}
\begin{proof}
Tightness was discussed above.
We know that the limit of $\Law (Z_L)$ as $L \to \infty$ is a
Gaussian free field on $\R^2$; this follows for instance from the
convergence of the covariances. It remains to show that
\begin{equation}
\E \norm{ \psi }^p_{H^{2-2\varepsilon} (\rho^2)} < \infty,
\end{equation}
but since $\norm{ \psi }^2_{H^{2-\varepsilon} (\rho)}$ is lower semicontinuous on
$H^{2 - 2\varepsilon} (\rho^2)$ we have by weak convergence
\begin{equation}
\E [\norm{ \psi }^p_{H^{2-2\varepsilon} (\rho^2)}]
\leq \liminf_{L \to \infty} \E [\norm{ \psi_L }^p_{H^{2-\varepsilon} (\rho)}]
<\infty. \qedhere
\end{equation}
\end{proof}

\begin{remark}
We were careful to state the preceding theorem for ``any limiting point $\mu$''.
When the coupling parameter $\lambda$ in~\eqref{eq:phi4 measure} is large enough,
there exist subsequences of $(\phi^4_{2,L})$ that converge to different weak limits.
This is one of the main complications in our study.
\end{remark}

\section{Invariance of periodic NLW}\label{sec:periodic}

Let us now move on to solving the nonlinear wave equation.
We fix a periodic domain $\Lambda_L = {[{-L},{L}]}^2$ and consider
\begin{equation}\label{eq:nlw bounded}
\begin{aligncases}
    \partial_{tt} u(x, t) + (m^2 - \Delta) u(x,t) &= -\wick{u(x,t)^3},\\
    u(x, 0) &= u_0(x),\\
    \partial_t u(x,0) &= u_0'(x)
\end{aligncases}
\end{equation}
on $\Lambda_L \times \R_+$.
The initial data will be sampled from $\vec{\mu}_L$,
meaning that $u_0$ is from the $\phi^4$ measure (Definition~\ref{def:phi4 periodic})
and the initial time derivative $u_0'$ from a white noise measure (Definition~\ref{def:white noise}).

\begin{remark}
The Wick ordering will always be taken with respect to the infinite-volume covariance
(Definition~\ref{def:wick full space}),
even if we start from periodic initial data.
\end{remark}

\begin{remark}
We now relabel $\varepsilon$ and $\rho$ such that the space
$H^{-2\varepsilon}(\rho^2)$ at the end of Section~\ref{sec:phi42 powers}
is now denoted by $H^{-\varepsilon}(\rho)$.
\end{remark}

By solving the equation in Fourier space, we can write the mild solution as
\begin{equation}\label{eq:mild solution}
u(t) = \Cc_t u_0 + \Cs_t u_0' - \int_0^t [\Cs_{t-s} \wick{u(s)^3}] \ds,
\end{equation}
where we use the cosine and sine operators
\begin{equation}
\Cc_t = \cos((m^2 - \Delta)^{1/2} t),
\quad
\Cs_t = \frac{\sin((m^2 - \Delta)^{1/2} t)}{(m^2 - \Delta)^{1/2}}.
\end{equation}
These are defined as Fourier multiplier operators.
We see that
$\Cc_t$ preserves the $H^s(\Lambda_L)$ regularity of its argument whereas $\Cs_t$ increases it by one derivative.

We again split the solution into nonlinear and linear parts $u = v + w$.
Here $w(x, t) = \Cc_t u_0(x) + \Cs_t u_0'(x)$ solves the linear wave (Klein--Gordon) equation
\begin{equation}\label{eq:periodic w}
\partial_{tt} w(x, t) + (m^2 - \Delta) w(x,t) = 0.
\end{equation}
This leaves $v$ to solve the coupled equation
\begin{equation}
\partial_{tt} v(x, t) + (m^2 - \Delta) v(x,t) = -\, \wick{(v+w)^3}
\end{equation}
with zero initial data.
We will see that $v$ has one degree higher regularity than $w$,
and its growth is controlled by $w$.

The almost sure wellposedness of \eqref{eq:nlw bounded} with a more general nonlinearity
was proved by Oh and Thomann \cite[Theorem~1.5]{oh_invariant_2020}.
It was also stated without proof by Bourgain
in a lecture note two decades earlier \cite[Theorem~111]{bourgain_nonlinear_1999}.
The argument presented below replaces the more specific Fourier restriction norm
by a general Besov norm, and includes the details on convergence of solutions.

\subsection{Linear part}

It is a basic property of the wave equation that all wave packets travel at a fixed speed.
The propagators are then also bounded in weighted spaces
since the weight does not change too much within a ball.
The finite speed of propagation applies to the nonlinear equation \eqref{eq:mild solution} as well,
as we show in Lemma~\ref{thm:nonlinear speed of sound}.

\begin{lemma}[Finite speed of propagation, linear part]\label{thm:speed of sound}
If the initial data $(u_0, u_0')$ and $(\tilde u_0, \tilde u_0')$ coincide on $B(0, R)$, $R > 0$,
then the corresponding linear wave equation solutions $w(t)$ and $\tilde w(t)$
coincide on $B(0, R - \abs t)$ up to times $\abs t < R$.
Moreover, this result holds also in the infinite volume $\R^2$.
\end{lemma}
\begin{proof}
\cite[Section~12.1.2]{evans_partial_2010}.
\end{proof}

\begin{lemma}[Boundedness of linear propagators]\label{thm:linear operator bounds}
For $s \in \R$, $\varepsilon > 0$,
and $f \in H^{s}(\rho)$ we have
\[
\begin{gathered}
\norm{\Cc_t f}_{H^s (\rho)} \lesssim (1 + \abs t)^{1+\alpha/2} \norm{f}_{H^{s+\varepsilon}(\rho)},\\
\norm{\Cs_t f}_{H^s (\rho)} \lesssim (1 + \abs t)^{1+\alpha/2} \norm{f}_{H^{s-1+\varepsilon}(\rho)},
\end{gathered}
\]
where $\alpha$ is the parameter of $\rho$.
Fixing $T$, we get uniform bounds in $\abs t \leq T$.
\end{lemma}
\begin{proof}
Let us consider $\Cc_t$.
For $\Cs_t$ the proof is identical,
except that we gain a derivative.
By going to the fractional Sobolev space with Theorem~\ref{thm:besov sobolev}
(which costs $\varepsilon$ derivatives),
we can assume $s = 0$ since $\Cc_t$ commutes with $\inorm\nabla^s$.
By Lemma~\ref{thm:speed of sound} and the decomposition $w(t) = \Cc_t u_0 + \Cs_t u_0'$,
the finite speed of propagation also applies to $\Cc_t$ and $\Cs_t$ individually.

Let $P(n)$ be the decomposition of $\R^2$ into unit rectangles
as in Lemma~\ref{lemma:wick-ordering-rem},
and $\chi_n$ the sharp indicator function of $P(n)$.
Given $t \in \R$, let $\tilde\chi_n$ be the sharp indicator of $P(n) + B(0, \abs t)$.
Then we have
\begin{equation}
\norm{\chi_n \rho \Cc_t f}_{L^2}^2
\leq \left[ \sup_{x \in P(n)} \rho(x)^2 \right] \norm{\chi_n \Cc_t f}_{L^2}^2
\leq \left[ \sup_{x \in P(n)} \rho(x)^2 \right] \norm{\Cc_t \tilde\chi_n f}_{L^2}^2.
\end{equation}
We see that $\Cc_t$ is bounded on $L^2$ with flat weight,
since it is a Fourier multiplier with bounded symbol.
Then we use the moderateness property $\rho(x) \leq \rho(x-y)\inv \rho(y)$
together with the estimate $\rho (x - y)\inv \lesssim (1 + \abs t)^{\alpha/2}$
that follows from $\tilde \chi_n$ vanishing outside $\abs{x-y} \lesssim 1 + \abs t$:
\begin{equation}
\sup_{x \in P(n)} \rho(x)^2 \int_{\R^2} \tilde\chi_n(y)^2 f(y)^2 \dy 
\leq (1 + \abs t)^\alpha \int_{\R^2} \tilde\chi_n(y)^2 \rho(y)^2 f(y)^2 \dy.
\end{equation}
Finally, it suffices to observe that any point of $\R^2$
supports order $(1 + \abs t)^2$ instances of $\tilde\chi_n$.
As we sum over $n$, we get
\begin{equation}
\begin{split}
\sum_n \norm{\chi_n \rho \Cc_t f}_{L^2}^2
&\lesssim \sum_n (1 + \abs t)^\alpha \norm{\tilde\chi_n \rho f}_{L^2}^2\\
&\lesssim (1 + \abs t)^{2+\alpha} \norm{\rho f}_{L^2}^2.
\qedhere
\end{split}
\end{equation}
\end{proof}

In probabilistic terms, the linear part looks almost like the coupled $\phi^4$ measure:
there is an invariant Gaussian free field part and a more regular term.
This stationarity property simplifies several proofs.

\begin{lemma}[Law of linear part]\label{thm:linear gff plus regular}
Let us sample $(u_0, u_0')$ from $\vec{\mu}$
and decompose $u_0 = Z_{L} + \psi_{L}$ as in Theorem~\ref{thm:phi42 weak limit}.
Then the linear part \eqref{eq:periodic w} can be written as
\[
w(\cdot, t) = \bigg[ \Cc_t Z_L + \Cs_t u_0' \bigg] + \Cc_t \psi_{L}.
\]
The law of the bracketed term is GFF for all $t \in \R$,
whereas $\Cc_t \psi \in H^{2-\varepsilon}(\rho^{1/2})$ almost surely.
\end{lemma}
\begin{proof}
The latter part follows from the boundedness of $\Cc_t$ on $H^s(\rho)$ shown above.
To prove the first part, we need to compute the covariance.
For any test functions $f$, $g$ we have
\begin{equation}
\begin{split}
&\E \bigg[ \dual{f, \Cc_t Z_L + \Cs_t u_0'} \dual{g, \Cc_t Z_L + \Cs_t u_0'} \bigg]\\
=\; &\E \bigg[ \dual{f, \Cc_t Z_L} \dual{g, \Cc_t Z_L} \bigg]
    + \E \bigg[ \dual{f, \Cs_t u_0'} \dual{g, \Cs_t u_0'} \bigg]
\end{split}
\end{equation}
by independence of $Z_L$ and $u_0'$.
Because $\Cc_t$ is a self-adjoint operator, the first term becomes
\begin{equation}
\E \bigg[ \dual{f, \Cc_t Z_L} \dual{g, \Cc_t Z_L} \bigg]
= \dual{\Cc_t f, \frac{\Cc_t g}{m^2-\Delta}}
= \dual{f, \frac{\cos((m^2-\Delta)^{1/2})^2}{m^2-\Delta} g}.
\end{equation}
For the second term we have white noise covariance instead:
\begin{equation}
\E \bigg[ \dual{f, \Cs_t u_0'} \dual{g, \Cs_t u_0'} \bigg]
= \dual{\Cs_t f, \Cs_t g}
= \dual{f, \frac{\sin((m^2-\Delta)^{1/2})^2}{m^2-\Delta} g}.
\end{equation}
Now the trigonometric identity $\sin^2 + \cos^2 = 1$ implies
\begin{equation}
\E \bigg[ \dual{f, \Cc_t Z_L + \Cs_t u_0'} \dual{g, \Cc_t Z_L + \Cs_t u_0'} \bigg]
= \dual{f, \frac{1}{m^2-\Delta} g}.
\qedhere
\end{equation}
\end{proof}

Not only the linear part but also its Wick powers are continuous in time.
This was shown by Oh, Okamoto, and Tzvetkov~\cite{oh_uniqueness_2024}
in the periodic case.
The result also yields a very good moment bound on $w_L$ and its Wick powers.

\begin{lemma}[Moment bounds for linear part]\label{thm:linear moment bounds}
There exists a version of $w_{L,N}$ such that each $\wick{w_{L,N}^j}$, $j \leq 3$,
belongs almost surely to
$C([0,T];\, \mathcal C^{-\varepsilon}(\rho))$ and satisfies the moment bound
\[
\sup_{L > 1, N \in \N \cup \{\infty\}}
\E \norm{\wick{w_{L,N}^j}}_{C([0,T];\, \mathcal C^{-\varepsilon}(\rho))}^p
\lesssim_p 1
\]
for any $1 \leq p < \infty$.
Moreover, for any finite $L$ we have
\[
\lim_{N \to \infty}
\E \norm{\wick{w_{L,N}^j} - \wick{w_L^j}}_{C([0,T];\, \mathcal C^{-\varepsilon}(\rho))}^p
= 0.
\]
\end{lemma}
\begin{proof}
We defer the proof of the first part to Appendix~\ref{sec:continuity}.
\cite[Proposition~1.1]{oh_uniqueness_2024} gives both results for a space equipped with flat weight.
The second claim then follows from it and Theorem~\ref{thm:besov periodic into polynomial}.
\end{proof}

We can now show that powers of the linear parts converge as $L \to \infty$.
We will use this result as we pass to the full space in Section~\ref{sec:globalization}.
This could be done by modifying the argument of Appendix~\ref{sec:continuity},
but an easier $L^p$-in-time bound is sufficient and follows from the stationarity.

\begin{lemma}[Convergence of linear parts]\label{thm:linear convergence to full space}
Let $1 \leq p < \infty$,
and let $w$ solve the linear wave equation
started from infinite-volume $(\phi^4_2, \text{white noise})$ initial data.
There is again the moment bound
\[
\E \norm{\wick{w^j}}_{C([0,T];\, \mathcal C^{-\varepsilon}(\rho))}^p
\lesssim_{p} 1.
\]
As $L \to \infty$ along the subsequence from Theorem~\ref{thm:phi42 weak limit},
$\wick{w^{i}_{L}}$ converges in probability to $\wick{w^{i}}$
in $L^{p}([0,T],\mathcal{C}^{-\varepsilon}(\rho^3))$.
\end{lemma}
\begin{proof}
Let us decompose the initial value $u_{0,L} = w_{\stat,L} + \psi_L$
as in Lemma~\ref{thm:linear moment bounds}.
The convergence of $\Cc_{t}\psi_{L}$ to $\Cc_{t}\psi$ in $H^{2-\varepsilon}(\rho)$
follows from continuity of $\Cc_{t}$ in $H^{2-\varepsilon}(\rho)$.
We need to show that $\wick{w^{i}_{\stat,L}} \to \wick{w_\stat^{i}}$ in $L^{p}([0,T],\mathcal{C}^{-\varepsilon}(\rho^3))$.
Then continuity of Besov product
from $\mathcal C^{-\varepsilon} \times H^{2-\varepsilon}$ to $\mathcal C^{-\varepsilon}$
implies convergence of $\wick{(w_{\stat,L} + \psi_L)^3}$.

We have that $w_{\stat,L} \to w_\stat$ in $C([0,T];\,H^{-\varepsilon}(\rho))$ by continuity of the linear operators.
Now with $f^{\delta}$ as in Lemma~\ref{thm:wick approx} we have
\begin{equation}
\begin{split}
\wick{w^{i}_{\stat,L}}-\wick{w^{i}_\stat}
=\;& [\wick{w^{i}_{\stat,L}}-f^{\delta}(w_{\stat,L})]\\
&\quad + [f^{\delta}(w_{\stat,L})-f^{\delta}(w_\stat)] + [f^{\delta}(w_\stat)-\wick{w^{i}_\stat}].
\end{split}
\end{equation}
The middle term goes to $0$ as $L \to \infty$ since $f^{\delta}$ is continuous
from $H^{-\varepsilon}(\rho)$ to $\mathcal{C}^{-\varepsilon}(\rho^{3})$,
and for the first and last term we have by stationarity
\begin{equation}
\begin{split}
&\mathrel{\phantom{=}} \E \bigg[ \int_{0}^{T} \norm{f^{\delta}(w_{\stat,L})-\wick{w^{i}_{\stat,L}}}
  _{\mathcal{C}^{-\varepsilon}(\rho^3)}^p \ds \bigg]\\
&= T \E \norm{f^{\delta}(w_{\stat,L}(0))-\wick{w^{i}_{\stat,L}(0)}}_{\mathcal{C}^{-\varepsilon}(\rho^3)}.
\end{split}
\end{equation}
By Lemma~\ref{thm:wick approx} this is a $\delta$-dependent constant independently of $L$,
so we may first pass $L \to \infty$ and then $\delta \to 0$.
\end{proof}

\subsection{Fixed-point iteration}

We now use the standard fixed-point argument to solve
\begin{equation}\label{eq:fixpoint nlw equation}
v(x, t) = -\int_0^t [\Cs_{t-s} \wick{(v+w)^3}](x) \ds.
\end{equation}
up to a short time.
We do the iteration in $C([0,\tau];\, H^{1-\varepsilon}(\Lambda_L))$.
The spatial weight must be flat because we need it to be the same
on both sides of the product estimates.

This argument is completely deterministic when the linear part $w$ from~\eqref{eq:periodic w} is fixed.
We control the growth of $v$ by assuming bounds on $w$;
these bounds will be verified by stochastic estimates in Section~\ref{sec:periodic stoch}.

\begin{lemma}[Boundedness]\label{thm:fixpoint boundedness}
Let $M = \max_{j=1,2,3} \norm{\wick{w^j}}_{L^\infty([0,1];\, \mathcal C^{-\varepsilon}(\rho))}$
and $\tau \leq 1$.
The operator
\[
(\mathcal F v)(x,t) \coloneqq -\int_0^t [\Cs_{t-s} \wick{(v+w)^3}](x) \ds
\]
maps a ball of radius $R$ into a ball of radius $C_L \tau M (1+R^3)$
in the space $C([0,\tau];\, H^{1-\varepsilon}(\Lambda_L))$.
\end{lemma}
\begin{proof}
We can commute the Fourier multiplier and apply Jensen's inequality in
\begin{equation}
\begin{split}
\norm{\mathcal F v}_{L^\infty_\tau H^{1-\varepsilon}(\Lambda_L)}
&= \sup_{0 \leq t \leq \tau} \left[
    \int_{\Lambda_L} \abs{ \inorm\nabla^{1-\varepsilon}
        \int_0^t \Cs_{t-s} \wick{(v+w)^3} \ds }^2 \dx
\right]^{1/2}\\
&\leq \tau^{1/2} \sup_{0 \leq t \leq \tau} \left[
    \int_{\Lambda_L} \int_0^t \abs{ \inorm\nabla^{1-\varepsilon}
        \Cs_{t-s} \wick{(v+w)^3} }^2 \ds \dx
\right]^{1/2}\\
&\lesssim \tau^{1/2} \left[
    \int_0^\tau \int_{\Lambda_L} \abs{ \inorm\nabla^{-\varepsilon}
        \wick{(v+w)^3} }^2 \ds \dx
\right]^{1/2}\\
&= \tau \norm{\wick{(v+w)^3}}_{L^\infty_\tau H^{-\varepsilon}(\Lambda_L)}.
\end{split}
\end{equation}
In the second-to-last step we used the increase in Besov regularity from $\Cs_t$;
on periodic space there is no $\varepsilon$-loss of differentiability of Lemma~\ref{thm:linear operator bounds}.

We can now expand the binomial power by triangle inequality
and estimate each term separately.
First, $\norm{\wick{w^3}}_{L^\infty H^{-\varepsilon}(\Lambda_L)} \lesssim L^c M$,
where $c$ depends on $\rho$ through Lemma~\ref{thm:besov periodic into polynomial}.
The second term is estimated as
\begin{equation}
\norm{\wick{w^2} v}_{L^\infty_\tau H^{-\varepsilon}(\Lambda_L)}
\lesssim \norm{\wick{w^2}}_{L^\infty_\tau \mathcal C^{-\varepsilon}(\Lambda_L)}
    \norm{v}_{L^\infty_\tau H^{2\varepsilon}(\Lambda_L)},
\end{equation}
and for the third one we use Theorem~\ref{thm:besov multiplication} twice:
\begin{equation}
\norm{w v^2}_{L^\infty_\tau H^{-\varepsilon}(\Lambda_L)}
\lesssim \norm{w}_{L^\infty_\tau \mathcal C^{-\varepsilon}(\Lambda_L)}
    \norm{v^2}_{L^\infty_\tau H^{2\varepsilon}(\Lambda_L)}
\lesssim L^c M \norm{v}_{L^\infty_\tau B^{3\varepsilon}_{4,4}(\Lambda_L)}^2.
\end{equation}
We also perform the a similar multiplicative estimate for the $v^3$ term.
Thus we have estimated
\begin{equation}
\begin{split}
&\norm{\wick{(v+w)^3}}_{L^\infty_\tau H^{-\varepsilon}(\Lambda_L)}\\
\lesssim\; &L^c M \left[
    1 + \norm{v}_{L^\infty_\tau \mathcal H^{2\varepsilon}(\Lambda_L)}
    + \norm{v}_{L^\infty_\tau B^{3\varepsilon}_{4,4}(\Lambda_L)}^2
    + \norm{v}_{L^\infty_\tau B^{3\varepsilon}_{6,6}(\Lambda_L)}^3
\right]\!,
\end{split}
\end{equation}
which yields the required bound after
embedding $H^{1-\varepsilon}$ into $B^{3\varepsilon}_{6,6}$
by Theorem~\ref{thm:besov embedding}.
With the estimates above, this is possible for $\varepsilon < 1/12$.

Continuity in time follows from
\begin{equation}
\begin{split}
&\mathrel{\phantom{=}} \mathcal F v(t+s) - \mathcal F v(t)\\
&= -\int_0^t [\Cs_{t+s-r} - \Cs_{t-r}] \wick{(v+w)^3} \diff r
    - \int_t^{t+s} \Cs_{t+s-r} \wick{(v+w)^3} \diff r,
\end{split}
\end{equation}
since $\Cs_{t+s-r} \to \Cs_{t-r}$ pointwise in $H^{-1-\varepsilon}(\Lambda_L)$ as $s \to 0$.
\end{proof}

\begin{lemma}[Contraction]\label{thm:fixpoint contraction}
In the setting of Lemma~\ref{thm:fixpoint boundedness},
we also have
\[
\norm{\mathcal Fv - \mathcal F\tilde v}_{C([0,\tau];\, H^{1-\varepsilon})}
\lesssim C_L \tau M (1+R^2)
    \norm{v - \tilde v}_{C([0,\tau];\, H^{1-\varepsilon})}.
\]
\end{lemma}
\begin{proof}
We can begin as in Lemma~\ref{thm:fixpoint boundedness} to get the upper bound
\begin{equation}
\tau
\norm{\wick{(v+w)^3} - \wick{(\tilde v + w)^3}}_{L^\infty_\tau H^{-\varepsilon}(\Lambda_L)}.
\end{equation}
When we again expand the binomials, we get three terms to estimate.
First,
\begin{equation}
\begin{split}
\norm{\wick{w^2} (v - \tilde v)}_{L^\infty_\tau H^{-\varepsilon}(\Lambda_L)}
&\lesssim \norm{\wick{w^2}}_{L^\infty_\tau \mathcal C^{-\varepsilon}(\Lambda_L)}
    \norm{v - \tilde v}_{L^\infty_\tau H^{2\varepsilon}(\Lambda_L)}\\
&\lesssim M \norm{v - \tilde v}_{L^\infty_\tau H^{1-\varepsilon}(\Lambda_L)}.
\end{split}
\end{equation}
In the second term we additionally need to expand
\begin{equation}
\begin{split}
\norm{v^2 - \tilde v^2}_{L^\infty_\tau H^{2\varepsilon}(\Lambda_L)}
&\lesssim \norm{v - \tilde v}_{L^\infty_\tau B^{3\varepsilon}_{4,4}(\Lambda_L)}
    \norm{v + \tilde v}_{L^\infty_\tau B^{3\varepsilon}_{4,4}(\Lambda_L)}\\
&\lesssim 2R \norm{v - \tilde v}_{L^\infty_\tau H^{1-\varepsilon}(\Lambda_L)}.
\end{split}
\end{equation}
In the final term, the corresponding expansion is
\begin{equation}
\begin{split}
&\norm{v^3 - \tilde v^3}_{L^\infty_\tau H^{2\varepsilon}}\\
=\; &\norm{(v - \tilde v)(v^2 + v\tilde v + \tilde v^2)}_{L^\infty_\tau H^{2\varepsilon}}\\
\lesssim\; &\norm{v - \tilde v}_{L^8_\tau B^{3\varepsilon}_{4,4}}
    \left(
        \norm{v}_{L^\infty_\tau B^{4\varepsilon}_{8,8}}^2
        + \norm{v}_{L^\infty_\tau B^{4\varepsilon}_{8,8}}
          \norm{\tilde v}_{L^\infty_\tau B^{4\varepsilon}_{8,8}}
        + \norm{\tilde v}_{L^\infty_\tau B^{4\varepsilon}_{8,8}}^2
    \right)\\
\lesssim\; &\norm{v - \tilde v}_{L^\infty_\tau H^{1-\varepsilon}(\Lambda_L)} R^2.
\end{split}
\end{equation}
All together, we get the claimed inequality for $\varepsilon$ small.
\end{proof}

\begin{theorem}\label{thm:fixpoint finished}
Assume that the moment bound in Lemma~\ref{thm:fixpoint boundedness} holds
with $M \geq 1$.
Then the nonlinear equation \eqref{eq:fixpoint nlw equation} has a unique solution
\[
v \in C([0,\tau];\, H^{1-\varepsilon}(\Lambda_L))
\]
of norm at most $M$,
where the time $\tau$ depends on both $M$ and the period $L$.
\end{theorem}
\begin{proof}
It only remains to choose $R$ and $\tau$ such that
\begin{equation}
\begin{cases}
C_L \tau M (1 + R^3) \leq R,\\
C_L \tau M (1 + R^2) \leq \frac 1 2.
\end{cases}
\end{equation}
We can select $R = M$ and $\tau = (4 C_L R^3)^{-1}$.
\end{proof}

\subsection{Globalization in time}\label{sec:periodic stoch}

The analysis of previous sections also applies to the truncated equation
\begin{equation}\label{eq:nlw truncated}
\begin{aligncases}
    \partial_{tt} u(x, t) + (m^2 - \Delta) u(x,t) &= -P_{N} \wick{P_N u^3},\\
    u(x, 0) &= P_N u_0(x),\\
    \partial_t u(x,0) &= P_N u_0'(x)
\end{aligncases}
\end{equation}
posed on $\Lambda_L \times \R_+$,
where $P_N$ truncates the Fourier series to terms with frequency at most $2^N$ in absolute value.%
\footnote{Recall that we define the Besov space with a full-space Fourier transform;
the Fourier transform is a linear combination of Dirac deltas in this case.}
The estimates are only changed by a constant factor since the projection operators
$P_N$ are bounded uniformly in $H^s(\Lambda_L)$ norm,
and the linear operators $\Cc_t$ and $\Cs_t$ do not change the Fourier support.

The reason to pass to \eqref{eq:nlw truncated} is that the state space
now consists of finitely many Fourier modes.
Because the equation is Hamiltonian,
a theorem of Liouville automatically implies invariance of the corresponding Gibbs measure.

\begin{definition}[Truncated Gibbs measure]\label{def:periodic truncated measure}
The measure $\vec{\mu}_{L,N}$ is supported on the subset of $\dataspace^{-\varepsilon}(\rho)$
that contains $2L$-periodic functions Fourier-truncated to ${[{-2^N},{2^N}]}^2$,
and is given by the density
\[
f(u, u')
= \exp\left( - \int_{\Lambda_L} \frac{\wick{P_{N} u(x)^4}}{4} \dx \right)
\]
with respect to the periodic, truncated (GFF, white noise) product measure.
\end{definition}

\begin{theorem}[Local-in-time invariance]\label{thm:liouville invariance}
Let us recall that we denote by $\dataspace^{-\varepsilon}(\rho)$
the space of pointwise-in-time solution pairs $H^{-\varepsilon}(\rho) \times H^{-1-\varepsilon}(\rho)$.
Then
\begin{itemize}
\item The flow $\flow_{L,N,t} \colon \dataspace^{-\varepsilon} \to \dataspace^{-\varepsilon}$
    of \eqref{eq:nlw truncated} is well-defined for $0 \leq t \leq \tau$,
    where $\tau$ depends on the data.
\item For any measurable set of initial data $A \subset \dataspace^{-\varepsilon}$
    such that the solution exists almost surely up to~$\tau$,
    we have $\vec{\mu}_{L,N}(P_N A) = \vec{\mu}_{L,N}(\flow_{L,N,t} P_N A)$ for all $0 \leq t \leq \tau$.
\end{itemize}
\end{theorem}
\begin{proof}
Existence of solutions was already discussed.
Equation~\eqref{eq:nlw truncated} can be written as the Hamiltonian system
\begin{equation}
\frac{\diff u}{\diff t} = \frac{\partial H(u, u')}{\partial u'}, \quad
\frac{\diff u'}{\diff t} = -\frac{\partial H(u, u')}{\partial u},
\end{equation}
with energy
\begin{equation}
H(u, u') = \int_{\Lambda_L} \frac{\wick{u(x)^4}}{4}
    + \frac{\abs{\nabla u(x)}^2 + m^2 u(x)^2 + u'(x)^2}{2} \dx.
\end{equation}
Hence the measure can be written as
\begin{equation}
\diff\vec{\mu}_{L,N}(u, u') = \exp(- H(u, u'))
    \prod_{k \in {[{-2^N},{2^N}]}^2} \diff\hat u(k) \diff\hat u'(k).
\end{equation}
The energy is constant under a Hamiltonian flow \cite[Section~15]{arnold_mathematical_1980},
whereas the Lebesgue measure of $A$ is preserved by Liouville's theorem \cite[Section~16]{arnold_mathematical_1980}.
\end{proof}

The globalization argument is motivated by \eqref{eq:nls}.
For $L^2$ solutions of \eqref{eq:nls},
the local time $\tau$ only depends on the $L^2$ norm of initial data,
which is conserved by the flow.
Then one can restart the flow from $u(\tau)$ and get a solution up to time $2\tau$,
and by induction to any time.

Such a conservation law is not expected for generic $H^s$ norms,
which motivated the probabilistic argument of Bourgain~\cite{bourgain_periodic_1994}.
By invariance of measure, random solutions at time $\tau$
are distributed identically to the initial data,
and hence we can control the solution on a high-probability set.

\begin{definition}[Bounded-moment set]\label{def:bounded-moment set}
Fix $T \geq 1$.
We define
\[
B_M \coloneqq \left\{ \norm{\vec{u}_0}_{\dataspace^{-\varepsilon}(\rho)} \leq M
    \text{ such that }
    \norm{\wick{w^j}}_{C([0,T];\, \mathcal C^{-\varepsilon}(\rho))} \leq M
    \text{ for } j=1,2,3 \right\},
\]
where $w$ is the $L$-periodic linear part \eqref{eq:periodic w} with data
$\vec{u}_0 \coloneqq (u_0, u_0')$.
\end{definition}

\begin{remark}
We fix the final time $T$ to an arbitrary positive value
in order to simplify the exposition.
We will extend the solution to all times $t \in [0, \infty)$
with some post-processing in Lemma~\ref{thm:global post-process}.
\end{remark}

Since the definition of $B_M$ matches the moment bound in Lemma~\ref{thm:fixpoint boundedness},
it follows that $\flow_{L,N,t} B_M$ is well-defined up to time $\tau(M)$
for all $N \in \N$.
We can then restart the flow, and overlap such local solution intervals:

\begin{lemma}[Growth bound]\label{thm:local growth bound}
Let us define
\[
\mathcal B_{M,L,N} \coloneqq
    B_M \cap \flow_{L,N,\tau/2}^{-1} B_M
    \cap \cdots \cap \flow_{L,N,\tau/2}^{-2m} B_M,
\]
where $m = T / \tau$ ($\tau$ dependent on $M$).
For all $(u_0, u_0') \in \mathcal B_{M,L,N}$,
there exists a unique solution $u_N \in C([0,T];\, H^{-\varepsilon}(\Lambda_L))$
to \eqref{eq:nlw truncated}, and
\begin{equation}\label{eq:local growth bound}
\norm{\wick{u_N^j}}_{C([0,T];\, H^{-\varepsilon}(\Lambda_L))} \lesssim (T M)^j
\end{equation}
for $j = 1,2,3$.
The constant is independent of $N$, $T$, and $L$.

Moreover, $u_N$ can be written as $u_N = w_N + v_N$,
where $w_N$ solves~\eqref{eq:periodic w} and satisfies the bounds
in Definition~\ref{def:bounded-moment set},
and $v_N \in C([0,T];\, H^{1-\varepsilon}(\Lambda_N))$ has norm at most $T M$.
\end{lemma}
\begin{proof}
Although the definition of $B_M$ uses the non-truncated linear equation,
we may pass to the truncated equation since $\Cc_t$ and $\Cs_t$ commute with $P_N$.

By construction, a local solution $u_N^{(k)} = w_N^{(k)} + v_N^{(k)}$ exists on each interval
$[k\tau/2, (k+2)\tau/2]$.
As the intervals overlap and each local solution is continuous and unique,
the global solution has the same properties.
In the decomposition, the bound on $w$ and its Wick powers
follows from Lemma~\ref{thm:linear moment bounds}.
We extend $v_N^{(k)}$ to all times by the mild solution formula
\begin{equation}
v_N(t) \coloneqq -\int_0^t \Cs_{t-s} \wick{u_N(s)^3} \ds.
\end{equation}
Thanks to the regularizing effect of $\Cs_{t-s}$, it satisfies
\begin{equation}
\norm{v_N(t)}_{H^{1-\varepsilon}(\Lambda_L)}
\lesssim \int_0^t \norm{\wick{u_N(s)^3}}_{H^{-\varepsilon}(\Lambda_L)} \ds
\lesssim T M.
\end{equation}
It thus remains to verify~\eqref{eq:local growth bound}.

For $j=1$ the claim follows immediately from
\begin{equation}
\norm{v_N}_{L^\infty([0,T], H^{1-\varepsilon}(\Lambda_L))}
+ \norm{w_N}_{L^\infty([0,T];\, H^{-\varepsilon}(\Lambda_L))}
\lesssim T M + M.
\end{equation}
For $j=2$ we are to estimate
\begin{equation}
\norm{\wick{w_N^2}}_{L^\infty H^{-\varepsilon}}
+ 2\norm{v_N w_N}_{L^\infty H^{-\varepsilon}}
+ \norm{v_N^2}_{L^\infty H^{-\varepsilon}}.
\end{equation}
Here the only relevant difference is estimating
\begin{equation}
\norm{v_N w_N}_{L^\infty H^{-\varepsilon}}
\lesssim \norm{v_N}_{L^\infty H^{2\varepsilon}}
    \norm{w_N}_{L^\infty \mathcal C^{-\varepsilon}}
\end{equation}
with Besov multiplication and Hölder.
Thanks to regularity of $v$, we have
\begin{equation}
\norm{v_N^2}_{L^\infty H^{2\varepsilon}}
\lesssim \norm{v_N}_{L^\infty H^{1-\varepsilon}}^2
\leq (T M)^2.
\end{equation}
The case $j=3$ follows similarly.
\end{proof}

Moreover, this set of initial data has high probability.
Here we use the finite-dimensional invariance to bound the probabilities.

\begin{lemma}[Data has high probability]\label{thm:local high probability data}
Given $k \in \N$, there exists $M_k$ such that
$\vec{\mu}_{L,N}(\mathcal B_{M_k,L,N}) \geq 1 - 2^{-k}$.
The value of $M_k$ depends on $L$ and $T$ but not $N$.
\end{lemma}
\begin{proof}
We may first use the triangle inequality and union bound to estimate
\begin{equation}
\begin{split}
&\Prob\left( \max_{\substack{j=1,2,3\\ k=0,\ldots,m}}
    \norm{\wick{w_N^j}}_{C([k\tau,k\tau+1];\, \mathcal C^{-\varepsilon}(\rho))} > M \right)\\
\leq\; &\sum_{j=1}^3 \sum_{k=0}^m \Prob \left(
    \norm{\wick{w_N^j}}_{C([k\tau,k\tau+1];\, \mathcal C^{-\varepsilon}(\rho))} > M \right).
\end{split}
\end{equation}
The pointwise-in-time norms
$\max_k \norm{(u(k\tau), \partial_t u(k\tau))}_{\dataspace^{-\varepsilon}(\rho)}$
are bounded with the same argument.
Then $\Prob((\mathcal B_{M_k,L,N})^c)$ is bounded from above by
\begin{equation}
\begin{split}
&\mathrel{\phantom{=}} \sum_{k=0}^m
    \frac{\E \norm{\wick{w_N^j}}_{C([k\tau,k\tau+1];\, \mathcal C^{-\varepsilon}(\rho))}^p
        + \E \norm{\vec u(k\tau)}_{\dataspace^{-\varepsilon}(\rho)}^p}{M^p}\\
&\lesssim m
    \frac{\E \norm{\wick{w_N^j}}_{C([0,1];\, \mathcal C^{-\varepsilon}(\rho))}^p
        + \E \norm{\vec u_0}_{\dataspace^{-\varepsilon}(\rho)}^p}{M^p}.
\end{split}
\end{equation}
The expectations are bounded by Section~\ref{sec:SQ} and Lemma~\ref{thm:linear moment bounds}
for any large $p$;
this estimate is uniform in $N$.
Now we substitute $m = T / \tau$ and $\tau = C_L M^{-3}$ from Theorem~\ref{thm:fixpoint finished}.
To finish the proof, we can choose e.g.\ $p=6$ to get the final estimate
\begin{equation}
\Prob\left( w_N \notin \mathcal B_{M_k,L,N} \right)
\leq C_L T M^{-3},
\end{equation}
which implies that the claim holds when
$M_k = C_L (2^k T)^{1/3}$.
\end{proof}

\subsection{Invariance of non-truncated measure}

Let us use Lemma~\ref{thm:local high probability data} to rename the sets of initial data
defined above.
We can then take a limit of these sets and get a high-probability set of initial data
with respect to the untruncated measure $\vec{\mu}_L$ defined in Theorem~\ref{thm:SQ}.
We follow here the argument of Burq and Tzvetkov \cite[Section 6]{burq_invariant_2007}.

\begin{definition}[High-probability set of data]\label{def:periodic data}
We define the set $\mathcal D_{k,L,N}$ to equal $\mathcal B_{M_k,L,N}$,
where $M_k$ is chosen with Lemma~\ref{thm:local high probability data}
such that $\vec{\mu}_{L,N}(\mathcal D_{k,L,N}) \geq 1 - 2^{-k}$.
\end{definition}

\begin{definition}[Limiting set of initial data]\label{def:limiting initial data}
We define $\mathcal D_{k,L} \subset \dataspace^{-\varepsilon}(\rho)$
as the set of limits $(u_0, u_0')$ of all sequences
$((u_{0,N_m}, u_{0,N_m}') \in \mathcal D_{k,L,N_m})_{m \in \N}$
that have $N_m \to \infty$ and converge in $\dataspace^{-\varepsilon}(\rho)$.
\end{definition}

\begin{lemma}[Total variation convergence]
We have
\[
\lim_{N \to \infty} \sup_A \abs{\vec\mu_L(A) - \vec\mu_{L,N}(A)} = 0,
\]
where the supremum is taken over all
measurable subsets of $\dataspace^{-\varepsilon}(\Lambda_L)$.
\end{lemma}
\begin{proof}
It suffices to consider the measure componentwise.
See e.g.~\cite[Remark~3]{barashkov_variational_2020} for the result on $\mu_L$.
\end{proof}

\begin{theorem}[Estimate for $\mathcal D_{k,L}$]\label{thm:limiting initial data}
We have $\vec{\mu}_L(\mathcal D_{k,L}) \geq 1 - 2^{-k}$.
\end{theorem}
\begin{proof}
It follows from the definition that
\begin{equation}
\limsup_{N \to \infty} \mathcal D_{k,L,N} \subset \mathcal D_{k,L},
\end{equation}
and then Fatou's lemma implies
\begin{equation}
\begin{split}
\vec{\mu}_{L}(\mathcal D_{k,L})
&\geq \vec{\mu}_{L}\left( \limsup_{N \to \infty} \mathcal D_{k,L,N} \right)\\
&\geq \limsup_{N \to \infty} \vec{\mu}_{L}\left( \mathcal D_{k,L,N} \right)\\
&= \limsup_{N \to \infty} \vec{\mu}_{L,N} \left( \mathcal D_{k,L,N}\right)\\
&\geq 1 - 2^{-k}.
\end{split}
\end{equation}
Here the equality holds by the total variation convergence.
\end{proof}

To show invariance of the limiting measure as $N \to \infty$,
we need to approximate full solutions by Fourier-truncated solutions.
The next lemma gives convergence in a qualitative sense.
It depends on pointwise bounds that follow from Fourier projections in Besov spaces.
For them we need to drop the regularity of our target space by $\varepsilon$.
Again, this change is irrelevant since $\varepsilon$ is arbitrarily small.

\begin{theorem}[Limit solves NLW]\label{thm:local limit solves nlw}
For almost all initial data $(u_0, u_0') \in \mathcal D_{k,L}$,
equation~\eqref{eq:nlw bounded} has a unique mild solution $u$ up to time $T$,
satisfying the moment bound in Definition~\ref{def:bounded-moment set} with $M = M_k$.
Moreover if $u_m$ are the solutions to~\eqref{eq:nlw truncated} with data $(u_{0,N_m}, u_{0,N_m}')$
from the approximating sequence,
then $u_m \to u$ in the space $C([0,T];\, H^{-2\varepsilon})$.

Consequently, \eqref{eq:nlw bounded} has a unique mild solution for $\mu_L$-almost all data.
We then denote the flow of \eqref{eq:nlw bounded} by $\flow_{L,t}$.
\end{theorem}
\begin{proof}
As Theorem~\ref{thm:fixpoint finished} holds in the untruncated case,
the solution $u$ with limiting initial data $(u_0, u_0')$ exists at least up to a short time.
We will extend it to $T$ by a continuity argument.

As the linear propagators
$(\Cc_t, \Cs_t) \colon \dataspace^{-\varepsilon}(\Lambda_L) \to H^{-\varepsilon}(\Lambda_L)$
are continuous, the linear part converges for all times:
\begin{equation}
w(t) = \Cc_t u_0 + \Cs_t u_0'
= \lim_{m \to \infty} \left( \Cc_t u_{0,N_m} + \Cs_t u_{0,N_m}' \right).
\end{equation}
Let us then consider the integral part in \eqref{eq:mild solution}.
We need to show that
\begin{equation}
\lim_{m \to \infty} \int_0^t \Cs_{t-s} \left( P_{N_m} \wick{P_{N_m} u_m^3} - \wick{u^{3}} \right)\!(x,s) \ds = 0
\end{equation}
for all $0 \leq t \leq T$.
By Lemma~\ref{thm:local invariance duhamel} the integral is bounded
in $H^{1-2\varepsilon}(\Lambda)$ norm by
\begin{equation}\label{eq:limit solves nlw duhamel}
\left( \max_{j = 1, 2, 3} \norm{\wick{w^j} - \wick{w_{N_m}^j}}
    _{L^\infty([0,T];\, \mathcal C^{-2\varepsilon}(\Lambda_L))} + 2^{-\varepsilon N_m} \right)\!
M_k^3 \exp (C M_k^2),
\end{equation}
once we have shown the moment bound
\begin{equation}\label{eq:limit solves nlw moment}
\norm{\wick{w^j}}_{C([0,T];\, \mathcal C^{-\varepsilon}(\rho))} \leq M_k.
\end{equation}

Since $t \mapsto \wick{w(t)^j}$ only depends on the initial data $\vec{u}_0$,
let us introduce the notations $F^{j}(\vec{u}_0) \coloneqq \wick{w^j}$
and $F^{j,N}(\vec{u}_0) \coloneqq \wick{w_N^j}$.
These functions are measurable as limits of the continuous
approximations from Lemma~\ref{thm:wick approx}.

By the convergence in expectation shown in Lemma~\ref{thm:linear moment bounds}
and changing the probability space with Skorokhod's theorem (Lemma~\ref{thm:skorokhod}),
we have
\begin{equation}
\lim_{N \to \infty} \norm{F^{j,N}(\vec{u}_0) - F^{j}(\vec{u}_0)}
    _{C([0,T];\, \mathcal C^{-\varepsilon}(\Lambda_{L}))} = 0
\end{equation}
for $\mu_L$-almost every $\vec{u}_0$.
Thus by Egorov's theorem there exists a set $A^{1}_{\delta}$
such that $\vec{\mu}_{L}((A^{1}_{\delta})^{c}) \leq \delta$ and 
\begin{equation}
\lim_{N \to \infty} \sup_{u\in A^{1}_{\delta}}
    \norm{F^{j,N}(\vec{u}_0) - F^{j}(\vec{u}_0)}
        _{C([0,T];\, \mathcal C^{-\varepsilon}(\Lambda_{L}))} = 0.
\end{equation}
Moreover by Lusin's theorem we can find $A_{\delta}^{2}$ such that
$\vec{\mu}_{L}((A^{2}_{\delta})^{c}) \leq \delta$ and $F^{j}$ is continuous on $A^{2}_{\delta}$.

Let us then set $A_{\delta}=A^{1}_{\delta} \cap A^{2}_{\delta}$.
If $\vec{u}_{0,N_m} \in A_\delta \cap \mathcal D_{k,L,N_m}$ is a sequence
converging to $u_0 \in A_\delta \cap \mathcal D_{k,L}$, then
\begin{equation}
\begin{split}
&\mathrel{\phantom{=}} \lim_{m \to \infty}
    \norm{F^{j,N_m}(\vec{u}_{0,N_m}) - F^{j}(\vec{u}_0)}
        _{C([0,T];\, \mathcal C^{-\varepsilon}(\Lambda_{L}))}\\
&\leq \lim_{m \to \infty} \sup_{\vec{u}_{0} \in A_{\delta}}
    \norm{F^{j,N_m}(\vec{u}_0)- F^{j}(\vec{u}_0)}
        _{C([0,T];\, \mathcal C^{-\varepsilon}(\Lambda_{L}))}\\
&\qquad + \lim_{m \to \infty} \norm{F^{j}(\vec{u}_{0,N_m}) - F^{j}(\vec{u}_0)}
            _{C([0,T];\, \mathcal C^{-\varepsilon}(\Lambda_{L}))}\\
&= 0.
\end{split}
\end{equation}
Hence on this subset of $\mathcal D_{k,L}$ we can approximate $\wick{w^j}$ by $\wick{w_N^j}$.
Combined with the definition of $\mathcal D_{k,L,N_m}$ this implies~\eqref{eq:limit solves nlw moment},
and the prefactor in~\eqref{eq:limit solves nlw duhamel} vanishes as $m \to \infty$.
By convergence in total variation we have
\begin{equation}
\begin{split}
\vec{\mu}_L(\limsup_{N \to \infty}(A_{\delta}\cap \mathcal{D}_{k,L,N}))
&\geq \limsup \vec{\mu}_{L}(A_{\delta}\cap \mathcal{D}_{k,L,N})\\
&\geq 1 - 2^{-k} - 2\delta.
\end{split}
\end{equation}
As we set
\begin{equation}
\mathcal{Q} = \bigcup_{k=1}^{\infty}\bigcup_{\delta>0}
    \limsup_{N\to \infty}(A_{\delta}\cap \mathcal{D}_{k,L,N}),
\end{equation}
we have that $\vec{\mu}_{L}(\mathcal{Q})=1$ and on $\mathcal{Q}$
there is a unique solution to~\eqref{eq:nlw bounded}.
\end{proof}

We can then proceed to invariance of the measure under the flow just found.
The next lemma shows that it is enough to show that $\vec{\mu}_L \circ \flow_{L,t}$ and $\vec{\mu}_L$
coincide when tested against a nice class of test functions.
We then only need to apply pointwise bounds for the flow in a high-probability set.

We will further advance this strategy in Lemma~\ref{thm:global reduction}.
This technique of adapting the test functions to the specific model is very common;
see the book of Ethier and Kurtz \cite[Section~3.4]{ethier_markov_1986}.

\begin{lemma}[Test functions]\label{thm:test functions}
Let $\mathcal F$ be the set of bounded Lipschitz functions
$\varphi \colon \dataspace^{-2\varepsilon}(\Lambda_L) \to \R$.
Let $\mu_1$ and $\mu_2$ be Borel probability measures on $\dataspace^{-2\varepsilon}(\Lambda_L)$.
If
\[
\int \varphi(f) \diff\mu_1(f) = \int \varphi(f) \diff\mu_2(f)
\]
for all $\varphi \in \mathcal F$, then $\mu_1 = \mu_2$.
\end{lemma}
\begin{proof}
It suffices to show that $\mathcal F$ separates points in the sense of~\cite{ethier_markov_1986}.
The claim then follows from~\cite[Theorem~3.4.5]{ethier_markov_1986}.

Fix two distinct elements $(f, f')$ and $(g, g')$ in $\dataspace^{-2\varepsilon}(\Lambda_L)$.
By general theory of distributions,
there exist $\alpha, \beta \in C^\infty(\Lambda_L)$ such that
$\dual{\alpha, f-g} \neq 0$ or $\dual{\beta, f'-g'} \neq 0$.
We then define the bounded functions
\begin{equation}
\eta_1(f, f') \coloneqq \arctan(\dual{\alpha, f}), \quad
\eta_2(f, f') \coloneqq \arctan(\dual{\beta, f'}).
\end{equation}
They are Lipschitz continuous since
\begin{equation}
\abs{\arctan(\dual{\beta, f'}) - \arctan(\dual{\beta, g'})}
\lesssim \abs{\dual{\beta, f' - g'}}
\lesssim \norm{\beta}_{H^2} \norm{f' - g'}_{H^{-1-2\varepsilon}},
\end{equation}
and similarly for $\eta_1$ in $H^{-2\varepsilon}$.
Hence $\eta_1$ and $\eta_2$ belong to $\mathcal F$,
and by construction $\eta_i(f, f') \neq \eta_i(g, g')$ for at least one of $i = 1, 2$.
\end{proof}

\begin{theorem}[Invariance of finite-volume measure]\label{thm:local invariance}
We have $\vec\mu_L(\flow_{L,t} A) = \vec\mu_L(A)$ for all $t \in {[{0},{T}]}$.
\end{theorem}
\begin{proof}
We apply Lemma~\ref{thm:test functions} so that it suffices to show
\begin{equation}
\int f(\Phi_{L,t} \vec\varphi) \diff\vec\mu_{L}(\vec\varphi)
    - \int f(\vec\varphi) \diff\vec\mu_{L}(\vec\varphi) = 0
\end{equation}
for all bounded and Lipschitz continuous $f \colon \dataspace^{-2\varepsilon}(\Lambda_L) \to \R$.
We split the integrals over the sets $\mathcal{D}_{k,L} \cap \mathcal{Q}_{\delta,N}$
and $(\mathcal{D}_{k,L} \cap \mathcal{Q}_{\delta,N})^{c}$,
where we restrict the linear solution $w$ to
\[
\mathcal{Q}_{\delta,N} \coloneqq \left\{ (u_n,u_{n}') \colon
    \max_{j \in \{ 1,2,3 \}} \norm{\wick{(P_N w)^j} - \wick{w^j}}
        _{L^\infty([0,T];\, \mathcal C^{-\varepsilon}(\rho))}
    \leq \delta \right\}.
\]
By Lemma~\ref{thm:linear moment bounds} we have
$\vec\mu_{L}(\mathcal{Q}_{\delta,N}) \geq 1 - \delta$
for all $N$ large enough.
We can then estimate the residual contribution as
\begin{equation}
\abs{\int_{(\mathcal{D}_{k,L}\cap \mathcal{Q}_{\delta,N})^c} \hspace{-1em}
        [f(\Phi_{L,t} \vec\varphi) - f(\vec\varphi)] \diff\vec\mu_{L}(\vec\varphi) }
\leq 2\vec\mu_L((\mathcal{D}_{k,L} \cap \mathcal{Q}_{\delta,N})^c) \norm{f}_\infty.
\end{equation}
Let us then note that
\begin{equation}
\begin{split}
\int f(\vec\varphi) \diff\vec\mu_L(\vec\varphi)
&= \int f(\flow_{L,N,t} \vec\varphi) \diff\vec\mu_L(\vec\varphi)\\
    &\quad + \int f(\flow_{L,N,t} \vec\varphi) \diff\,[\vec\mu_{L,N}(\vec\varphi) - \vec\mu_L(\vec\varphi)]\\
    &\quad + \int [f(\vec\varphi) - f(\flow_{L,N,t} \vec\varphi)] \diff\vec\mu_{L,N}(\vec\varphi)\\
    &\quad + \int f(\vec\varphi) \diff\,[\vec\mu_L(\vec\varphi) - \vec\mu_{L,N}(\vec\varphi)].
\end{split}
\end{equation}
On the second and fourth lines we use boundedness of $f$ and the total variation convergence,
whereas the third line vanishes by invariance of the truncated flow.
Hence we can write
\begin{equation}\label{eq:lipschitz-d}
\begin{split}
&\mathrel{\phantom{=}} \lim_{N \to \infty}
    \abs{ \int f(\Phi_{L,t} \vec\varphi) - f(\vec\varphi) \diff\vec\mu_{L}(\vec\varphi) }\\
&\lesssim \lim_{N \to \infty} \int_{\mathcal{D}_{k,L}\cap \mathcal{Q}_{\delta,N}} \hspace{-2em}
        \abs{ f(\Phi_{L,t} \vec\varphi)-f(\Phi_{L,N,t} \vec\varphi) } \diff\vec\mu_{L}
    + 2\vec\mu_L((\mathcal{D}_{k,L} \cap \mathcal{Q}_{\delta,N})^c) \norm{f}_\infty\\
&\leq 
    \int_{\mathcal{D}_{k,L}\cap \mathcal{Q}_{\delta,N}} \hspace{-2em}
        (\operatorname{Lip}_f + 2\norm{f}_\infty)
         \lim_{N \to \infty} \norm{\vec\varphi - \Phi_{L,N,t} \vec\varphi }_{\dataspace^{-\varepsilon}} \diff\vec\mu_{L}\\
&\qquad + 2\vec\mu_L((\mathcal{D}_{k,L} \cap \mathcal{Q}_{\delta,N})^c) \norm{f}_\infty.
\end{split}
\end{equation}
It therefore suffices to bound the difference of flows in both components.

By the uniform bounds in Theorem~\ref{thm:limiting initial data},
we know that the full solution $u(t) = \Pi_1 \flow_{L,t}(\vec u_0)$
and the truncated solution $u_N(t) = \Pi_1 \flow_{L,N,t}(P_N \vec u_0)$
are well-defined for all $t \leq T$.
We split the pathwise difference $u(t) - u_N(t)$ again into linear and Duhamel parts
\begin{equation}\label{eq:local invariance pathwise}
w(t) - w_N(t)
+ \int_0^t \Cs_{t-s} [\wick{ u(s)^3 } - P_{N} \wick{ u_N(s)^3 }] \ds.
\end{equation}
For the linear part we use the bound
\begin{equation}\label{eq:local invariance linear}
\begin{split}
\norm{w(t) - w_N(t)}_{H^{-2\varepsilon}(\Lambda_L)}
&= \norm{P_{>N} w(t)}_{H^{-2\varepsilon}(\Lambda_L)}\\
&\lesssim 2^{-\varepsilon N} \norm{w(t)}_{H^{-\varepsilon}(\Lambda_L)}\\
&\leq 2^{-\varepsilon N} M_k
\end{split}
\end{equation}
coming from the definition of $\mathcal D_{k,L}$.
We separate the estimate for the Duhamel term as Lemma~\ref{thm:local invariance duhamel} below.
Together they give the bound
\begin{equation}\label{eq:local invariance u bounded}
\lim_{N \to \infty} \norm{u(t)-u_N(t)}_{H^{-2\varepsilon}}
\lesssim \delta M_k^3 \exp (C M_k^2).
\end{equation}

For the time derivative component $\partial_t u(t) = \Pi_2 \flow_{L,t}(\vec u_0)$,
we use Lemma~\ref{thm:local invariance derivative} to find
\begin{equation}\label{eq:local invariance deriv bounded}
\lim_{N \to \infty} \norm{\partial_t u(t) - \partial_t u_N(t)}_{H^{-1-2\varepsilon}(\Lambda_L)}
\lesssim \delta M_k^3 \exp(C M_k^2).
\end{equation}
Hence~\eqref{eq:lipschitz-d}
is bounded by~\eqref{eq:local invariance u bounded} and~\eqref{eq:local invariance deriv bounded}
and the measure of $(\mathcal D_{k,L} \cap \mathcal Q_{\delta,N})^c$.
We can now finish by passing first $\delta \to 0$ and then $k \to \infty$.
\end{proof}

The pointwise bounds used in the preceding two theorems are as follows:

\begin{lemma}[Fourier approximation, nonlinearity]
\label{thm:local invariance duhamel}
Let us denote
\[
H_{N} \coloneqq
    \max_{j = 1, 2, 3} \norm{\wick{w^j} - \wick{w_N^j}}
        _{L^\infty([0,T];\, \mathcal C^{-2\varepsilon}(\Lambda_L))}.
\]
When the initial data $(u_0, u_0') \in \mathcal D_{k,L}$,
the solutions $u$ and $u_N$ to~\eqref{eq:nlw bounded} and~\eqref{eq:nlw truncated} satisfy
\[
\bignorm{ \int_0^t \Cs_{t - s} \left[ \wick{u(s)^3} - P_N \wick{u_N(s)^3} \right] \ds}
    _{H^{1 - 2 \varepsilon}(\Lambda_L)}
\!\!\lesssim (H_{N} + 2^{-\varepsilon N}) M_k^3 \exp (C M_k^2)
\]
for all $0 \leq t \leq T$.
\end{lemma}
\begin{proof}
Let us write the left-hand side $\norm{v(t) - v_N(t)}_{H^{1-2\varepsilon}}$ as
\begin{equation}
\bignorm{\int_0^t \Cs_{t - s} \left[ \wick{u(s)^3} - \wick{u_N(s)^3} \right] \ds
    + \int_0^t \Cs_{t-s} P_{>N} \wick{u_N(s)^3} \ds}_{H^{1-2\varepsilon}}.
\end{equation}
The last term is bounded with boundedness of $\Cs_t$ and the Bernstein estimate:
\begin{equation}
\int_0^t \norm{P_{>N} \wick{u_N(s)^3}}_{H^{-2\varepsilon}} \ds
\lesssim 2^{-\varepsilon N} M_k^3.
\end{equation}
Similarly, we estimate the other terms as
\begin{equation}
\bignorm{\int_0^t \Cs_{t-s} [\wick{u(s)^3} - \wick{u_N(s)^3}] \ds}_{H^{1-2\varepsilon}}
\lesssim \int_0^t \norm{\wick{u(s)^3} - \wick{u_N(s)^3}}_{H^{-2\varepsilon}} \ds.
\end{equation}
We can rewrite the pointwise difference as
\begin{equation}\label{eq:local invariance duhamel pointwise}
\begin{split}
\wick{u^3} - \wick{u_N^3}
&= \sum_{j = 0}^3 \binom{3}{j} (\wick{w^j} v^{3 - j} - \wick{w^j_N} v_N^{3 - j})\\
&= \sum_{j = 0}^3 \binom{3}{j} \left[ (\wick{w^j} - \wick{w^j_N}) v_N^{3 - j}
    + \wick{w^j} (v^{3 - j} - v_N^{3 - j}) \right].
\end{split}
\end{equation}
When $j = 0$, the first summand vanishes,
and otherwise it is bounded with
\begin{equation}
\begin{split}    
&\mathrel{\phantom{=}} \int_0^t \norm{(\wick{w^j} - \wick{w^j_N}) v_N^{3 - j}}_{H^{-2\varepsilon}} \ds\\
&\lesssim \int_0^t \norm{\wick{w^j} - \wick{w^j_N}}_{\mathcal C^{-2\varepsilon}}
    \norm{v_N^{3 - j}}_{H^{3\varepsilon}} \ds\\
&\lesssim \norm{\wick{w^j} - \wick{w^j_N}}_{L^\infty([0,T];\, \mathcal C^{-2\varepsilon})}
    \norm{v_N}_{L^\infty([0,T];\, H^{1-2\varepsilon})}^{3 - j}\\
&\lesssim H_{N} M_k^{3-j}.
\end{split}
\end{equation}
The second summand vanishes when $j = 3$,
and for $j \leq 2$ we have
\begin{equation}
\begin{split}
\int_0^t \norm{\wick{w^j} (v^{3 - j} - v_N^{3 - j})}_{H^{-2\varepsilon}} \ds
&\leq \int_0^t \norm{\wick{w^j}}_{\mathcal C^{-2\varepsilon}}
    \norm{v^{3 - j} - v_N^{3 - j}}_{H^{3\varepsilon}} \ds\\
&\leq \int_0^t K_j \norm{\wick{w^j}}_{\mathcal C^{-2\varepsilon}}
    \norm{v - v_N}_{H^{1-2\varepsilon}} \ds,
\end{split}
\end{equation}
where
\begin{equation}
K_j = \begin{cases}
    2 \norm{v}_{H^{1-2\varepsilon}}^2 + 2 \norm{v_N}_{H^{1-2\varepsilon}}^2, & j = 0,\\
    \norm{v}_{H^{1-2\varepsilon}} + \norm{v_N}_{H^{1-2\varepsilon}}, & j = 1,\\
    1, & j = 2
\end{cases}
\end{equation}
is bounded by $C M_k^2$.
Hence we have shown
\begin{equation}
\begin{split}
\norm{v(t) - v_N(t)}_{H^{1-2\varepsilon}}
&\lesssim (2^{-\varepsilon N} + H_{N}) M_k^3\\
&\qquad + \int_0^t \sum_{j=0}^2 K_j \norm{\wick{w^j}}_{\mathcal C^{-2\varepsilon}}
        \norm{v - v_N}_{H^{1-2\varepsilon}} \ds,
\end{split}
\end{equation}
from which Grönwall's inequality yields
\begin{equation}\label{eq:local invariance duhamel finished}
\norm{v(t) - v_N(t)}_{H^{1-2\varepsilon}}
\lesssim (2^{-\varepsilon N} + H_{N}) M_k^3
    \exp\!\left( \int_0^t \sum_{j=0}^2 K_j \norm{\wick{w^j}}_{\mathcal C^{-2\varepsilon}} \ds \right).
\qedhere
\end{equation}
\end{proof}

\begin{lemma}[Fourier approximation, derivative]\label{thm:local invariance derivative}
Under the assumptions of Lemma~\ref{thm:local invariance duhamel}, we also have
\[
\lim_{N \to \infty} \norm{\partial_t u(t) - \partial_t u_N(t)}_{H^{-1-2\varepsilon}(\Lambda_L)}
\lesssim (H_N + 2^{-\varepsilon N}) M_k^6 \exp(C M_k^2).
\]
\end{lemma}
\begin{proof}
By the mild formulation we have
\begin{equation}
\begin{split}
\frac{u(t+s) - u(t)}{s}
&= \frac{\Cc_{t+s} - \Cc_t}{s} u_0 + \frac{\Cs_{t+s} - \Cs_t}{s} u_0'\\
&\quad - \int_0^t \frac{\Cs_{t+s-r} - \Cs_{t-r}}{s} \wick{u(r)^3} \diff r\\
&\quad - \frac 1 s \int_t^{t+s} \Cs_{t+s-r} \wick{u(r)^3} \diff r.
\end{split}
\end{equation}
The first two terms give a bounded operator
from $\dataspace^{-2\varepsilon}(\Lambda_L)$ to $H^{-1-2\varepsilon}(\Lambda_L)$
as $s \to 0$, as can be seen by considering the Fourier multiplier symbols.
Lemma~\ref{thm:local growth bound} implies that $\wick{u(r)^3}$ is continuous in $r$,
so the last two terms converge to
\begin{equation}
\int_0^t \Cc_{t-r} \wick{u(r)^3} \diff r + \wick{u(t)^3}.
\end{equation}
Hence
\begin{equation}
\begin{split}
&\mathrel{\phantom{=}} \norm{\partial_t [u(t) - u_N(t)]}_{H^{-1-2\varepsilon}(\Lambda_L)}\\
&\lesssim \norm{P_{>N} \vec u_0}_{\dataspace^{-2\varepsilon}(\Lambda_L)}
    + \int_0^T \norm{ \wick{u(r)^3} - \wick{u_N(r)^3} }_{H^{-1-2\varepsilon}(\Lambda_L)} \diff r\\
&\qquad + \norm{\wick{u(t)^3} - \wick{u_N(t)^3}}_{H^{-1-2\varepsilon}(\Lambda_L)}.
\end{split}
\end{equation}
Now the first term is estimated as in~\eqref{eq:local invariance linear}
and the second term is at most $C (H_N + 2^{-\varepsilon N}) M_k^3 \exp(C M_k^2)$
by a direct modification of Lemma~\ref{thm:local invariance duhamel}.
Finally, the last term is bounded by~\eqref{eq:local invariance duhamel pointwise}
and~\eqref{eq:local invariance duhamel finished}.
\end{proof}

\section{Global invariance of NLW}\label{sec:globalization}

Let us now move to \eqref{eq:nlw} over $\R^2 \times \R_+$.
Lemma~\ref{thm:speed of sound} states that at any given point the linear propagators only depend on the light cone,
and we show below in Lemma~\ref{thm:nonlinear speed of sound}
that the same holds for the nonlinear term.
We are thus able to go back to periodic solution theory. 
Within any bounded region of $\R^2 \times \R_+$,
it is impossible to distinguish between different $L$-periodized flows
as soon as $L$ is large enough.
We use this property to pass $L \to \infty$.

Let us first define what we mean by a solution to \eqref{eq:nlw}.
We still fix $T > 0$ throughout this section.
We pass to $\R_+$ in the concluding Lemma~\ref{thm:global post-process}.

\begin{definition}[Solution on $\R^2$]\label{def:Solution Whole Space}
Let $u_0, u'_0$ be random distributions with $\text{Law}(u_{0}, u'_{0}) = \vec{\mu}$,
and set $w(t) \coloneqq \Cc_t u_{0} + \Cs_t u'_{0}$.

A distribution $u$ solves \eqref{eq:nlw} on $\R^2$
with initial data $(u_{0}, u'_{0})$ if there exists $v \colon \R_+ \times \R^2 \to \R$
such that
\begin{itemize}
\item $u = w + v$,
\item for any spatial cutoff $\chi \in C^{\infty}_{c}(\R^2)$ we have
    $\chi v \in C([0,T] \times H^{1-\varepsilon})$, and
\item $\displaystyle v(t) = \int_0^t \Cs_{t-s}
    \left[ \sum_{j=1}^{3} \binom{3}{j} \wick{w^{3-j}(s)} \, v^{j}(s) \right] \ds$.
\end{itemize}
\end{definition}
Note that the right-hand side of the last point is well-defined
since $v$ is a function,
$\wick{w^j} \in C([0,T], \mathcal{C}^{-\varepsilon}(\rho))$
by Lemma~\ref{thm:linear convergence to full space},
and the kernel of $\Cs_{t-s}$ has bounded support by Lemma~\ref{thm:speed of sound}.  

\bigskip
Let us then introduce some notation used in this section.
As in Section~\ref{sec:periodic} we will denote by $(u_{L,0},u'_{L,0})$ initial data sampled from $\vec{\mu}_{L}$,
and by $u_{L}$ the corresponding solution to \eqref{eq:nlw bounded} constructed in
Lemma~\ref{thm:local limit solves nlw}, where also the flow $\flow_{L,t}$ is defined.
We will write $w_{L}(t) = \Cs_t u_{0,L} + \Cc_t u'_{0,L}$ as in \eqref{eq:periodic w}
and decompose $u_L = w_L + v_L$ as before.

We will also need some spatial cutoffs.
Given $R > 0$, we define two smooth, non-negative functions on $\R^2$:
\begin{itemize}
\item $\chi_{1}=1$ on $B(0,R)$ and $\chi_{1}=0$ outside of $B(0,2R)$, and
\item $\chi_{2}=1$ on $B(0,2R+T)$ and $\chi_{2}=0$ outside of $B(0, 3R+T)$.
\end{itemize}

We will first contruct the infinite-volume solution started from initial data sampled from $\mu$
in Section \ref{sec:construction infinite}.
In the process we will show that the unperiodic flow can be approximated by periodic solutions
started from periodic data.
Then we will show invariance in Section~\ref{sec:globalization measure}.

\subsection{Construction of solution in infinite volume}\label{sec:construction infinite}
Assuming that the period is large enough,
a periodic solution restricted to a compact domain $D$ and horizon time $T$
is independent of the periodization.
However, the initial data sampled from $\vec{\mu}_L$ still depends on the period~$L$.
In this section we quantify the convergence of solutions
and construct a limiting solution as $L \to \infty$.

Let us first construct a probabilistic solution set associated with compact $D \subset \R^2$.
This argument is analogous to Lemma~\ref{thm:local high probability data}, but with a twist:
by Theorem~\ref{thm:fixpoint finished}
the growth bound in $D$ is independent of the periodization,
but the local solution time $\tau$ is not.
However, at discrete times $\{ k\tau \}$ we can use the invariance of measure;
this property is qualitative and holds for all period lengths.

\begin{lemma}[Finite speed of propagation, nonlinear part]\label{thm:nonlinear speed of sound}
Fix $R > 0$ and let $L > 3R + T$.
Let $w_L$ be as above.
Assume that $v_{L} \in C([0,T];\, H^{1-\varepsilon}(\Lambda_{L}))$ solves
\[
v_{L}(t)=\int_{0}^{t} \sum_{i=1}^{3} \Cs_{t-s} \binom{3}{i} \wick{w_{L}(s)^{3-i}} \, v_{L}(s)^{i} \ds
\]
for all $t \in [0,T]$.
Let $\tilde v\in C([0,\tau];\, H^{1-\varepsilon}(\Lambda_{L}))$ for some $\tau \in (0, T]$ solve 
\begin{equation}\label{eq:speed of sound auxiliary}
\tilde v(t)=\int_{0}^{t} \sum_{i=1}^{3} \Cs_{t-s} \binom{3}{i} \chi_{2}\, \wick{w_{L}(s)^{3-i}} \, \tilde v(s)^{i} \ds .
\end{equation}
Then $v_{L}|_{B(0,R)}(t) = \tilde v|_{B(0,R)}(t)$ for all $t \leq \tau$.
\end{lemma}
\begin{proof}
It is sufficient to show that $(\tilde v-v_{L}) \I_{B(0,R+T-t)}=0$.
To see this we observe that by~Lemma \ref{thm:speed of sound}
\begin{equation}
\begin{split}
&\I_{B(0,R+T-t)}(\tilde v-v_{L})(t) \\
=\; &\I_{B(0,R+T-t)} \int_{0}^t \sum_{i=1}^{3} \binom{3}{i} \Cs_{t-s} \I_{B(0,R+T-s)}
    (\wick{w^{3-i}_{L}} [\tilde v^{i}-v_{L}^{i}])(s) \ds.
\end{split}
\end{equation}
Now we can use Lemma~\ref{thm:characteristic function}
to bound the above expression as 
\begin{equation}
    \begin{split}
&\norm{\I_{B(0,R+T-t)}(\tilde v-v_{L})(t)}_{H^{1/4}} \\
\leq\; & (R+T)^{1/2} \bignorm{\int_{0}^t \sum_{i=1}^{3} \binom{3}{i} \Cs_{t-s} \I_{B(0,R+T-s)}
    (\wick{w^{3-i}_{L}} [\tilde v^{i}-v_{L}^{i}])(s) \ds }_{H^{1-\varepsilon}}.
\end{split}
\end{equation}
Mimicking the proof of Lemma~\ref{thm:fixpoint contraction},
the norm can be estimated by
\begin{equation}
    \begin{split}
& \left(1+\norm{v_{L}}^{2}_{C([0,\tau];\, H^{1-\varepsilon}(\Lambda_{L}))}
    + \norm{\tilde v}^{2}_{C([0,\tau];\, H^{1-\varepsilon}(\Lambda_{L}))}\right) \\
\times & \sum_{i=1}^{3} \int_{0}^{t}
    \norm{\wick{w_{L}^{3-i}(s)} \I_{B(0,R+T-s)}}_{B_{14,\infty}^{-\varepsilon}(\Lambda_{L})}\\
    &\hspace{4em}\norm{\I_{B(0,R+T-s)} (\tilde v(s)-v_{L}(s))}_{H^{1/4}(\Lambda_{L})} \ds.
\end{split}
\end{equation}
We can again use Lemma~\ref{thm:characteristic function} to estimate
\begin{equation}
    \begin{split}
&\norm{\wick{w_{L}^{3-i}(s)} \I_{B(0,R+T-s)}}_{B_{14,\infty}^{-\varepsilon}(\Lambda_{L})}\\
\lesssim\;& \norm{\I_{B(0,R+T-s)}}_{B^{1/14}_{14,\infty}(\Lambda_{L})}
    \norm{\wick{w_{L}(s)^{3-i}}}_{\mathcal{C}^{-\varepsilon}(\Lambda_{L})} \\
\lesssim\;& (R+T)^{1/7} \norm{\wick{w_{L}(s)^{3-i}}}_{\mathcal{C}^{-\varepsilon}(\Lambda_{L})},
\end{split}
\end{equation}
which is integrable in time (for almost all $w_L$).
Hence we have shown
\begin{equation}
\begin{split}
&\mathrel{\phantom{=}} \norm{\I_{B(0,R+T-t)}(\tilde v(t) - v_{L}(t))}_{H^{1/4}}\\
&\lesssim \int_0^t \norm{\wick{w_{L}(s)^{3-i}}}_{\mathcal{C}^{-\varepsilon}(\Lambda_{L})}
    \norm{\I_{B(0,R+T-s)} (\tilde v(s)-v_{L}(s))}_{H^{1/4}(\Lambda_{L})} \ds,
\end{split}
\end{equation}
and Grönwall's inequality implies that the left-hand side is zero for all $t$.
\end{proof}

\begin{remark}\label{rem:restriction}
That a solution $\tilde v$ to \eqref{eq:speed of sound auxiliary} exists
follows from a straightforward fixed-point argument
for $t \leq \tau \simeq \min ( M^{-c}, R^{-c})$ where
\[
M= \sum_{i=1}^{3} \norm{\wick{w_{L}^{i}}}_{L^{2}([0,T];\, \mathcal{C}^{-\varepsilon}(B(0,3R+T)))}.
\]
Then 
$\norm{\tilde v}_{C([0,\tau];\, H^{-\varepsilon}(\Lambda_{L}))} \leq 2 M$.
The rest of the argument in Lemma~\ref{thm:nonlinear speed of sound} holds up to time $T$,
but below we will use the result only for a short time interval.
\end{remark}

\begin{lemma}[Bound for $v$ in a bounded domain]\label{thm:growth infinite volume}
For any compact $D\subset \mathbb{R}^{2}$ and $q\geq 1$,
there exists a constant $C_{D,T}$ independent of $L$ such that 
\[
\vec{\mu}_{L}(\norm{v_{L}}_{C([0,T];\,H^{1-\varepsilon}(D))} \geq M) \leq C_{D,T} M^{-{q}}.
\]
\end{lemma}
\begin{proof}
First, we have the following bound for any $\tau > 0$:
\begin{equation}
\norm{v_{L}}_{L^{\infty}({[{0},{T}]};\, H^{1-\varepsilon}(D))}
\leq \sup_{0\leq k \leq T/\tau}
    \norm{v_{L}}
        _{L^{\infty}({[{k\tau},{(k+1)\tau}]};\, H^{1-\varepsilon}(D))}.
\end{equation}
Let $R$ satisfy $D+B(0,T) \subset B(0,R)$.
If we assume that 
\begin{equation}\label{eq:global growth assumption}
\norm{\wick{(\flow_{\text{lin}} \flow_{L,k\tau} u_{0,L})^j}}_{C({[{0,1}]};\, H^{-2\varepsilon}(B(0,R)))}
\leq M
\end{equation}
for $j = 1, 2, 3$ and all $k \leq T / \tau$,
then the local solution theory and Lemma~\ref{thm:nonlinear speed of sound} imply
that a local nonlinear part $v_L$ exists and
\begin{equation}
\norm{v_{L}}_{L^\infty([k\tau,(k+1)\tau],H^{1-\varepsilon}(D))}
\leq 2M.
\end{equation}
This requires that $\tau \leq M^{-c}$ for $c\in \mathbb{N}$ sufficiently large.
From now on we fix $\tau=M^{-c}$.

The probability that $v_L$ can be constructed
is bounded from below by the probability of assumption \eqref{eq:global growth assumption} holding.
That in turn is bounded by
\begin{equation}
1- \Prob\left(
    \max_{\substack{j=1,2,3\\ k=0,\ldots,T/\tau}}
    \norm{\wick{(\flow_{\text{lin}} \flow_{L,k\tau} u_{0,L})^j}}
        _{C({[{0,1}]};\, H^{-2\varepsilon}(B(0,R)))}^2
    > M
\right).
\end{equation}
It is here that we use the invariance of $\vec{\mu}_{L}$ under $\flow_{L,t}$.
As in Lemma~\ref{thm:local high probability data},
we can then bound the probability from below by
\begin{equation}
1 - C T M^c \frac{\E \norm{\wick{(\flow_{\text{lin}} u_{0,L})^j}}
    _{C({[{0,1}]};\, H^{-2\varepsilon}(B(0,R)))}^p}
    {M^p}.
\end{equation}
The expectation is bounded by Lemma~\ref{thm:linear moment bounds} uniformly in $L$.
Again we conclude by choosing $p \geq c+q$.
\end{proof}

We will now construct the full solution $u$
by showing that $u_L$ is a Cauchy sequence.
Let us first show that the nonlinear parts $v_L$ form a Cauchy sequence
in a probabilistic set.

\begin{lemma}[Stability, $\phi^4$ component]\label{thm:global stability}
Assume that
\[
\max_{j=1,2,3} \norm{\wick{w_L^j}}_{C({[0,T]};\, \mathcal C^{-2\varepsilon}(\rho))} \leq M,
\quad \norm{v_{L}}_{L^\infty([0,T];\, H^{1-\varepsilon}(B(0,R+T)))} \leq M
\] 
hold for all $L \in \N \cup \{ \infty \}$. Set 
\[
H_{L,L'} \coloneqq \sup_{j \leq 3} \int_0^T \norm{(\wick{w_{L'}^j} - \wick{w_L^j})(s)}
    _{\mathcal C^{-2\varepsilon}(\rho)} \ds.
\]
Then for all $R$ there exists $C > 0$ (depending on $R$) such that
\begin{equation} \label{eq:bound-v}
\norm{\I_{B(0,R+T-t)}  (v_{L'}-v_{L})}_{H^{1/4}} \lesssim \exp(CM^3) H_{L,L'}.
\end{equation}
Consequently
\begin{align*}
	&\norm{\chi_1 (u_L - u_{L'})(t)}_{H^{-2\varepsilon}(\rho)}\\
\lesssim\; & \norm{\chi_2 [(u_{0,L}, u_{0,L}') - (u_{0,L'}, u_{0,L'}')]}_{\dataspace^{-\varepsilon}(\rho)}
    + \exp(CM^3) H_{L,L'},
\end{align*}
\end{lemma}
\begin{proof}
The second claim will follow from the first and properties of the linear propagators
(Lemmas~\ref{thm:speed of sound} and~\ref{thm:linear operator bounds}).
We thus estimate $\I_{B(0,R+T-t)} (v_{L'}-v_{L})$.
We can repeat the computations from Lemma~\ref{thm:nonlinear speed of sound} to obtain 
\begin{equation}
\begin{split}
&\norm{\I_{B(0,R+T-t)}(v_{L}-v_{L'})}_{H^{1/4}}\\
\leq\; & \bignorm{ \int_{0}^t \I_{B(0,R+T-t)} \Cs_{t-s} [\wick{u_{L'}(s)^3} - \wick{u_L(s)^3}] \ds}_{H^{1/4}}\\
\leq\; &\norm{\I_{B(0,R+T-t)}}_{B^{1/4}_{4,\infty}}
    \int_0^t \norm{\Cs_{t-s} \I_{B(0,R+T-s)}[ \wick{u_{L'}(s)^3} - \wick{u_L(s)^3}]}_{H^{1-3\varepsilon}} \ds\\
\lesssim\; &\int_0^t \norm{\I_{B(0,R+T-s)} [\wick{u_{L'}(s)^3} - \wick{u_L(s)^3}]}_{H^{-2\varepsilon}} \ds.
\end{split}
\end{equation}

We then perform the same manipulations as in Lemma~\ref{thm:local invariance duhamel},
only replacing the Fourier cutoff $N$ by the period length $L$
and multiplying everything by $\I_{B(0,R+T-s)}$.
Thanks to Lemma~\ref{thm:characteristic function} and the bounded support,
we can measure $\wick{w^{i}_{L}}$ in a weighted norm, such as in
\begin{equation}
\begin{split}
&\norm{\wick{w_{L}^2(s)} (v_{L'}-v_L)(s) \I_{B(0,R+T-s)}}_{H^{-2\varepsilon}}\\
\lesssim\; &\norm{\I_{B(0,R+T-s)} \wick{w_{L}^2(s)}}_{B^{-\varepsilon}_{14,\infty}}
    \norm{\I_{B(0,R+T-s)} (v_{L'}-v_L)(s)}_{H^{1/4}}\\
\lesssim \; & (R+T)^{\alpha}
   \norm{\I_{B(0,R+T-s)}}_{B^{1/14}_{14,\infty}} \norm{\wick{w_{L}^2(s)}}_{\mathcal C^{-\varepsilon}(\rho)}\\
\qquad& {} \times \norm{\I_{B(0,R+T-s)} (v_{L'}-v_L)(s)}_{H^{1/4}}.
\end{split}
\end{equation}
In the end, we have bounded
\begin{equation}
\begin{split}
&\norm{\I_{B(0,R+T-t)} (v_{L'}-v_L)(t)}_{H^{1/4}}\\
\lesssim_{R,T}\; & M^3\left[ H_{L,L'} + \int_0^t \sum_{j=1}^2 \norm{\wick{w^j}}_{\mathcal C^{-2\varepsilon}(\rho)}
    \norm{\I_{B(0,R+T-s)} (v_{L'} - v_L)(s)}_{H^{1/4}} \ds \right]\!,
\end{split}
\end{equation}
and again Grönwall gives
\begin{equation}
\norm{\I_{B(0,R+T-t)}(v_{L'}-v_L(t))}_{H^{1/4}}
\lesssim M^3 \exp(CM^3) H_{L,L'}.
\qedhere
\end{equation}
\end{proof}

We can then show that the limit of the Cauchy sequence really is a solution in our sense.

\begin{lemma}[Limit is a solution]\label{thm:construction solution infinite}
Let $(u_{0},u'_{0})$ be distributed according to $\vec{\mu}$.
Then there exists almost surely a solution to \eqref{eq:nlw} on
$\mathbb{R}^2$ with initial data $(u_{0},u'_{0})$
in the sense of Definition~\ref{def:Solution Whole Space}.
Furthermore for every compact $D \subset \R^2$ we have
\[
\lim_{M \to \infty} \vec{\mu}(\norm{v}_{C([0,T];\, H^{1-\varepsilon}(D))} \geq M) = 0.
\] 
\end{lemma}
\begin{proof}
Let $u_{L}$ be the solutions contructed in Section~\ref{sec:periodic} with initial data $(u_{0,L},u'_{0,L})$.
By Lemma~\ref{thm:skorokhod} we may put $u_{0,L}$ and $w_{L}$ for all $L$ in the same probability space $\tilde{\mathbb{P}}$,
and assume that $(u_{0,L}, u'_{0,L}) \to (u_{0},u'_{0})$ in $\dataspace^{-\varepsilon}(\rho)$ and
$\wick{w^{i}_{L}} \to \wick{w^{i}}$ in $L^1([0,T];\, H^{-\varepsilon}(\rho))$ almost surely.
We first need to show that $v_L$ has almost surely a unique limit as $L \to \infty$.
By Lusin's theorem we can find $A_{\delta}$ such that
$\vec{\mu}_{L}(A_{\delta}) \geq 1 - \delta$ and
$F^{j}(\vec{u}_0) \coloneqq \wick{w^j}$ is continuous on $A_{\delta}$.
Let us temporarily fix $R > 0$ and define a set where $v_L = u_L - w_L$ satisfies a good bound:
\begin{equation}
\begin{gathered}
\mathcal{D}_{L,M,R} \coloneqq \left\{ \norm{v_{L}}_{C([0,T];\, H^{1-\varepsilon}(B(0,R+T)))} \leq M \right\}\cap A_{\delta},
\text{ and }\\
\mathcal{D}_{\infty,M,R} \coloneqq \limsup_{L \to \infty} \mathcal{D}_{L,M,R}.
\end{gathered}
\end{equation}
Recall that by Lemma~\ref{thm:growth infinite volume}
we have $\tilde{\Prob}(\mathcal{D}_{L,M,R})\geq 1 - M^{-q} - \delta$,
and by Fatou also $\tilde{\Prob}(\mathcal{D}_{\infty,M,R}) \geq 1 - M^{-q} - \delta$.
We also observe that any $v \in \mathcal{D}_{\infty,M,R}$ is the limit of
a (random) subsequence $v_{L_{n}}$ such that $\norm{v_{L_{n}}}_{C([0,T];\, H^{1-\varepsilon}(B(0,R)))}\leq M$.

By Lemma~\ref{thm:global stability} we have 
\begin{equation}
\norm{\I_{B(0,R+T-t)} (v_{L_{n}} - v)(t)}_{H^{1/4}(\R^2)}
    \lesssim \exp(CM^3) H_{L,\infty},
\end{equation}
and by assumption $H_{L, \infty} \to 0$ almost surely as $L \to \infty$.
This shows that $\I_{B(0,R+T-t)} v_{L_{n}}(t)$ is a Cauchy sequence
also in the space $C([0,T];\,H^{1/4}(\R^2))$.
Let us denote its limit by $v^{R}$.
We need that show that for $R'>R$ we have $v^{R}|_{B(0,R)}=v^{R'}|_{B(0,R)}$.

Indeed note that $v^{R'}$ is the limit of another random subsequence $v_{L'_n}$,
where $v_{L'_n}$ satisfies
\begin{equation}
\norm{ v_{L'_{n}}(t)}_{C([0,T];\, H^{1-\varepsilon}(B(0,R'+T)))} \leq M.
\end{equation}
This implies that also $\norm{v_{L'_{n}}(t)}_{C([0,T];\, H^{1-\varepsilon}(B(0,R)))}\leq M$,
so Lemma~\ref{thm:global stability} gives 
\begin{equation}
\norm{\I_{B(0,R+T-t)} (v_{L_{n}}-v_{L'_{n}})(t)}_{H^{1/4}(\R^2)} \leq \exp(CM^3) H_{L_n,L'_{n}}.
\end{equation}
Again the right-hand side goes to $0$ as $n \to \infty$, which implies the claim.
Thus we can set 
$v(x,t) \coloneqq v^{R}(x,t)$
if $|x| \leq R+T-t$, and this is uniquely defined.

\bigskip\noindent%
To show that $u$ satisfies Definition~\ref{def:Solution Whole Space},
we need to prove that the above holds for any spatial cutoff;
that is, that we can pass $R \to \infty$.

In the above, we already passed $n \to \infty$ to take the infinite-volume limit.
As we then take the union of $\mathcal D_{\infty,M,R}$ over all $M > 0$,
we get a set of probability~$1$.
We can then intersect over $R \in \N$.

\bigskip\noindent%
Finally, we still need to show that
\begin{equation}
v(t)=\int_{0}^{t} \Cs_{t-s} \wick{u(s)^3} \ds.
\end{equation}
Equivalently, we can show that
\begin{equation}
\begin{split}
&\lim_{n \to \infty} \norm{ \chi_{1} (v(t) - v_{L_{n}}(t))}_{H^{1-2\varepsilon}(\rho)}\\
\lesssim\; &\lim_{n \to \infty} 
    \int_0^t \norm{\Cs_{s-t} [\chi_{2} (\wick{u(s)^3} - \wick{u_{L_{n}}(s)^3})]}_{H^{1-2\varepsilon}(\rho)} \ds
\end{split}
\end{equation}
vanishes as $L_{n} \to \infty$.
Here we again used Lemma~\ref{thm:speed of sound} to move $\chi_2$ into the integral. 
By Lemma~\ref{thm:linear operator bounds} we are left with estimating
\begin{equation}
\lim_{n \to \infty} \bignorm{\sum_{j=0}^3 \left[
    \wick{w(s)^j} \chi_{2} v(s)^{3-j} - \wick{w_{L_{n}}(s)^j} \chi_{2} v_{L_{n}}(s)^{3-j}
\right]}_{L^1([0,T];\, H^{-\varepsilon}(\rho))}.
\end{equation}
By assumption $w_L$ converges in $L^1([0,T];\, \mathcal C^{-\varepsilon}(\rho))$,
and by the first part of this proof $\chi_2 v_{L_n} \to \chi_2 v$ in $L^\infty H^{1/4}(\R^2)$.
As $v_{L_n}$ is a bounded sequence in $L^\infty H^{1-\varepsilon}(\rho)$,
it follows that $\chi_2 v_{L_{n}} \to \chi_2 v$ also in $L^\infty H^{1-2\varepsilon}(\rho)$.
Therefore the product can be estimated with Theorem~\ref{thm:besov multiplication}
and taken to the $n \to \infty$ limit.
This shows that $u$ satisfies the mild formulation.
We can then conclude by taking union over $\delta > 0$.
\end{proof}

\begin{remark}\label{rem:global flow}
A small modification of the proof of Lemma~\ref{thm:global stability}
gives that the solution $u$ with initial data $(u_{0},u'_{0})$ sampled from $\vec{\mu}$
in the sense of Definition~\ref{def:Solution Whole Space} is unique.
We will from now on denote its flow as $\Phi_{t}(u_{0},u'_{0})=u(t)$. 
\end{remark}

Since we are interested in the invariance of a product measure,
we also need to show that the sequence of $\partial_t u_L$ converges to $\partial_t u$.
This can be bootstrapped from the mild solution formula
as in Lemma~\ref{thm:local invariance derivative}.

\begin{lemma}[Stability, white noise component]\label{thm:global stability derivative}
Assuming $u$ and $u_L$ as in Lemma~\ref{thm:global stability}, we have almost surely
\begin{align*}\label{eq:stability-derivative} 
\norm{\chi_1 \partial_{t}[u_{L}(t) - u(t)]}_{H^{-1-2\varepsilon}(\rho)}
&\lesssim_T \norm{\chi_2 [(u_{0}, u'_{0}) - (u_{L,0}, u'_{L,0})]}_{\dataspace^{-\varepsilon}(\rho)}\\
    &\qquad+ \norm{\chi_1 [\wick{u(t)^3} - \wick{u_L(t)^3}]}_{H^{-2\varepsilon}}\\
    &\qquad+ H_{L,\infty} M^3 \exp(CM^3).
\end{align*}
\end{lemma}
\begin{proof}
By passing to the mild formulation we have
\begin{equation}
\begin{split}
\chi_1 \frac{u(t+s) - u(t)}{s}
&= \chi_1 \frac{\Cc_{t+s} - \Cc_t}{s} [\chi_2 u_0]
    + \chi_1 \frac{\Cs_{t+s} - \Cs_t}{s} [\chi_2 u_0']\\
&\quad- \frac{\chi_1}{s} \int_{0}^{t} (\Cs_{t+s-r} - \Cs_{t-r}) [\chi_2 \wick{u(r)^3}] \diff r\\
&\quad- \frac{\chi_1}{s} \int_{t}^{t+s} \hspace{-1em} \Cs_{t+s-r} [\chi_2 \wick{u(r)^3}] \diff r.
\end{split}
\end{equation}
The first two terms give a bounded linear operator 
from $\dataspace^{-\varepsilon}(\rho)$ to $H^{-1-2\varepsilon}(\rho)$ as $s \to 0$
by Lemma~\ref{thm:linear operator bounds}.
Since $\wick{u(r)^3}$ is continuous in $r$
by Lemma~\ref{thm:construction solution infinite} and Corollary~\ref{thm:continuity phi4},
the last two terms converge to
\begin{equation}
\chi_1 \int_0^t \Cc_{t-r} \chi_2 \wick{u(r)^3} \diff r
    + \chi_1 \chi_2 \wick{u(t)^3}.
\end{equation}
The same computations can be done for $u_{L}$.
Reusing the proof of Lemma~\ref{thm:global stability}, we then find
\begin{equation}
\int_0^t \norm{\Cc_{t-r} \chi_2 \wick{u(r)^3} - \wick{u_L(r)^3}}_{H^{-1-2\varepsilon}(\rho)} \diff r
\lesssim H_{L,\infty} M^3 \exp(CM^3).
\qedhere
\end{equation}
\end{proof}

\subsection{Proof of invariance}\label{sec:globalization measure}

As is well known, the Borel $\sigma$-algebra of $\R^2$ can be generated by just closed balls.
We will show an analogous result for the Borel $\sigma$-algebra of $\dataspace^{-2\varepsilon}(\rho)$:
the $\sigma$-algebra is generated by restrictions of distributions to compact domains.

\begin{theorem}[$\sigma$-algebra from compact-domain functions]
\label{thm:globalization_generator}
Let $s, s' \in \R$,
and let $\mathcal A^s$ be the family of Borel sets where inclusion only depends
on restrictions to compact domains:
\[
\mathcal A^s \coloneqq \left\{ A \subset H^s(\rho) \text{ Borel} \colon
    \exists \text{ compact } D \text{ s.t. }
    f \in A \Longleftrightarrow f|_D \in A \quad\forall f \in H^s(\rho)
\right\}.
\]
That is, $\I_A(f) = g_A(f|_D)$ for some $g_A \colon \dataspace^s(D) \to \{0,1\}$.
Then
\begin{enumerate}
\item the Borel $\sigma$-algebra of $H^s(\rho)$ is a sub-$\sigma$-algebra of $\sigma(\mathcal A^s)$;
\item the Borel $\sigma$-algebra of $H^s(\rho) \times H^{s'}(\rho)$
    is a sub-$\sigma$-algebra of $\sigma(\mathcal A^s) \times \sigma(\mathcal A^{s'})$.
\end{enumerate}
\end{theorem}
\begin{proof}
By the definition of the product $\sigma$-algebra,
it suffices to show the first claim for any $s \in \R$.
To do that, it is sufficient to construct the closed ball $\bar B = \bar B(f, R)$
for arbitrary $f \in H^s$ and $R > 0$.
By density, we can even assume $f \in C^\infty_c(\R^2)$.
We can write
\begin{equation}
\begin{split}
\bar B
&= \left\{ g \in H^s(\rho) \colon
    \int_{\R^2} \rho(x)^2 \abs{ (1 - \Delta)^{s/2} (f-g) }^2(x) \dx \leq R^2
\right\}\\
&= \left\{ g \in H^s(\rho) \colon
    \int_{\R^2} \rho(x)^2 \abs{ \int_{\R^2} K_s(x-y) (f-g)(y) \dy }^2 \dx \leq R^2
\right\}\\
&= \limsup_{N \to \infty} \bigg\{ g \in H^s(\rho) \colon \hspace{-0.5em}\sum_{\ell, m, n = 1}^N
    \int_{A_\ell} \!\rho(x)^2
        \left[ \int \chi_m(y) K_s(x-y) (f-g)(y) \dy \right]\\
&\hspace{15em}
        \left[ \int \chi_n(y) K_s(x-y) (f-g)(y) \dy \right] \dx \leq R^2
\bigg\}.
\end{split}
\end{equation}
Here we denote by $K_s$ the convolution kernel of $(1-\Delta)^{s/2}$,
by $(A_j)_{j \in \N}$ some measurable partitioning of $\R^2$, e.g.\ by unit squares,
and by $(\chi_j)$ a smooth partition of unity such that $\supp \chi_j \subset A_j + B(0,1)$.
For finite $N$, the set thus depends on $f$ and $g$ only inside
the compact set $\cup_{j=1}^N \overline{A_j + B(0,1)}$.

Since taking a $\limsup$ is a closed operation within the $\sigma$-algebra,
this proves that closed balls can be constructed from sets in $\mathcal A^s$.
\end{proof}

We now repeat the argument of Theorem~\ref{thm:local invariance}.
Thanks to the finite speed of propagation, we can assume our Lipschitz test functions
to be local in $\R^2$.

\begin{lemma}[Reduction to bounded domains]\label{thm:global reduction}
Let $\mathcal F$ be the set of bounded Lipschitz functions
$\varphi \colon \dataspace^{-2\varepsilon}(\rho) \to \R$
that depend only on the restriction of argument to some compact domain:
for any $\varphi \in \mathcal F$, there exists a compact $D \subset \R^2$
such that $\varphi(f) = \varphi(f|_D)$ for all $f \in \dataspace^{-2\varepsilon}(\rho)$.

Let $\mu_1$ and $\mu_2$ be Borel probability measures on $\dataspace^{-2\varepsilon}(\rho)$.
If
\[
\int \varphi(f) \diff\mu_1(f)
= \int \varphi(f) \diff\mu_2(f)
\]
for all $\varphi \in \mathcal F$, then $\mu_1 = \mu_2$.
\end{lemma}
\begin{proof}
We repeat the argument of Lemma~\ref{thm:test functions}.
Fix two distinct points $(f, f')$ and $(g, g')$ in $\dataspace^{-2\varepsilon}(\rho)$.
There again exist $\alpha, \beta \in C_c^\infty(\R^2)$
such that $\dual{\alpha, f - g} \neq 0$ or $\dual{\beta, f' - g'} \neq 0$;
note that these functions are compactly supported.
Then
\begin{equation}
\eta_1(f, f') \coloneqq \arctan(\dual{\alpha, f}), \quad
\eta_2(f, f') \coloneqq \arctan(\dual{\beta, f'})
\end{equation}
are bounded,
depend on their arguments only on $\supp \alpha \cup \supp \beta$,
and are Lipschitz continuous over the weighted spaces since
\begin{equation}
\abs{\arctan(\dual{\beta, f'}) - \arctan(\dual{\beta, g'})}
\lesssim \abs{\dual{\beta, f' - g'}}
\lesssim \norm{\beta}_{H^2(\R^2)} \norm{f' - g'}_{H^{-1-2\varepsilon}(\rho)}
\end{equation}
and similarly for $\eta_1$ in $H^{-2\varepsilon}(\rho)$.
\end{proof}

\begin{theorem}[Global invariance]\label{thm:global invariance}
We have $\vec{\mu} \circ \flow_t\inv = \vec{\mu}$ for all $0 \leq t \leq T$.
\end{theorem}
\begin{proof}
We know \emph{a priori} that the pushforward measure $\vec{\mu} \circ \flow_t\inv$ exists
as $\flow_t$ is a measurable map.
(It is well-defined by Remark~\ref{rem:global flow}.
By Theorem~\ref{thm:globalization_generator}, we only need to check restrictions to bounded domains.
There $\flow_t$ is almost surely defined as composition of small-time periodic flows.)

By the weak limit and finite-volume invariance, we also have that
for all bounded and continuous $f \colon \dataspace^{-2\varepsilon}(\rho) \to \R$,
\begin{equation}
\int_{\dataspace} \!f(\vec{u}_0) \diff\vec{\mu}
= \lim_{L \to \infty} \int_{\dataspace}
    f(\flow_{L,t} \vec{u}_{L,0}) \diff\vec{\mu}_L.
\end{equation}
Recall that the weak limit is unique along a fixed subsequence $L \to \infty$.
To show $\vec{\mu} = \vec{\mu} \circ \flow_t\inv$, we then only need to show that
\begin{equation}
\lim_{L \to \infty} \int_{\dataspace}
    f(\flow_{L,t} \vec{u}_{L,0}) \diff\vec{\mu}_L
= \int_{\dataspace} f(\flow_t \vec{u}_0) \diff\vec{\mu}.
\end{equation}
Lemma~\ref{thm:global reduction}
lets us assume that $f$ is Lipschitz in $\dataspace^{-2\varepsilon}$
and depends on the restriction of its arguments to some $B(0,R)$.
We can further pass to a common probability space by Lemma~\ref{thm:skorokhod}.

Let $\mathcal G$ be the set on which
all of the following hold:
\begin{equation}
\begin{gathered}
\norm{v_{L}}_{C([0,T];\, H^{1-\varepsilon}(B(0,R+T)))}\leq M,
\quad \norm{v}_{C([0,T];\, H^{1-\varepsilon}(B(0,R+T)))}\leq M,\\
\sum_{j=1}^{3} \norm{\wick{w^j}}_{C([0,T];\, H^{-\varepsilon}(\rho))}
+ \norm{\wick{w_{L}^j}}_{C([0,T];\, H^{-\varepsilon})(\rho)} \leq M.
\end{gathered}
\end{equation}
We suppress the dependency on $L$ and $M$ in the notation for simplicity.
For any $k \in \N$, we can choose $M$ such that $\Prob(\mathcal G) \geq 1 - 2^{-k}$
for all $L$ (sufficiently large).
It is essential that $M$ only depends on $R,T$ and not $L$.
We can then estimate
\begin{equation}
\begin{split}
&\lim_{L \to \infty} \tilde\E\, \abs{
    f(\flow_{L,t} \vec{u}_{L,0})
    - f(\flow_t \vec{u}_0) }\\
\leq\; &\lim_{L \to \infty} \tilde\E\, \abs{
    \I_{\mathcal G} [f(\flow_{L,t} \vec{u}_{L,0})
    - f(\flow_t \vec{u}_0)] } + 2^{-k} \norm{f}_\infty\\
\leq\; &\lim_{L \to \infty} \text{Lip}_f \, \tilde \E\, (\I_{\mathcal G} \norm{ \chi_1 [
        \flow_{L,t} \vec{u}_{L,0} - \flow_{t} \vec{u}_0 ]
    }_{\dataspace^{-2\varepsilon}(\rho)} \wedge 1)
    + 2^{-k} \norm{f}_\infty.
\end{split}
\end{equation}
Here we used respectively the boundedness and Lipschitz continuity of $f$.
Note that the spatial cutoff $\chi_1$ depends on $f$ through $R$.

The two components of $\dataspace^{-2\varepsilon}(\rho)$ are estimated with
Lemmas~\ref{thm:global stability} and~\ref{thm:global stability derivative},
leading to the upper bound
\begin{equation}
\begin{split}
\tilde\E
&\left( \norm{\chi_2 (\vec{u}_{L,0} - \vec{u}_0)}_{\dataspace^{-2\varepsilon}(\rho)} \wedge 1 \right)
    + \tilde\E \left( \norm{\chi_1 [\wick{u(t)^3} - \wick{u_L(t)^3}]}_{H^{-2\varepsilon}(\rho)} \wedge 1 \right)\\
&\quad + \exp(CM^3) \tilde\E [H_{L,\infty} \wedge 1].
\end{split}
\end{equation}
The initial data converges almost surely as $L \to \infty$
and dominated convergence allows us to commute limit and expectation.
Hence the first two terms vanish in the limit.
The same holds for the third term as $L \to \infty$ with $M$ still fixed.
We then pass $k \to \infty$ (and hence $M \to \infty$) to get the claim.
\end{proof}

We can finally post-process this result to obtain that the solution
is almost surely in $C([0,\infty);\, H^{-\varepsilon}(\rho))$
instead of only the bounded time interval $[0,T]$.
This finishes the proof of Theorem~\ref{thm:global moment bounds}.

\begin{lemma}\label{thm:global post-process}
Let $u$ be the solution constructed in Lemma~\ref{thm:construction solution infinite}.
Then $\vec{\mu}$-almost surely $u$ survives for infinite time and
\[
u \in C([0,\infty);\, H^{-\varepsilon}(\rho)).
\]
\end{lemma}
\begin{proof}
Since $\vec{\mu}$ is invariant under $\flow_{t}$ we have 
\begin{equation}
\E_{\vec{\mu}} \left[\norm{\wick{u^3}}^{p}_{L^{p}([0,T];\, H^{-\varepsilon}(\rho))} \right]
= T \, \E_{\vec{\mu}} \norm{\wick{u_{0}^{3}}}_{H^{-\varepsilon}(\rho)}^p
\end{equation}
for any $T > 0$.
From this and Lemma~\ref{thm:phi42 wick moments} we deduce 
\begin{equation}
\vec{\mu}\left( \norm{\wick{u^3}}^{p}_{L^{p}([0,T];\, H^{-\varepsilon})(\rho)} \geq T^{3} \right) \lesssim T^{-2}.
\end{equation} 
Thus by Borel--Cantelli there exists $\vec{\mu}$-almost surely $T^{\ast} > 0$ such that 
\begin{equation}
\norm{\wick{u^3}}^{p}_{L^{p}([0,T];\, H^{-\varepsilon})(\rho)} \leq C_p T^{3}
\end{equation}
for every $T>T^{\ast}$. 
This also implies that 
\begin{equation}
\norm{\wick{u^3}}^{p}_{L^{p}([0,t];\, H^{-\varepsilon}(\rho))} \leq C_p (t+T^{\ast})^{3}.
\end{equation}
Thus from the mild solution formula, Minkowski's integral inequality,
and the $t$-dependent bound for $\Cs_{t}$ in Lemma~\ref{thm:linear operator bounds},
we obtain that
\begin{equation}
\begin{split}
\norm{v(t)}_{H^{-\varepsilon}(\rho)}
&\leq \bignorm{\int_0^t \Cs_{t-s} \wick{u^3(s)} \ds }_{H^{1-2\varepsilon}(\rho)}\\
&\lesssim (1 + t)^{1+\alpha} \norm{\wick{u^{3}}}_{L^1([0,t];\, H^{-\varepsilon}(\rho))}\\
&\lesssim (t+T^{\ast})^{4+\alpha},
\end{split}
\end{equation}
where $\alpha$ is the parameter of $\rho$.
Therefore $v$ is continuous in $t$ as an integral of an $L^{p}$ function.
Finally, we observe that $w_{t}$ is continuous
since $\Cs_{t}$ and $\Cc_{t}$ are continuous in $t$,
and that uniqueness follows from finite-time uniqueness.
\end{proof}

\section{Weak invariance of NLS}\label{sec:schrodinger}
Let us then turn to proving Theorem~\ref{thm:nls main result}.
We begin by considering the nonlinear Schrödinger equation
\begin{equation}\label{eq:nls periodic}
\begin{aligncases}
i \partial_t u_L + \Delta u_L &= \wick{u_{L} \abs{u_L}^2},\\
\Law(u_L (0)) &= \phi_{2, L}^4,
\end{aligncases}
\end{equation}
on $\Lambda_L \times \R$.
Invariance of the periodic complex $\phi^4_2$ measure under this equation was shown
already by Bourgain \cite{bourgain_invariant_1996};
see also \cite{oh_pedestrian_2018} that expands the result in a pedagogic way.
The notions of solution and invariance are both weaker than in the wave case,
as explained in the latter reference.

\begin{remark}
Our construction of the complex $\phi^4_2$ measures and renormalized objects
in Section~\ref{sec:stochastic} uses the massive Gaussian free field,
but no mass term appears in~\eqref{eq:nls periodic}.
This is not an issue as the $L^2$ norm is conserved under the nonlinear Schrödinger flow;
see the discussion around~\cite[Eq.~(1.8)]{oh_pedestrian_2018}.
\end{remark}

\begin{theorem}[Solution in periodic space, {{\cite[Theorem~1.4]{oh_pedestrian_2018}}}]
Equation \eqref{eq:nls periodic} has almost surely a weak solution
in $C(\R_+;\, H^{-\varepsilon}(\Lambda_L))$ for any $L > 0$ and $\varepsilon > 0$.
The law of $u_L(t)$ is the complex $\phi_{2,L}^4$ measure for all $t \geq 0$.
\end{theorem}

Our preceding extension argument is broken for two reasons.
The linear propagator
\begin{equation}
\Ct_t u
\coloneqq \exp(it\Delta) u
\coloneqq \Fou\inv \left[ \exp(-it|\xi|^2) \hat u(\xi) \right]
\end{equation}
does not increase the regularity of its argument.
Therefore the mild solution
\begin{equation}\label{eq:nls mild}
u_L(t) = \Ct_t u_L(0) + \int_0^t [\Ct_{t-s} \wick{u_L(s) \abs{u_L(s)}^2}](x) \ds
\end{equation}
is not amenable to the fixpoint argument of Section~\ref{sec:periodic} in a Besov space.
Moreover, NLS does not possess finite speed of propagation:
wave packets propagate at a speed proportional to their frequency squared.
This means that the argument in Section~\ref{sec:globalization} is not applicable either.

However, if we can accept some loss of regularity,
we can still use the previous tightness argument.
That allows us to approximate full-space solutions by (a subsequence of) periodic solutions.
This sense of invariance was introduced by Albeverio and Cruzeiro \cite{albeverio_global_1990}
in the context of Navier--Stokes equations.

Compactness is given by a version of the usual embedding theorem for Hölder-continuous functions:

\begin{lemma}[Compact embedding II]\label{thm:arzela-ascoli}
For any $0 < \alpha < 1$, the Hölder space $C^{\alpha} ([0, T];\, H^s(\rho))$ is defined by the norm
\[
\norm{f}_{C^{\alpha} ([0, T];\, H^s(\rho))} \coloneqq
\norm{f}_{L^\infty_t H^s(\rho)}
+ \sup_{0 \leq s \neq t \leq T} \frac{\norm{f(t) - f(s)}_{H^s(\rho)}}{\abs{t-s}^\alpha}.
\]
Then the embedding
\[
C^{2\varepsilon} ([0, T];\, H^{s} (\rho))
\hookrightarrow C^{\varepsilon } ([0,T];\, H^{s-\varepsilon} (\rho^{1 + \varepsilon}))
\]
is compact.
\end{lemma}
\begin{proof}
This is an application of the Arzelà--Ascoli theorem.
Fix $R > 0$ and let $B$ be the ball $B(0,R)$ in $C^{2\varepsilon}([0,T];\, H^{s} (\rho))$.
By \cite[X.\S2.5, Corollary~1]{bourbaki_general_1966},
it suffices to verify two conditions.

First, $B$ must be equicontinuous from $[0,T]$ to $H^{s-\varepsilon} (\rho^{1 + \varepsilon})$.
By construction we have for all $f \in B$ and $0 \leq s < t \leq T$ the bound
\begin{equation}
\norm{f(t) - f(s)}_{H^{s-\varepsilon} (\rho^{1 + \varepsilon})}
\leq R \abs{t-s}^{2\varepsilon},
\end{equation}
so this condition holds.

Second, for any fixed $t \in [0,T]$ the point evaluations $B[t] \coloneqq \{ f(t) \colon f \in B \}$
must have compact closure in $H^{s-\varepsilon} (\rho^{1 + \varepsilon})$.
This is true by Theorem~\ref{thm:besov compactness} since $B[t]$ is bounded in $H^s(\rho)$.

This implies that any sequence $u_n$ in $B$ has a subsequence
that converges in $C([0,T];\, H^{s-\varepsilon} (\rho^{1 + \varepsilon}))$.
We can upgrade the convergence to $C^\varepsilon$ in time, since for any $v = u_n - u_m$
with $n$, $m$ in the subsequence we have
\begin{equation}
\frac{\norm{v(t) - v(s)}_{H^{s-\varepsilon}(\rho^{1+\varepsilon})}}{\abs{t-s}^\varepsilon}
\leq \sqrt{\norm{v}_{L^\infty_t H^{s-\varepsilon}(\rho^{1+\varepsilon})}}
    \sqrt{\frac{\norm{v(t) - v(s)}_{H^s(\rho)}}{\abs{t-s}^{2\varepsilon}}},
\end{equation}
where the second term is again bounded in $B$.
Hence the subsequence $u_n$ is Cauchy
in $C^{\varepsilon} ([0,T];\, H^{s-\varepsilon} (\rho^{1 + \varepsilon}))$.
\end{proof}

We first collect a lemma needed for the tightness proof.
The linear propagator $\Ct_t$ is an isometry over an unweighted Besov space $H^s(\R^2)$.
This is not the case in a weighted space.
By giving up some differentiability, we can still get a bound that depends on time.

Let us emphasize that we have not tried to find optimal bounds.
A strong invariance result would require no loss of differentiability at all
(possibly assuming a sufficiently short time interval).

\begin{lemma}[Weighted estimate]\label{thm:approx finite speed}
Fix $1 \leq p, q \leq \infty$, and
let us assume that the weight $\rho$ over $\R^2$
has form $\rho(x) = (1 + \abs x^2)^{-\alpha}$ for $\alpha \in \N$.
The Schrödinger propagator $\Ct_t$ then satisfies for all $s \in \R$ the estimate
\[
\norm{\Ct_t f}_{B^{s}_{p,q}(\rho)} \lesssim (1 + t^{\alpha+2}) \norm{f}_{B^{s+\alpha+2}_{p,q}(\rho)}.
\]
\end{lemma}
\begin{proof}
We will estimate the $L^p$ norm inside
\[
\norm{\Ct_t f}_{B^{s}_{p,q}(\rho)}
= \norm{ 2^{ks} \norm{\Delta_k \Ct_t f}_{L^p(\rho)} }_{\ell^q_k}.
\]
Let us first assume $k \geq 0$.
We can write $\Delta_k \Ct_t = \Delta_k \Delta_k' \Ct_t$
where $\Delta_k'$ is a smooth indicator of a larger annulus,
given by multiplier symbol $\varphi(2^{-k} \,\cdot\,)$.
Let $K_k$ be the convolution kernel of $\Delta_k' \Ct_t$;
by weighted Young's inequality \cite[Theorem~2.1]{mourrat_global_2017} we then have
\begin{equation}
\norm{K_k \ast (\Delta_k f)}_{L^p(\rho)}
\leq \norm{K_k}_{L^1(\rho\inv)} \norm{\Delta_k f}_{L^p(\rho)}.
\end{equation}
As in Theorem~\ref{thm:besov sobolev}, we can write the $L^1$ norm as
\begin{equation}\label{eq:approx finite speed l1}
\int_{\R^2} (1 + \abs x^2)^{-2} (1 + \abs x^2)^{\alpha + 2}
    \abs{\int_{\R^2} e^{ix \cdot \xi} \varphi(2^{-k} \xi) e^{-it \abs\xi^2} \diff\xi} \dx.
\end{equation}
Since we assumed $\alpha$ to be integer,
it is a direct computation to verify
\begin{equation}
(1 + \abs x^2)^{\alpha + 2} e^{ix \cdot \xi}
= (1 - \partial_{\xi_1}^2 - \partial_{\xi_2}^2)^{\alpha + 2} e^{ix \cdot \xi}.
\end{equation}
Since this operator is self-adjoint, we can bound \eqref{eq:approx finite speed l1} with
\begin{equation}
\int_{\R^2} (1 + \abs x^2)^{-2}
    \int_{\R^2} \abs{e^{ix \cdot \xi}
    (1 - \partial_{\xi_1}^2 - \partial_{\xi_2}^2)^{\alpha + 2}
        \!\left[\varphi(2^{-k} \xi) e^{-it \abs\xi^2} \right]} \diff\xi \dx.
\end{equation}
The inner integral is bounded by $C (1 + 2^{(\alpha + 2)k} t^{\alpha + 2})$
since $\varphi$ is smooth and compactly supported.
The integral over $x$ is finite since the weight is integrable.
This gives the required bound.

The case $k = -1$ also gives a constant factor
since the multiplier $\Delta'_{-1} \Ct_t$ is rapidly decreasing.
Again we define $\Delta'_{-1}$ as a smooth indicator of a larger ball,
taking value $1$ in the support of $\Delta_{-1}$.
\end{proof}

\begin{theorem}[Tightness]\label{thm:nls tightness}
Assume that $\rho$ is as in Lemma~\ref{thm:approx finite speed} and $\varepsilon > 0$.
The sequence of periodic solutions $u_L$ is tight in
$C^{1/2-2\varepsilon}([0, T];\, H^{-\alpha - 4 - 2\varepsilon}(\rho^{1+\varepsilon}))$.
\end{theorem}
\begin{proof}
We will show that
\begin{equation}
\sup_L \E \norm{ u_L }_{C^{1/2-\varepsilon} ([0, T];\, H^{-\alpha - 4 - \varepsilon}(\rho))}^2 < \infty.
\end{equation}
This implies tightness in a slightly less regular space by Lemma~\ref{thm:arzela-ascoli}.

From the mild formulation of the equation we obtain that
\begin{equation}
u_L (t) - u_L (s) = (\Ct_{t} -\Ct_{s}) u_L(0)
    + \int^t_s \Ct_{t-r} \wick{u_L(r) \abs{u_L(r)}^2} \diff r,
\end{equation}
so we will need to estimate
\begin{equation}
\norm{ \int^t_s \Ct_{t-r} \wick{u_L(r) \abs{u_L(r)}^2} \diff r }_{H^{-\alpha - 4 - \varepsilon} (\rho)}
+ \norm{ (\Ct_{t} - \Ct_{s}) u (0) }_{H^{-\alpha - 4 - \varepsilon} (\rho)}.
\end{equation}
For the first term, we can use Cauchy--Schwarz to exchange the integrals:
\begin{equation}
\begin{split}
&\bignorm{\int_s^t \Ct_{t-r} \wick{u_L(r) \abs{u_L(r)}^2} \diff r}_{H^{-\alpha-4-\varepsilon}(\rho)}\\
\leq\; &\abs{t-s}^{1/2}
    \left[ \int_s^t \norm{\Ct_{t-r} \wick{u_L(r) \abs{u_L(r)}^2}}_{H^{-\alpha-4-\varepsilon}(\rho)}^2 \diff r \right]^{1/2}\\
\lesssim\; &\abs{t-s}^{1/2}
    \left[ \int_0^T (1 + T^{2\alpha + 4}) \norm{\wick{u_L(r) \abs{u_L(r)}^2}}_{H^{-2-\varepsilon}(\rho)}^2 \diff r \right]^{1/2}.
\end{split}
\end{equation}
Here we used the bound from Lemma~\ref{thm:approx finite speed}.
The Wick power is bounded in expectation by Theorem~\ref{thm:phi42 wick moments},
and the bound is uniform in $L$.

For the second term we use the functional derivative
\begin{equation}
(e^{-it\Delta} - e^{-is\Delta}) f
= \int_s^t (-i\Delta) e^{-ir\Delta} f \diff r
\end{equation}
and fundamental theorem of calculus to compute
\begin{equation}
\begin{split}
\norm{ (\Ct_{t} - \Ct_{s}) u_L(0) }_{H^{-\alpha- 4 - \varepsilon} (\rho)}
&= \norm{ \int_s^t \Delta \Ct_r u_L(0) \diff r }_{H^{-\alpha- 4 - \varepsilon} (\rho)}\\
&\leq \abs{t-s}^{1/2}
    \left[\int_0^T \norm{ \Delta \Ct_r u_L(0) }_{H^{-\alpha- 4 - \varepsilon} (\rho)}^2 \diff r\right]^{1/2}\\
&\lesssim \abs{t-s}^{1/2}
    \left[\int_0^T (1 + T^{2\alpha + 4}) \norm{ u_L(0) }_{H^{- \varepsilon} (\rho)}^2 \diff r\right]^{1/2}.
\end{split}
\end{equation}
Again the expectation is bounded.
All of these estimates are uniform in $L$ and hold for all $t, s \in {[{0},{T}]}$.
\end{proof}

By changing the probability space with Skorokhod's theorem (Lemma~\ref{thm:skorokhod})
we can assume that $u_L \to u$ almost surely.
Convergence of $u_L$ in distribution on $C^{1/2-2\varepsilon}([0, T];\, H^{-\alpha-4-2\varepsilon}(\rho))$
implies that $u_{L}(t)$ converges in distribution on $H^{-\alpha-4-2\varepsilon}(\rho)$. 
Since $\Law(u_L(t)) = \mu_{L}$,
it follows that any limit point $u$ will have $\Law(u(t)) = \mu$.

We still need to establish that $u$ solves the equation;
with the loss of regularity, we see that it satisfies the mild formulation.
However, this result gives no information about pathwise properties in
$H^{-\varepsilon}(\rho)$ where the $\phi^4_2$ measure is supported.

\begin{theorem}[Limit solves NLS]
There exists a probability space $\tilde{\Prob}$ and random variable
$\tilde{u} \in L^2 (\tilde\Prob, C^{\varepsilon} ([0, T];\, H^{-\alpha - 4 - \varepsilon} (\rho)))$ such that
\[ \tilde{u} (t) = \Ct_t \tilde u(0) + \int^t_0 \Ct_{t-s} \wick{\tilde u(s) \smallabs{\tilde u(s)}^2} \diff s \]
and $\Law (\tilde{u} (t)) = \mu$ for all $t \in [0, T]$. 
\end{theorem}
\begin{proof}
For clarity we omit the tildes on $\tilde u$ in the proof.
In addition to almost sure convergence of $u_L$ and $\wick{u_L^3}$, we have by tightness
\begin{equation}
u_L \to u
\quad\text{in}\quad
L^2(\tilde\Prob;\, C^{\varepsilon}([0,T];\, H^{-\alpha-4-\varepsilon}(\rho))).
\end{equation}
Hölder continuity implies that we also have convergence of
\begin{equation}
u_L(t) \to u(t)
\quad\text{in}\quad
L^2(\tilde\Prob;\, H^{-\alpha-4-\varepsilon}(\rho))
\end{equation}
for all $t \in [0,T]$.

We repeat the approximation argument of Lemma~\ref{thm:linear convergence to full space}.
Let $f^{3,\delta}(u)$ approximate $\wick{u^3}$ as in Lemma~\ref{thm:wick approx}.
Then
\begin{equation}
\begin{split}
\int_0^t \Ct_{t-s} (\wick{u(s)^3} - \wick{u_L(s)^3}) \ds
&= \int_0^t \Ct_{t-s} (\wick{u(s)^3} - f^{3,\delta}(u(s))) \ds\\
&\quad+ \int_0^t \Ct_{t-s} (f^{3,\delta}(u(s)) - f^{3,\delta}(u_L(s))) \ds\\
&\quad+ \int_0^t \Ct_{t-s} (\wick{u_L(s)^3} - f^{3,\delta}(u_L(s))) \ds.
\end{split}
\end{equation}
We will now bound the expectation uniformly in $L$.
The third term, and analogously the first, is bounded with 
\begin{equation}
    \begin{split}
&\mathrel{\phantom{=}} \E \bignorm{\int_0^t \Ct_{t-s} (\wick{u_L(s)^3} - f^{3,\delta}(u_L(s))) \diff s}
    _{H^{-4-\varepsilon}(\rho)}\\
&\lesssim_{t}  \sup_{L} \E_{\mu_{L}}
    \norm{\wick{\varphi^3} - f^{\delta}(\varphi)}_{H^{-\varepsilon}(\rho)}\\
&\lesssim \delta^{\alpha},
\end{split}
\end{equation}
where we used our bound on $\mathcal{T}$ and Lemma~\ref{thm:approx-wick}.

To bound the second term, 
we use boundedness of $f^{3,\delta}$:
\begin{equation}
\E \norm{\Ct_{t-s} f^{3,\delta}(u(s))}_{H^{-\alpha-2}(\rho)}
\lesssim \E \norm{f^{3,\delta}(u(s))}_{L^2(\rho)}
\lesssim \E \norm{u(s)}^{3}_{H^{-4-\varepsilon}(\rho)},
\end{equation}
and the same for $u_L$.
This gives that
\begin{equation}
\lim_{L \to \infty} \E \bignorm{\int_0^t \Ct_{t-s} (\wick{u(s)^3} - \wick{u_L(s)^3}) \ds}
    _{H^{-{4-\varepsilon}}(\rho)}
\lesssim_{t} \delta^\alpha.
\end{equation}
Since $\delta$ was arbitrary, this implies that 
\begin{equation}
\lim_{L \to \infty} \E \bignorm{\int_0^t \Ct_{t-s} (\wick{u(s)^3} - \wick{u_L(s)^3}) \ds}
    _{H^{-{4-\varepsilon}}(\rho)} = 0.
\end{equation}
By passing to a further subsequence, we then have almost sure convergence of these nonlinear terms.
This finishes the proof.
\end{proof}

Finally, we can modify the post-processing argument from Lemma~\ref{thm:global post-process}
to extend the solution from the time interval $[0,T]$ to $\R_+$.
This completes the proof of Theorem~\ref{thm:nls main result}.

\appendix

\section{Proof of Lemma~\ref{thm:wick approx}}\label{sec:wick approx proof}

Let us begin with some approximation results for the Green's function.

\begin{lemma}\label{thm:approx-gaussian}
Let $\Omega = \Lambda_{L}$ or $\Omega = \R^2$,
and let $\phi_1, \phi_2$ be two Gaussian fields
in $\Omega$ with translation invariant law,
considered as elements of $\mathcal{C}^{-\varepsilon}(\rho)$.

Let us fix a Fourier cutoff $\chi \in C^{\infty}_c(\R^2)$ supported on the unit ball
and define $\chi_{\delta}(x) = \chi(\delta x)$.
Denote
\begin{itemize}
\item $\lim_{\delta \to 0} \E[(\chi_{\delta}\ast \phi_{1})(x) (\chi_{\delta}\ast \phi_{1})(y)]=G_1(x-y)$,
\item $\lim_{\delta \to 0} \E[(\chi_{\delta}\ast\phi_{2}(x))(\chi_{\delta}\ast\phi_{2}(y))]=G_2(x-y)$,
\item $\lim_{\delta \to 0} \E[(\chi_{\delta}\ast\phi_{1}(x))(\chi_{\delta} \ast \phi_{2})(y)]=G_{1,2}(x-y)$,
\end{itemize}
and assume that for all $2 \leq q < \infty$ we have
\[
\begin{split}
\norm{G_{1}}_{L^{3q}}
    + \norm{G_{2}}_{L^{3q}} + \norm{G_{1,2}}_{L^{3q}} &\lesssim 1,\\
\norm{G_{1}-G_{1,2}}_{L^{3q}} + \norm{G_{2}-G_{1,2}}_{L^{3q}} &\leq \gamma.
\end{split}
\]
Then for $j \leq 3$, $2 \leq p < \infty$, and $\kappa > 2/p$ we have
\[
\E \norm{\wick{\phi^j_1} - \wick{\phi^j_2}}_{\mathcal{C}^{-\kappa}(\rho)}^p
\lesssim \gamma^{p/2}.
\]   
In the complex case this also holds with $\wick{\phi^j}$ replaced by
$\wick{|\phi|^2}$, $\phi^2$, or $\wick{|\phi|^2 \phi}$.
\end{lemma}
\begin{proof}
We treat the real case.
The complex case follows similarly.
Furthermore we set $j = 3$ for concreteness, as the other cases are simpler.
By Theorem~\ref{thm:besov embedding} we can consider the $B_{p,p}^{-\kappa/2}$ norm instead.
Then
\begin{equation}\label{eq:wick approx hc}
\sum_{k \geq -1} 2^{-kp\kappa/2}
    \E \norm{\Delta_k (\wick{\phi^3_1} - \wick{\phi^3_2})}_{L^p(\rho)}^p
\lesssim \sum_{k \geq -1} 2^{-kp\kappa/2}
    \left[ \E \abs{\Delta_k (\wick{\phi^3_1} - \wick{\phi^3_2})(0)}^2 \right]^{p/2}
\end{equation}
by translation invariance of the law and hypercontractivity.
It is hence sufficient to show that
\begin{equation}
\E \abs{\Delta_k(\wick{\phi^3_1} - \wick{\phi^3_2})(0)}^2 \lesssim \gamma 2^{k\kappa/2}.
\end{equation} 
Let $K_{k}$ be the kernel of $\Delta_k$.
We apply Wick's theorem to get
\begin{equation}\label{eq:approx-gaussian expanded}
\begin{split}
&\mathrel{\phantom{=}} \iint K_k(x) K_k(y)
    \E [(\wick{\phi^3_1(x)} - \wick{\phi^3_2(y)})(\wick{\phi^3_1(x)} - \wick{\phi^3_2(y)})] \dx \dy\\
&= \iint K_k(x) K_k(y) \, 3! \bigg[G_1(x-y)^3 + G_2(x-y)^3 - 2 G_{1,2}(x-y)^3 \bigg] \dx \dy\\
&\simeq \iint K_k(x) K_k(y) \bigg[(G_1 - G_{1,2})(G_1^2 + G_1 G_{1,2} + G_{1,2}^2)\\
&\qquad + (G_2 - G_{1,2})(G_2^2 + G_2 G_{1,2} + G_{1,2}^2) \bigg](x-y) \dx \dy\\
&\lesssim \gamma \norm{K_k}^{2}_{L^{q/(q-1)}}
    (\norm{G_1}_{L^{3q}}^2 + \norm{G_2}_{L^{3q}}^2 + \norm{G_{1,2}}_{L^{3q}}^2)\\
&\lesssim \gamma 2^{2k/q}.
\end{split}
\end{equation}
In the last line we used that $\norm{K_k}_{L^{q/(q-1)}} \lesssim 2^{k/q}$,
which follows from interpolating the $L^1$ and $L^\infty$ bounds for the kernel.
Choosing $q = 4/\kappa$, we get that~\eqref{eq:wick approx hc} converges.
\end{proof}

\begin{lemma}\label{thm:green-approx}
Let $\chi_{\delta}$ be a mollifier as above,
$2 \leq p < \infty$, and $G \in W^{1,q}(\R^2)$ for all $q < 2$.
For $\alpha < 1/{3p}$ we have
\[
\norm{\chi_{\delta}\ast G-G}_{L^{3p}}
    \lesssim \delta^{\alpha}
\;\text{ and }\;
\norm{\chi_{\delta} \ast \chi_{\delta} \ast G-  \chi_{\delta}\ast G}_{L^{3p}}
    \lesssim \delta^{\alpha}.
\]
Furthermore on $L\mathbb{T}^2$, the truncated Green's function takes the form
\[
G^N(x) = \sum_{n \in \frac{1}{L}, \abs n \leq N} \frac{1}{m^2 + |n|^2}e^{i n \cdot x}.
\]
Then for any $N_1 \leq N_2 \in \N$ and $p < \infty$,
there exists $\alpha > 0$ such that
\[
\norm{G^{N_1}(x) - G^{N_2}}_{L^{3p}} \lesssim_L N_1^{-\alpha}.
\]
\end{lemma}
\begin{proof}
The first statement follows from the assumption,
Besov embedding, and the convolution estimate~\cite[Lemma~A.8]{gubinelli_pde_2021}.
The second is proven in~\cite[Lemma~4.2]{oh_pedestrian_2018}. 
\end{proof}

With these estimates, we can first prove Lemma~\ref{thm:wick approx}
and then state a result on the convergence of Fourier-truncated fields.
In the following proof we work with the more convenient smooth cutoff instead of sharp truncation,
but this does not change the limiting objects.

\begin{lemma}\label{thm:approx-wick}
Let $Z$ be sampled from $\nu$ and $Z_L$ from $\nu_L$.
Define $Z_\delta = \chi_{\delta} (\inorm\nabla) Z$,
$Z_{L,\delta} = \chi_{\delta} (\inorm\nabla) Z_L$,
and $a_{\delta} = \E [(Z_\delta)^2]$.
In the real case we define
\[
f^{3,\delta}(Z) \coloneqq (Z_\delta)^3 - 3 a_\delta Z_\delta,
\]
whereas in the complex case we put
\begin{gather*}
f^{3,\delta}(Z)
\coloneqq Z_\delta |Z_\delta|^2 - 2 a_\delta Z_\delta,\\
f^{2,\delta}(Z) \coloneqq (Z_\delta)^2 - a_\delta.
\end{gather*}
Then if $2 \leq p < \infty$ and $\psi \in L^{4p} (\Prob, B_{3p, 3p}^{\varepsilon}(\rho^{1/4}))$,
there exists $\alpha > 0$ such that
\begin{align}
&\E \norm{\wick{(Z + \psi)^3} -
    f^{3,\delta}(Z_\delta + \psi)}^{p}_{B^{-\varepsilon}_{p,p} (\rho)}
\lesssim \delta^{\alpha},
\text{ and}\label{eq:wick approx first}\\
&\sup_{L} \E \norm{\wick{(Z_L + \psi)^3} -
    f^{3,\delta}(Z_{L,\delta} + \psi)}^{p}_{B^{-\varepsilon}_{p,p} (\rho)}
\lesssim \delta^{\alpha}.
\label{eq:wick approx second}
\end{align}
Analogous statements hold for $\wick{(Z + \psi)^2}$ and in the complex case.
\end{lemma}
\begin{proof}
We show only \eqref{eq:wick approx first} in the real case.
Equation~\eqref{eq:wick approx second} follows similarly,
once we note that we may replace the renormalization constant by $a_{\delta,L} = \E Z^{2}_{\delta,L}$
as in Lemma~\ref{lemma:wick-ordering-rem}.
Furthermore the square and the complex case follow analogously.

Since $f^{3,\delta}$ restricts the Fourier support of its argument to a bounded set,
it follows that its image is in $L^2(\rho)$ and the map is continuous.

Denote $\psi_{\delta} = \chi_\delta(\inorm\nabla)\psi$.
Then as in Lemma~\ref{thm:wick binomial} we have
\begin{equation}
f^{3,\delta}((Z_\delta + \psi)^3)
= \sum_{j = 0}^3 { (Z_\delta)^j }_\delta \,\psi_{\delta}^{3 - j},
\end{equation}
where $\wick{\,\cdot\,}_\delta$ denotes Wick ordering with renormalization constant $a_\delta$.
We can then estimate
\begin{equation}
\begin{split}
&\mathrel{\phantom{=}} \norm{\wick{(Z_\delta)^j}_{\delta} \, \psi_{\delta}^{3 - j}
    - \wick{Z^j} \, \psi^{3 - j}}_{B^{-\varepsilon}_{p,p}(\rho)}^p\\
&\lesssim \norm{\wick{Z_\delta^j}_\delta - \wick{Z^j}}_{\mathcal C^{-\varepsilon}(\rho^{1/4})}^p
        \norm{\psi_\delta^{3-j}}_{B^{3\varepsilon/2}_{p,p}(\rho^{3/4})}^p\\
&\qquad + \norm{\wick{Z^j}}_{\mathcal C^{-\varepsilon}(\rho^{1/4})}^p
        \norm{\psi_\delta^{3-j} - \psi^{3-j}}_{B^{3\varepsilon/2}_{p,p}(\rho^{3/4})}^p.
\end{split}
\end{equation}
We then take expectation and apply Hölder.
Then the bound for
\begin{equation}
\E \norm{\wick{Z_\delta^j}_\delta - \wick{Z^j}}_{\mathcal C^{-\varepsilon}(\rho^{1/4})}^{4p}
\end{equation}
follows from Lemmas~\ref{thm:approx-gaussian} and~\ref{thm:green-approx}
under the Green's function bounds given in \cite[Chapter~7]{glimm_quantum_1987}.
For $\psi - \psi_\delta$ we use the Bernstein estimate
\begin{equation}
\norm{\psi-\psi_\delta}_{B^{3\varepsilon/2}_{3p,3p}}
\leq \delta^{\varepsilon/2} \norm{\psi}_{B^{2\varepsilon}_{3p,3p}}.
\end{equation}
These finish the proof.
\end{proof}

\begin{lemma}\label{thm:fourier-wick-approx}
Let $Z_L$ be sampled from $\nu_L$ and let
$\psi \in L^{4p}(\Prob , B^{\varepsilon}_{p,p}(\Lambda_{L}))$.
Then there exists $\alpha > 0$ such that for all $N \geq \delta^{-1}$ we have
\begin{equation*}  
\E \norm{\wick{(P_N (Z_L + \psi))^3} - f^{3,\delta}(P_N [Z_{L,\delta} + \psi])}^{p}
    _{B^{-\varepsilon}_{p,p} (\rho)}
\lesssim \delta^{\alpha}.
\end{equation*}
\end{lemma}
\begin{proof}
The proof is a minor modification of Lemma~\ref{thm:approx-wick},
using now the second estimate in Lemma~\ref{thm:green-approx}.
\end{proof}

\section{Computations for Theorem~\ref*{thm:sq tightness}}\label{sec:phi42 appendix}

For brevity, we drop the subscript $L$ from the notation.
The following estimates are uniform in the period length $L$
and hold also in the infinite volume.
Similarly, we do not write the time dependency
since these pointwise-in-time estimates are uniform by stationarity.

Recall that we consider either the real or complex scalar field.
In order to prove Theorem~\ref{thm:sq tightness}, we need to bound the absolute value
of~\eqref{eq:sq tested} or~\eqref{eq:sq tested complex} with
\begin{equation}
Q(Z) + \delta \left(
    m^2 \norm{\psi}_{L^2(\rho)}^2
    + \norm{\psi}_{H^s(\rho)}^2
    + \norm{\psi}_{L^4(\rho^{1/2})}^4
\right),
\end{equation}
where $Q(Z)$ is bounded in expectation and $\delta > 0$ is chosen to be small.

We bound each of the terms in the following lemmas,
selecting $Q(Z)$ to consist of norms of Wick powers of $Z$.
The norms have bounded expectation by Lemma~\ref{thm:wick gff moments}.
In each lemma we use the product inequality (Theorem~\ref{thm:besov multiplication})
and Besov duality (Theorem~\ref{thm:besov duality}).
These calculations are originally due to Mourrat and Weber \cite{mourrat_global_2017}.

\begin{lemma}\label{thm:gh lemma}
Let $\rho_1$ and $\rho_2$ be polynomial weights and $s, \varepsilon > 0$.
We have the following two estimates:
\begin{align*}
\norm{f^2}_{B^s_{1,1}(\rho_1 \rho_2)}
&\lesssim \norm{f}_{L^2(\rho_1)} \norm{f}_{H^{s+\varepsilon}(\rho_2)}\\
\norm{f^3}_{B^s_{1,1}(\rho_1^2 \rho_2)}
&\lesssim \norm{f}_{L^4(\rho_1)}^2 \norm{f}_{H^{s+\varepsilon}(\rho_2)}.
\end{align*}
In the complex setup we can replace $f^3$ by $f \abs f^2$ or $\overline f \abs f^2$.
\end{lemma}
\begin{proof}
\cite[Lemma A.7]{gubinelli_pde_2021}.
The decomposition used in the proof adapts naturally to the complex variant.
\end{proof}

\begin{lemma}
Assume that $\rho \in L^1(\R^2)$ and $\varepsilon < 1/4$.
Then for any $\delta>0$ there exists a constant $C>0$ that
\[
\abs{\int_{\R^2} \rho^2 \psi^3 Z \dx}
\leq C \norm{Z}_{\mathcal C^{-\varepsilon}(\rho^{1/8})}^8
    + \delta \left( \norm{\psi}_{L^4(\rho^{1/2})}^4 + \norm{\psi}_{H^s(\rho)}^2 \right).
\]
In the complex case we can replace $\psi^3$ on the left by $\psi \abs\psi^2$.
\end{lemma}
\begin{proof}
We first use duality and Lemma~\ref{thm:gh lemma} to estimate
\begin{equation}\label{eq:bound psi3 z}
\begin{split}
\int_{\R^2} \abs{\rho^2 \psi^3 Z} \dx
&\lesssim \norm{\psi^3}_{B^\varepsilon_{1,1}(\rho^{15/8})}
    \norm{Z}_{\mathcal C^{-\varepsilon}(\rho^{1/8})}\\
&\lesssim \norm{\psi}_{L^4(\rho^{1/2})}^2
    \norm{\psi}_{H^{2\varepsilon}(\rho^{7/8})}
    \norm{Z}_{\mathcal C^{-\varepsilon}(\rho^{1/8})}.
\end{split}
\end{equation}
Inside the middle Besov norm, we can trade off some weight via
\begin{equation}
\norm{\rho^{7/8} \Delta_j \psi}_{L^2}
\leq \norm{\rho^{1/8}}_{L^8} \norm{\rho^{3/4} \Delta_j \psi}_{L^{8/3}}.
\end{equation}
We can also increase the regularity from $2\varepsilon$ to $1/2$.
This simplifies the interpolation
\begin{equation}
\norm{\psi}_{B^{1/2}_{8/3,\infty}(\rho^{3/4})}
\lesssim \norm{\psi}_{B^{0}_{4,\infty}(\rho^{1/2})}^{1/2}
    \norm{\psi}_{B^{1}_{2,\infty}(\rho)}^{1/2}
\lesssim \norm{\psi}_{L^4(\rho^{1/2})}^{1/2}
    \norm{\psi}_{H^{1}(\rho)}^{1/2}.
\end{equation}
We substitute this back into~\eqref{eq:bound psi3 z}
and finish with Young's inequality.
\end{proof}

\begin{lemma}
Assume that $\rho^{1/2} \in L^1(\R^2)$ and $\varepsilon < 1/2$.
Then for every $\delta>0$ there exists a constant $C>0$ such that
\[
\abs{\int_{\R^2} \rho^2 \psi^2 \wick{Z^2} \dx}
\lesssim C \norm{\wick{Z^2}}_{\mathcal C^{-\varepsilon}(\rho^{1/8})}^4
    + \delta \left( \norm{\psi}_{L^4(\rho^{1/2})}^4 + \norm{\psi}_{H^s(\rho)}^2 \right).
\]
The same result holds in the complex case
with $\wick{Z^2}$ replaced by either $\wick{\abs Z^2}$ or $Z^2$ on both sides,
and $\psi^2$ optionally replaced by $\abs\psi^2$ on the left.
\end{lemma}
\begin{proof}
Again, duality and Lemma~\ref{thm:gh lemma} give
\begin{equation}
\int_{\R^2} \abs{\rho^2 \psi^2 \wick{Z^2}} \dx
\lesssim \norm{\psi}_{L^2(\rho^{7/8})}
    \norm{\psi}_{H^{2\varepsilon}(\rho)}
    \norm{\wick{Z^2}}_{\mathcal C^{-\varepsilon}(\rho^{1/8})}.
\end{equation}
We can again trade off some weight in
\begin{equation}
\norm{\rho^{7/8} \psi}_{L^2}
\leq \norm{\rho^{3/8}}_{L^{4/3}} \norm{\rho^{1/2} \psi}_{L^4}.
\end{equation}
We can increase the regularity in the middle term
and make the weight larger in the last term to make them match the statement.
Young's inequality again finishes the proof.
The complex variants are proved identically.
\end{proof}

\begin{lemma}
Assume that $\rho^{1/2} \in L^1(\R^2)$ and $\varepsilon < 1$.
Then for every $\delta>0$ there exists $C>0$ such that
\[
\abs{\int_{\R^2} \rho^2 \psi \wick{Z^3} \dx}
\leq C \norm{\wick{Z^3}}_{\mathcal C^{-\varepsilon}(\rho^{1/2})}^2
    + \delta \norm{\psi}_{H^s(\rho)}^2.
\]
The same bound holds with $\wick{Z \abs Z^2}$ in place of $\wick{Z^3}$.
\end{lemma}
\begin{proof}
By duality
\begin{equation}
\int_{\R^2} \abs{\rho^2 \psi \wick{Z^3}} \dx
\lesssim \norm{\psi}_{B^\varepsilon_{1,1}(\rho^{3/2})}
    \norm{\wick{Z^3}}_{\mathcal C^{-\varepsilon}(\rho^{1/2})}.
\end{equation}
Then we do a series of tradeoffs in
\begin{equation}
\norm{\psi}_{B^\varepsilon_{1,1}(\rho^{3/2})}
\lesssim \norm{\psi}_{B^{1}_{1,2}(\rho^{3/2})}
\lesssim \norm{\psi}_{B^{1}_{2,2}(\rho)}
\end{equation}
and finish with Young's inequality.
The other two cases are identical.
\end{proof}

\begin{lemma}
Assume that $\rho^{1/2} \in L^1(\R^2)$.
Then there exists $C > 0$ such that
\[
\abs{\int_{\R^2} \psi (\nabla\rho^2 \cdot \nabla\psi) \dx}
\leq C + \delta
    \left( \norm{\psi}_{L^4(\rho^{1/2})}^4 + \norm{\psi}_{H^s(\rho)}^2 \right).
\]
The same bound also holds if either $\psi$ on the left
is replaced by $\overline\psi$.
\end{lemma}
\begin{proof}
Let us observe that we can write the dot product components as
\begin{equation}
(\partial_j \rho^2)(\partial_j \psi)
= (\partial_j [1 + x_1^2 + x_2^2]^{-\alpha} )(\partial_j \psi)
= -\frac{2\alpha x_j \rho(x)^2}{1 + x_1^2 + x_2^2} (\partial_j \psi)
\end{equation}
The factor in front is uniformly bounded by $\alpha \rho(x)^2$.
Thus
\begin{equation}
\begin{split}
\int_{\R^2} \abs{ \psi(x) (\nabla\rho^2 \cdot \nabla\psi)(x) } \dx
&\leq \alpha \int_{\R^2} \rho(x)^2 \abs{\psi(x)} \abs{\nabla\psi(x)} \dx\\
&\leq \alpha \norm{\psi}_{L^2(\rho)} \norm{\nabla\psi}_{L^2(\rho)}\\
&\leq C + \delta
    \left( \norm{\psi}_{L^4(\rho^{1/2})}^4 + \norm{\psi}_{H^s(\rho)}^2 \right),
\end{split}
\end{equation}
where we did again a weight--$L^p$ tradeoff and applied Young.
\end{proof}

\section{Exponential tails}\label{sec:tails}

In order to prove the existence of Wick powers,
we need to establish exponential tails of some weighted Besov norms for
$\phi_{2, L}^4$ uniformly in $L$.
More concretely we will define the measure
\begin{equation}
\bar{\mu}_L(A) = \frac{1}{\mathcal Z_L}
    \int_A \exp (h(\norm{\inorm\nabla^{-\varepsilon} \phi}_{L^p (\rho)})) \diff \mu_L(\phi),
\end{equation}
where $h \colon \R \to \R$ is a smooth function, constant near $0$ and growing linearly at infinity,
and $\mathcal Z_L$ is the associated normalization constant.
We will prove that $\sup_L \mathcal Z_L < \infty$.

We begin with the following lemma.
In finite volume the Gaussian tails of $\mu_L$ are not difficult to
establish; see \cite[Section~3]{barashkov_variational_2020}.
This \emph{a priori} bound means that the assumptions of the lemma are satisfied.
The lemma then makes the uniform bound easier to derive.

\begin{lemma}[{{\cite[Lemma~A.7]{barashkov_elliptic_2021}}}]\label{lemma:exponential-trick}
Let $(\Omega, F)$ be a measurable space and $\upsilon$ be a probability
measure on $\Omega$. Let $S \colon \Omega \mapsto \R$ be a measurable
function such that
\[
\exp (S) \in L^1 (\mathrm d \upsilon).
\]
Define $\diff \nu_S = \frac{1}{\int \exp (S) \diff \upsilon} \exp (S) \diff \upsilon$.
Then
\[
\int \exp (S) \diff \upsilon \leq \exp \left( \int S (x) \diff\upsilon_S \right).
\]
\end{lemma}
\begin{proof}
Multiplying both sides of
\begin{equation}
    \diff \nu_S = \frac{1}{\int \exp (S) \diff \upsilon} \exp (S) \diff\upsilon
\end{equation}
by $\exp (- S)$ and integrating we obtain
\begin{equation}
\left( \int \exp (- S) \diff \upsilon_S \right) \left( \int \exp (S)
    \diff \upsilon \right) = 1.
\end{equation}
Then it remains to apply Jensen's inequality to the first factor.
\end{proof}

We choose $S = h(\norm{\rho \inorm\nabla^{- \varepsilon} \phi}_{L^p})$
and $\upsilon = \mu_L$ in the lemma.
Then the claim follows if we can find a uniform estimate for
\begin{equation}
\int_{H^{-\varepsilon}(\rho)} \hspace{-1em}
    h(\norm{\rho \inorm\nabla^{- \varepsilon} \phi}_{L^p}) \diff \mu_L(\phi).
\end{equation}
To do this we again use stochastic quantization.
By the chain rule the gradient operator of $h$ is
\begin{equation}
\nabla_{\phi}\, h(\norm{\rho \inorm\nabla^{- \varepsilon} \phi}_{L^p})
= \frac{h'(\norm{\rho \inorm\nabla^{- \varepsilon} \phi}_{L^p})}
    {\norm{\inorm\nabla^{- \varepsilon} \phi}^{p - 1}_{L^p(\rho)}}
(\rho \inorm\nabla^{- \varepsilon} \phi)^{p - 1} \rho \inorm\nabla^{- \varepsilon},
\end{equation}
and in the complex valued case 
\begin{equation}
\nabla_{\phi}\, h(\norm{\rho \inorm\nabla^{- \varepsilon} \phi}_{L^p})
= \frac{h'(\norm{\rho \inorm\nabla^{- \varepsilon} \phi}_{L^p})}
    {\norm{\inorm\nabla^{- \varepsilon} \phi}^{p - 1}_{L^p(\rho)}}
(\rho \inorm\nabla^{- \varepsilon} \phi)^{p - 2} \rho \inorm\nabla^{- \varepsilon} \bar{\phi}
\end{equation}
We can write the right-hand side via the adjoint of $\rho \inorm\nabla^{-\varepsilon}$ as
\begin{equation}
V(\phi) = \frac{h'(\norm{\rho \inorm\nabla^{- \varepsilon} \phi}_{L^p})}
    {\norm{\inorm\nabla^{- \varepsilon} \phi}^{p - 1}_{L^p(\rho)}}
\, (\rho \inorm\nabla^{- \varepsilon})^*\! \left[(\rho \inorm\nabla^{- \varepsilon} \phi)^{p - 1}\right].
\end{equation}
and analogously in the complex case. We then have the following lemma:

\begin{lemma}
The measure $\bar{\mu}_L$ is an invariant measure for the equation
\[
\partial_t X + (m^2 - \Delta) X + \wick{X^3} = V (X) + \xi,
\]
and in the complex case 
\[
\partial_t X + (m^2 - \Delta) X + \wick{|X|^2 X} = V (X) + \xi,
\]
where $\xi$ is space-time white noise.
\end{lemma}
\begin{proof}
Note that $V$ is continuous on $\Cc^{-\delta}(\Lambda_{L})$.
With this in mind the proof becomes a minor modification of the proof of Da~Prato and Debussche
\cite[Section 4]{da_prato_strong_2003} and we omit it.
\end{proof}

Again performing the Da~Prato--Debussche trick, i.e.\ decomposing $X = Z + \bar{\psi}$,
we obtain that $\bar{\psi}$ satisfies
\begin{equation}\label{eq:sq-exponential}
\partial_t \bar\psi + (m^2 - \Delta) \bar\psi + \wick{(Z + \bar\psi)^3} = V(Z + \bar\psi).
\end{equation}
We again test the equation with $\rho \bar\psi$ to obtain
\begin{equation}
\begin{split}
&\mathrel{\phantom{=}} \partial_t \int \rho \bar{\psi}^2 \dx + m^2 \int \rho \bar{\psi}^2 \dx
  + \int \abs{\nabla \bar{\psi}}^2 \dx + \int \bar{\psi}^4 \dx + G (Z,\bar{\psi})\\
&= \int \rho V (Z + \bar{\psi}) \bar{\psi} \dx,
\end{split}
\end{equation}
where the residual term $G$ is as in Theorem~\ref{thm:sq tightness}.
From the definitions and Hölder's inequality
\begin{equation}
\begin{split}
& \int \rho V (Z + \bar{\psi}) \bar{\psi} \dx\\
=\; & \frac{h'(\norm{\rho \inorm\nabla^{-\varepsilon} \psi}_{L^p})}
    {\norm{\inorm\nabla^{- \varepsilon} (Z + \bar{\psi})}^{p- 1}_{L^p (\rho)}}
    \int (\rho \inorm\nabla^{- \varepsilon} (Z + \bar{\psi}))^{p - 1}
    \rho \inorm\nabla^{- \varepsilon} (\rho \bar{\psi}) \dx\\
\lesssim\; & \frac{1}{\norm{\inorm\nabla^{- \varepsilon} (Z + \bar{\psi})}^{p - 1}_{L^p (\rho)}}
    \norm{\inorm\nabla^{-\varepsilon} (Z + \bar{\psi})}_{L^p (\rho)}^{p - 1}
    \norm{\inorm\nabla^{- \varepsilon} (\rho \bar{\psi})}_{L^p (\rho)}\\
\lesssim\; & \norm{\inorm\nabla^{- \varepsilon} (\rho \bar{\psi})}_{L^p (\rho)}\\
\lesssim\; & \norm{\rho}_{H^{-\varepsilon}(\R^2)} \norm{\bar{\psi}}_{H^1(\rho)}.\\
\leq\; & C + \frac 1 2 \norm{\bar{\psi}}^2_{H^1(\rho)}.
\end{split}
\end{equation}
We thus have that
\begin{equation}
\int \rho V (Z + \bar{\psi}) \bar{\psi} \dx
\leq \frac 1 2 \left( m^2 \int \rho \bar{\psi}^2 + \int | \nabla \bar{\psi} |^2 \dx \right) + C.
\end{equation}
We also apply the reasoning from Section \ref{sec:SQ} to the remainder term $G(Z, \bar\psi)$.
Upon taking an expectation, the time derivative and white noise integrals vanish.

This implies again the boundedness of $H^1$ norm.
Besov embedding then gives
$\sup_L \E \norm{\bar{\psi}}^2_{L^p (\rho)} < \infty$,
which gives the statement for exponential tails in $L^p$ norm.

\begin{corollary}
As $p$ was arbitrary, we also have
\[
\sup_L \int \exp (\norm{\phi_L}_{\mathcal C^{-2\varepsilon}(\rho)}) \diff \mu_L(\phi_L) < \infty.
\]
This implies that $\E \norm{\phi_L}^p_{\mathcal C^{-2\varepsilon}(\rho)}$ is
finite uniformly in $L$ for any $p < \infty$.
\end{corollary}
\begin{proof}
By Besov embedding it is sufficient to prove the claim
with $\norm{\phi_{L}}_{B^{-\varepsilon}_{p,p}(\rho)}$ in place of
$\norm{\phi_L}_{\mathcal C^{-2\varepsilon}(\rho)}$.
Now by Lemma~\ref{lemma:exponential-trick} we have
\begin{equation}
\log \int \exp (\norm{\phi_{L}}_{B^{-\varepsilon}_{p,p}(\rho)}) \diff \mu_L(\phi_L)
\leq \int \norm{\phi_{L}}_{B^{-\varepsilon}_{p,p}(\rho)} \diff \bar{\mu}_{L}(\phi_L),
\end{equation}
and the right-hand side is bounded by the above discussion. 
\end{proof}

\section{Continuity of linear solution}\label{sec:continuity}

In this appendix we show that the linear solution~\eqref{eq:periodic w}
is continuous in time in the polynomially weighted space.
The proof holds both for periodic and full-space initial data:
this dependency is fully encapsulated in the Green's function $G$
as in Appendix~\ref{sec:wick approx proof}.

We first show the claim for the flow started from Gaussian data.

\begin{lemma}[Continuity with GFF data]\label{thm:continuity gff}
Let $w_0$ be sampled from the Gaussian free field~$\nu$ (or $\nu_L$)
and $\xi_0$ from the white noise measure on~$\R^2$ (respectively $\Lambda_L$).
Denote by $w_t \coloneqq \Cc_t w_0 + \Cs_t \xi_0$ the solution to the linear wave equation.

For $j = 1$, $2$, $3$,
there exist versions $\wick{\tilde w_t^j} \in C([0,T];\, \mathcal C^{-2\varepsilon}(\rho))$
such that $\Prob(\wick{\tilde w_t^j} = \wick{w_t^j}) = 1$ for all $t \in [0, T]$,
and for all $p < \infty$ we have
\[
\E \norm{\wick{\tilde w^j}}_{C([0,T];\, \mathcal C^{-2\varepsilon}(\rho))}^p < \infty.
\]
\end{lemma}
\begin{proof}
The results are given by the Kolmogorov continuity theorem
once we have the estimates
\begin{equation}
\E \norm{\wick{w_{t}^j}}_{\mathcal C^{-2\varepsilon}(\rho)}^p
\lesssim 1, \quad
\E \norm{\wick{w_{t+s}^j} - \wick{w_t^j}}_{\mathcal C^{-2\varepsilon}(\rho)}^p
\lesssim \abs s^{1+\beta}
\end{equation}
for some $\beta > 0$ and all $t \in [0, T]$ and $\abs s \lesssim 1$.
We only prove the second estimate here as the first one is similar.

By stationarity we can fix $t = 0$,
and by Besov embedding replace the space by $B^{-\varepsilon}_{p,p}(\rho)$ for $p$ large.
As the weight belongs to $L^p$, translation invariance and hypercontractivity reduce the computation to
\begin{equation}
\begin{split}
&\mathrel{\phantom{=}} \E \norm{\wick{w_s^j} - \wick{w_0^j}}_{B^{-\varepsilon}_{p,p}(\rho)}^p\\
&\lesssim \sum_{k \geq -1} 2^{-k p \varepsilon} \left[
    \E \abs{\Delta_k [\wick{w_s^j} - \wick{w_0^j}](0)}^2
\right]^{p/2}.
\end{split}
\end{equation}
As in Lemma~\ref{thm:approx-gaussian},
the expectation is then expanded with the convolution kernels and Wick's theorem as
\begin{equation}\label{eq:continuity expanded}
\begin{split}
&\mathrel{\phantom{=}} \iint K_k(x) K_k(y)
    \E \bigg[ (\wick{w_s(x)^j} - \wick{w_0(x)^j})(\wick{w_s(y)^j} - \wick{w_0(y)^j}) \bigg] \dx\dy\\
&= 2 \iint K_k(x) K_k(y) j! G(x-y)^j \dx \dy\\
&\qquad - 2 \iint K_k(x) K_k(y) j! \left[ \E [w_s(x) w_0(y)] \right]^j \dx \dy.
\end{split}
\end{equation}
Here we again used stationarity.

Let us then remark that
\begin{equation}
\begin{split}
\E [\Cc_s w_0(x) w_0(y)]
&= \int \tilde K_s(z) \E [w_0(x-z) w_0(y)] \diff z\\
&= \int \tilde K_s(z) G(x-y-z) \diff z\\
&= \Cc_s G(x-y),
\end{split}
\end{equation}
where $\tilde K_s$ is the convolution kernel of $\Cc_s$.
(This formal computation can be made rigorous with Lemma~\ref{thm:green-approx}.)
Together with the independence of $w_0$ and $\xi_0$ this gives
\begin{equation}
\E [w_s(x) w_0(y)]
= \E [(\Cc_s w_0 + \Cs_s \xi_0)(x) w_0(y)]
= \Cc_s G(x-y).
\end{equation}
Hence we have shown
\begin{equation}
\eqref{eq:continuity expanded}
= 2 \iint K_k(x) K_k(y) j! [G(x-y)^j - \Cc_s G(x-y)^j] \dx \dy.
\end{equation}
Since for $q \geq 2$ we have
\begin{equation}
\begin{split}
\norm{G - \Cc_s G}_{L^{3q}}
&\lesssim \norm{G - \Cc_s G}_{W^{1-\varepsilon, 2}}\\
&= \norm{\inorm\nabla^{1-\varepsilon} (1 - \cos(\inorm\nabla s)) G}_{L^2}\\
&\lesssim \norm{\abs s^{\varepsilon/2} \inorm\nabla^{1-\varepsilon/2} G}_{L^2},
\end{split}
\end{equation}
we can proceed with the same Hölder estimate as in~\eqref{eq:approx-gaussian expanded} to get
\begin{equation}
\E \norm{\wick{w_s^j} - \wick{w_0^j}}_{B^{-\varepsilon}_{p,p}(\rho)}^p
\lesssim \sum_{k \geq -1} 2^{-k p \varepsilon}
    \left[ \abs s^{\varepsilon/2} 2^{2k/q} \right]^{p/2}.
\end{equation}
We now choose $q = 2/\varepsilon$ and $p$ such that $p > 4/\varepsilon$.
\end{proof}

\begin{corollary}[Continuity with $\phi^4_2$ data]\label{thm:continuity phi4}
The statements of Lemma~\ref{thm:continuity gff} hold also for
$w_t \coloneqq \Cc_t z_0 + \Cs_t \xi_0$, where
$z_0$ is sampled from the $\phi^4_2$ measure~$\mu$ (or $\mu_L$)
and $\xi_0$ from the white noise measure on~$\R^2$ (respectively $\Lambda_L$).
\end{corollary}
\begin{proof}
By Theorem~\ref{thm:phi42 weak limit} we can decompose $z_0 = w_0 + \psi$,
where $w_0$ is as in Lemma~\ref{thm:continuity gff} and $\psi \in H^{2-\varepsilon}(\rho)$.
Since $\Cc_t$ is continuous in time,
the results follow by expanding $\wick{(\Cc_t w_0 + \Cc_t \psi)^j}$,
Besov product and embedding formulas,
and boundedness of $\Cc_t$ as in Lemma~\ref{thm:linear operator bounds}.
\end{proof}

\bibliographystyle{plain}
\bibliography{nlw-2d-ref}

\begin{thebibliography}{10}

\bibitem{aizenman_marginal_2021}
Michael Aizenman and Hugo Duminil-Copin.
\newblock Marginal triviality of the scaling limits of critical {4D} {Ising}
  and ${\phi}_4^4$ models.
\newblock {\em Annals of Mathematics}, 194(1), July 2021.

\bibitem{albeverio_global_1990}
Sergio Albeverio and Ana-Bela Cruzeiro.
\newblock Global flows with invariant ({Gibbs}) measures for {Euler} and
  {Navier}--{Stokes} two dimensional fluids.
\newblock {\em Communications in Mathematical Physics}, 129(3):431--444, May
  1990.

\bibitem{albeverio_invariant_2020}
Sergio Albeverio and Seiichiro Kusuoka.
\newblock The invariant measure and the flow associated to the
  ${\Phi^4_3}$-quantum field model.
\newblock {\em Annali Scuola Normale Superiore di Pisa. Classe di Scienze},
  20(4):1359--1427, December 2020.

\bibitem{arnold_mathematical_1980}
Vladimir~I. Arnol'd.
\newblock {\em Mathematical methods of classical mechanics}, volume~60 of {\em
  Graduate texts in mathematics}.
\newblock Springer, 1980.

\bibitem{bahouri_fourier_2011}
Hajer Bahouri, Jean-Yves Chemin, and Raphaël Danchin.
\newblock {\em Fourier Analysis and Nonlinear Partial Differential Equations}.
\newblock Number 343 in Grundlehren der mathematischen {Wissenschaften}.
  Springer-Verlag, Berlin Heidelberg, 2011.

\bibitem{barashkov_elliptic_2021}
Nikolay Barashkov and Francesco~Carlo De~Vecchi.
\newblock Elliptic stochastic quantization of {{Sinh-Gordon QFT}}.
\newblock {\em Stochastics and Partial Differential Equations: Analysis and
  Computations}, July 2025.

\bibitem{barashkov_variational_2020}
Nikolay Barashkov and Massimiliano Gubinelli.
\newblock A variational method for $\phi^4_3$.
\newblock {\em Duke Mathematical Journal}, 169(17):3339--3415, November 2020.

\bibitem{barashkov_eyringkramers_2024}
Nikolay Barashkov and Petri Laarne.
\newblock Eyring--{{Kramers}} law for the hyperbolic $\phi^4$ model.
\newblock arXiv:2410.03495v1, October 2024.

\bibitem{billingsley_convergence_1999}
Patrick Billingsley.
\newblock {\em Convergence of probability measures}.
\newblock Wiley series in probability and statistics. Wiley, second edition,
  1999.

\bibitem{bourbaki_general_1966}
Nicolas Bourbaki.
\newblock {\em General topology. Part 2}.
\newblock Elements of mathematics. Hermann, 1966.

\bibitem{bourgain_periodic_1994}
Jean Bourgain.
\newblock Periodic nonlinear {Schrödinger} equation and invariant measures.
\newblock {\em Communications in Mathematical Physics}, 166(1):1--26, December
  1994.

\bibitem{bourgain_invariant_1996}
Jean Bourgain.
\newblock Invariant measures for the {2D}-defocusing nonlinear {Schrödinger}
  equation.
\newblock {\em Communications in Mathematical Physics}, 176:421--445, March
  1996.

\bibitem{bourgain_nonlinear_1999}
Jean Bourgain.
\newblock Nonlinear {Schrödinger} equations.
\newblock In {\em Hyperbolic equations and frequency interactions}, volume~5 of
  {\em {IAS}/{Park} {City} mathematics series}. American Mathematical Society,
  1999.

\bibitem{bourgain_invariant_2000}
Jean Bourgain.
\newblock Invariant measures for {NLS} in infinite volume.
\newblock {\em Communications in Mathematical Physics}, 210(3):605--620, April
  2000.

\bibitem{bourgain_invariant_2014}
Jean Bourgain and Aynur Bulut.
\newblock Invariant {Gibbs} measure evolution for the radial nonlinear wave
  equation on the 3d ball.
\newblock {\em Journal of Functional Analysis}, 266(4):2319--2340, February
  2014.

\bibitem{Bringmann_invariant_2022-1}
Bjoern Bringmann.
\newblock Invariant {Gibbs} measures for the three-dimensional wave equation
  with a {Hartree} nonlinearity {I}: measures.
\newblock {\em Stochastics and Partial Differential Equations: Analysis and
  Computations}, 10(1):1--89, March 2022.

\bibitem{Bringmann_invariant_2020}
Bjoern Bringmann.
\newblock Invariant {{Gibbs}} measures for the three-dimensional wave equation
  with a {{Hartree}} nonlinearity {{II}}: {{Dynamics}}.
\newblock {\em Journal of the European Mathematical Society}, 26(6):1933--2089,
  March 2023.

\bibitem{bringmann_invariant_2022}
Bjoern Bringmann, Yu~Deng, Andrea~R. Nahmod, and Haitian Yue.
\newblock Invariant {{Gibbs}} measures for the three dimensional cubic
  nonlinear wave equation.
\newblock {\em Inventiones mathematicae}, 236(3):1133--1411, June 2024.

\bibitem{bringmann_wave_2021}
Bjoern Bringmann, Jonas L{\"u}hrmann, and Gigliola Staffilani.
\newblock The wave maps equation and {{Brownian}} paths.
\newblock {\em Communications in Mathematical Physics}, 405(3):60, March 2024.

\bibitem{bringmann_invariant_2025}
Bjoern Bringmann and Gigliola Staffilani.
\newblock Invariant {{Gibbs}} measures for the one-dimensional quintic
  nonlinear {{Schr{\"o}dinger}} equation in infinite volume.
\newblock arXiv:2505.22478, May 2025.

\bibitem{brydges_statistical_1996}
David~C. Brydges and Gordon Slade.
\newblock Statistical mechanics of the 2-dimensional focusing nonlinear
  {Schrödinger} equation.
\newblock {\em Communications in Mathematical Physics}, 182(2):485--504,
  December 1996.

\bibitem{brzezniak_statistical_2022}
Zdzis{\l}aw Brze{\'z}niak and Jacek Jendrej.
\newblock Statistical mechanics of the wave maps equation in dimension 1 + 1.
\newblock {\em Journal of Functional Analysis}, 288(1):110688, January 2025.

\bibitem{burq_invariant_2007}
Nicolas Burq and Nikolay Tzvetkov.
\newblock Invariant measure for a three dimensional nonlinear wave equation.
\newblock {\em International Mathematics Research Notices}, page rnm108, 2007.

\bibitem{cacciafesta_invariance_2019}
Federico Cacciafesta and Anne-Sophie~de Suzzoni.
\newblock Invariance of {Gibbs} measures under the flows of {Hamiltonian}
  equations on the real line.
\newblock {\em Communications in Contemporary Mathematics}, 22:e1950012, March
  2020.

\bibitem{da_prato_strong_2003}
Giuseppe Da~Prato and Arnaud Debussche.
\newblock Strong solutions to the stochastic quantization equations.
\newblock {\em The Annals of Probability}, 31(4):1900--1916, October 2003.

\bibitem{deng_invariant_2019}
Yu~Deng, Andrea Nahmod, and Haitian Yue.
\newblock Invariant {{Gibbs}} measures and global strong solutions for
  nonlinear {{Schr{\"o}dinger}} equations in dimension two.
\newblock {\em Annals of Mathematics}, 200(2), September 2024.

\bibitem{ethier_markov_1986}
Stewart~N. Ethier and Thomas~G. Kurtz.
\newblock {\em Markov processes: characterization and convergence}.
\newblock Wiley series in probability and statistics. Wiley Interscience, 1986.

\bibitem{evans_partial_2010}
Lawrence~C. Evans.
\newblock {\em Partial differential equations}.
\newblock Number~19 in Graduate studies in mathematics. American Mathematical
  Society, second edition edition, 2010.

\bibitem{glimm_quantum_1987}
James Glimm and Arthur Jaffe.
\newblock {\em Quantum Physics: A Functional Integral Point of View}.
\newblock Springer-Verlag, New York, second edition, 1987.

\bibitem{gubinelli_pde_2021}
Massimiliano Gubinelli and Martina Hofmanová.
\newblock A {PDE} construction of the {Euclidean} ${\Phi^4_3}$ quantum field
  theory.
\newblock {\em Communications in Mathematical Physics}, 384(1):1--75, May 2021.

\bibitem{gubinelli_paracontrolled_2015}
Massimiliano Gubinelli, Peter Imkeller, and Nicolas Perkowski.
\newblock Paracontrolled distributions and singular {PDEs}.
\newblock {\em Forum of Mathematics. Pi}, 3:e6, August 2015.

\bibitem{gubinelli_renormalization_2018}
Massimiliano Gubinelli, Herbert Koch, and Tadahiro Oh.
\newblock Renormalization of the two-dimensional stochastic nonlinear wave
  equations.
\newblock {\em Transactions of the American Mathematical Society},
  370(10):7335--7359, October 2018.

\bibitem{gubinelli_global_2022}
Massimiliano Gubinelli, Herbert Koch, Tadahiro Oh, and Leonardo Tolomeo.
\newblock Global dynamics for the two-dimensional stochastic nonlinear wave
  equations.
\newblock {\em International Mathematics Research Notices},
  2022(21):16954--16999, October 2022.

\bibitem{gunaratnam_quasi-invariant_2022}
Trishen Gunaratnam, Tadahiro Oh, Nikolay Tzvetkov, and Hendrik Weber.
\newblock Quasi-invariant {Gaussian} measures for the nonlinear wave equation
  in three dimensions.
\newblock {\em Probability and Mathematical Physics}, 3(2):343--379, July 2022.

\bibitem{hairer_theory_2014}
Martin Hairer.
\newblock A theory of regularity structures.
\newblock {\em Inventiones mathematicae}, 198(2):269--504, March 2014.

\bibitem{kenig_focusing_2021}
Carlos Kenig and Dana Mendelson.
\newblock The focusing energy-critical nonlinear wave equation with random
  initial data.
\newblock {\em International Mathematics Research Notices},
  2021(19):14508--14615, October 2021.

\bibitem{kupiainen_renormalization_2016}
Antti Kupiainen.
\newblock Renormalization group and stochastic {PDEs}.
\newblock {\em Annales Henri Poincaré}, 17(3):497--535, March 2016.

\bibitem{lebowitz_statistical_1988}
Joel~L. Lebowitz, Harvey~A. Rose, and Eugene~R. Speer.
\newblock Statistical mechanics of the nonlinear {Schrödinger} equation.
\newblock {\em Journal of Statistical Physics}, 50(3):657--687, February 1988.

\bibitem{mckean_statistical_1994}
H.~P. McKean and K.~L. Vaninsky.
\newblock Statistical mechanics of nonlinear wave equations.
\newblock In Lawrence Sirovich, editor, {\em Trends and perspectives in applied
  mathematics}, volume 100 of {\em Applied Mathematical Sciences}, pages
  239--264. Springer-Verlag, New York, 1994.

\bibitem{mourrat_dynamic_2017}
Jean-Christophe Mourrat and Hendrik Weber.
\newblock The dynamic ${\Phi^4_3}$ model comes down from infinity.
\newblock {\em Communications in Mathematical Physics}, 356(3):673--753,
  December 2017.

\bibitem{mourrat_global_2017}
Jean-Christophe Mourrat and Hendrik Weber.
\newblock Global well-posedness of the dynamic $\phi^4$ model in the plane.
\newblock {\em The Annals of Probability}, 45(4):2398--2476, July 2017.

\bibitem{oh_stochastic_2021}
Tadahiro Oh, Mamoru Okamoto, and Leonardo Tolomeo.
\newblock {\em Stochastic quantization of the ${\Phi_3^3}$-model}.
\newblock Number~16 in {Memoirs of the European Mathematical Society}. EMS
  Press, 1 edition, February 2025.

\bibitem{oh_uniqueness_2024}
Tadahiro Oh, Mamoru Okamoto, and Nikolay Tzvetkov.
\newblock Uniqueness and non-uniqueness of the {{Gaussian}} free field
  evolution under the two-dimensional {{Wick}} ordered cubic wave equation.
\newblock {\em Annales de l'Institut Henri Poincar\'e, Probabilit\'es et
  Statistiques}, 60(3), August 2024.

\bibitem{oh_invariant_2021}
Tadahiro Oh, Tristan Robert, Philippe Sosoe, and Yuzhao Wang.
\newblock Invariant {Gibbs} dynamics for the dynamical sine-{Gordon} model.
\newblock {\em Proceedings of the Royal Society of Edinburgh Section A:
  Mathematics}, 151(5):1450--1466, October 2021.

\bibitem{oh_parabolic_2021}
Tadahiro Oh, Tristan Robert, and Yuzhao Wang.
\newblock On the parabolic and hyperbolic {Liouville} equations.
\newblock {\em Communications in Mathematical Physics}, 387(3):1281--1351,
  November 2021.

\bibitem{oh_quasi-invariant_2021}
Tadahiro Oh and Kihoon Seong.
\newblock Quasi-invariant {Gaussian} measures for the cubic fourth order
  nonlinear {Schrödinger} equation in negative {Sobolev} spaces.
\newblock {\em Journal of Functional Analysis}, 281(9):e109150, November 2021.

\bibitem{oh_remark_2020}
Tadahiro Oh, Kihoon Seong, and Leonardo Tolomeo.
\newblock A remark on {{Gibbs}} measures with log-correlated {{Gaussian}}
  fields.
\newblock {\em Forum of Mathematics, Sigma}, 12:e50, April 2024.

\bibitem{oh_optimal_2022}
Tadahiro Oh, Philippe Sosoe, and Leonardo Tolomeo.
\newblock Optimal integrability threshold for {Gibbs} measures associated with
  focusing {NLS} on the torus.
\newblock {\em Inventiones mathematicae}, 227(3):1323--1429, March 2022.

\bibitem{oh_optimal_2018}
Tadahiro Oh, Philippe Sosoe, and Nikolay Tzvetkov.
\newblock An optimal regularity result on the quasi-invariant {Gaussian}
  measures for the cubic fourth order nonlinear {Schrödinger} equation.
\newblock {\em Journal de l'École polytechnique — Mathématiques},
  5:793--841, October 2018.

\bibitem{oh_pedestrian_2018}
Tadahiro Oh and Laurent Thomann.
\newblock A pedestrian approach to the invariant {Gibbs} measures for the 2-d
  defocusing nonlinear {Schrödinger} equations.
\newblock {\em Stochastics and Partial Differential Equations: Analysis and
  Computations}, 6(3):397--445, September 2018.

\bibitem{oh_invariant_2020}
Tadahiro Oh and Laurent Thomann.
\newblock Invariant {Gibbs} measures for the 2-d defocusing nonlinear wave
  equations.
\newblock {\em Annales de la Faculté des sciences de Toulouse :
  Mathématiques}, 29(1):1--26, July 2020.

\bibitem{oh_hyperbolic_2022}
Tadahiro Oh, Leonardo Tolomeo, Yuzhao Wang, and Guangqu Zheng.
\newblock Hyperbolic ${P(\Phi)_2}$-model on the plane.
\newblock arXiv:2211.03735v2, November 2022.

\bibitem{oh_quasi-invariant_2017}
Tadahiro Oh and Nikolay Tzvetkov.
\newblock Quasi-invariant {Gaussian} measures for the cubic fourth order
  nonlinear {Schrödinger} equation.
\newblock {\em Probability Theory and Related Fields}, 169(3):1121--1168,
  December 2017.

\bibitem{oh_quasi-invariant_2020}
Tadahiro Oh and Nikolay Tzvetkov.
\newblock Quasi-invariant {Gaussian} measures for the two-dimensional
  defocusing cubic nonlinear wave equation.
\newblock {\em Journal of the European Mathematical Society}, 22(6):1785--1826,
  February 2020.

\bibitem{parisi_perturbation_1981}
Giorgio Parisi and Yong~Shi Wu.
\newblock Perturbation theory without gauge fixing.
\newblock {\em Scientia Sinica. Zhongguo Kexue}, 24(4):483--496, April 1981.

\bibitem{seong_invariant_2022}
Kihoon Seong.
\newblock Invariant {Gibbs} dynamics for the two-dimensional
  {Zakharov}--{Yukawa} system.
\newblock {\em Journal of Functional Analysis}, 286, February 2024.

\bibitem{simon_pphi_2_1974}
Barry Simon.
\newblock {\em The $P({\Phi})_2$ {Euclidean} (quantum) field theory}.
\newblock Princeton series in physics. Princeton University Press, 1974.

\bibitem{tolomeo_unique_2020}
Leonardo Tolomeo.
\newblock Unique ergodicity for a class of stochastic hyperbolic equations with
  additive space-time white noise.
\newblock {\em Communications in Mathematical Physics}, 377(2):1311--1347, July
  2020.

\bibitem{tolomeo_global_2021}
Leonardo Tolomeo.
\newblock Global well-posedness of the two-dimensional stochastic nonlinear
  wave equation on an unbounded domain.
\newblock {\em The Annals of Probability}, 49(3):1402--1426, May 2021.

\bibitem{triebel_theory_1992}
Hans Triebel.
\newblock {\em Theory of function spaces {II}}.
\newblock Number~84 in Monographs in mathematics. Birkhäuser Verlag, 1992.

\bibitem{tsatsoulis_spectral_2018}
Pavlos Tsatsoulis and Hendrik Weber.
\newblock Spectral gap for the stochastic quantization equation on the
  2-dimensional torus.
\newblock {\em Annales de l'Institut Henri Poincaré, Probabilités et
  Statistiques}, 54(3):1204--1249, August 2018.

\bibitem{tzvetkov_quasiinvariant_2015}
Nikolay Tzvetkov.
\newblock Quasiinvariant {Gaussian} measures for one-dimensional {Hamiltonian}
  partial differential equations.
\newblock {\em Forum of Mathematics, Sigma}, 3:e28, 2015.

\bibitem{xu_invariant_2014}
Samantha Xu.
\newblock Invariant {Gibbs} measure for {3D} {NLW} in infinite volume.
\newblock arXiv:1405.3856v1, May 2014.

\bibitem{zhidkov_invariant_1994}
Peter~E. Zhidkov.
\newblock An invariant measure for a nonlinear wave equation.
\newblock {\em Nonlinear Analysis: Theory, Methods \& Applications},
  22(3):319--325, February 1994.

\end{thebibliography}

\end{document}